\documentclass[12pt,oneside]{amsart}
\usepackage{amssymb,amscd,amsxtra,calc}
\usepackage{amsmath}
\usepackage{mathrsfs}
\usepackage{xypic}
\usepackage[graph]{xy}
\usepackage{supertabular}
\usepackage{longtable}
\usepackage{graphicx}
\usepackage{cases}
\usepackage{tikz}
\usepackage{longtable}
\usepackage{colortbl}
\usepackage{comment}

\author{Takayuki Miura} 
\title{Classification of del Pezzo surfaces with $\frac{1}{3}$(1,1)- and $\frac{1}{4}$(1,1)-singularities}
\date{\today}
\subjclass[2010]{}
\keywords{del Pezzo surface, minimal model program}
\address{}
\email{miura29@ms.u-tokyo.ac.jp {\rm or} t.29.mess@gmail.com}

\newcommand{\pr}{\mathbb{P}}
\newcommand{\PP}{\mathbb{P}^1 \times \mathbb{P}^1}
\newcommand{\N}{\mathbb{N}}
\newcommand{\Z}{\mathbb{Z}}
\newcommand{\Q}{\mathbb{Q}}

\newcommand{\C}{\mathbb{C}}
\newcommand{\F}{\mathbb{F}}

\newcommand{\Exc}{\operatorname{Exc}}
\newcommand{\Sing}{\operatorname{Sing}}

\newcommand{\Pic}{\operatorname{Pic}}

\newcommand{\sC}{\mathcal{C}}
\newcommand{\sO}{\mathcal{O}}

\newcommand{\sS}{\mathcal{S}}

\newcommand{\sB}{\mathcal{B}}

\newcommand{\tA}{\mathscr{A}}
\newcommand{\tB}{\mathscr{B}}

\newcommand{\fmor}{{\rm I}}
\newcommand{\smor}{{\rm I\hspace{-.1em}I}}
\newcommand{\tmor}{{\rm I\hspace{-.1em}I\hspace{-.1em}I}}

\newtheorem{thm}{Theorem}[section]
\newtheorem{lemma}[thm]{Lemma}
\newtheorem{proposition}[thm]{Proposition}
\newtheorem{corollary}[thm]{Corollary}
\newtheorem{claim}[thm]{Claim}

\theoremstyle{definition}
\newtheorem{definition}[thm]{Definition}
\newtheorem{proposition-definition}[thm]{Proposition-Definition}
\newtheorem{remark}[thm]{Remark}
\newtheorem{notation}[thm]{Notation}

\newtheorem*{ack}{Acknowledgments}

\begin{document}

\maketitle

\begin{abstract}

We classify all the del Pezzo surfaces with $\frac{1}{3}(1,1)$- and $\frac{1}{4}(1,1)$-singularities having no floating $(-1)$-curves into 39 types. 

\end{abstract}

\setcounter{tocdepth}{2}
\tableofcontents

\section{Introduction}\label{intro_sec}

\subsection{Main theorem}

Throughout this paper, we work over the complex number field $\mathbb{C}$. 
A {\it{del Pezzo surface}} is a normal projective surface whose anti-canonical divisor is an ample $\Q$-Cartier divisor.

Study of del Pezzo surfaces is one of the principal topics in the theory of algebraic surfaces.
They have fascinated many people since the 19th century (cf. \cite{Do05}).
In particular, del Pezzo surfaces with quotient singularities play important roles in klt minimal model program and many people are concerned in the classification of them nowadays.

In this paper, we obtain the complete classification of del Pezzo surfaces with at most $\frac{1}{3}(1,1)$- and $\frac{1}{4}(1,1)$-singularities having no floating $(-1)$-curves, where a {\it $\frac{1}{n}(a,b)$-singularity} is a surface cyclic quotient singularity $\mathbb{C}^2 / \mu_n$ where $\mu_n$ acts linearly on $\mathbb{C}^2$ with weights $a,b \in (\frac{1}{n} \mathbb{Z})/\mathbb{Z}$, and
a {\it floating $(-1)$-curve} is a $(-1)$-curve contained in the smooth locus of the surface.
More precisely, our main result is the following theorem.

\begin{thm}\label{main_thm}

Let $X$ be a del Pezzo surface with at most $\frac{1}{3}(1,1)$- and $\frac{1}{4}(1,1)$-singularities having no floating $(-1)$-curves.
Then $X$ is one of the surfaces in Table \ref{main_table}.
Moreover, all surfaces in Table \ref{main_table} really exist.

{\small \rm
\begin{table}[htb]
\caption{Del Pezzo surfaces with at most $\frac{1}{3}(1,1)$- and $\frac{1}{4}(1,1)$-singularities having no floating $(-1)$-curves}\label{main_table} 
\renewcommand{\arraystretch}{1.4}
  \begin{tabular}{c|c|c|c|c|c|c} 
    \multicolumn{1}{l|}{No.} & \multicolumn{1}{c|}{$X_{min}$} & directed seq. & \multicolumn{1}{c|}{($n_3, n_4$)}
      &  \multicolumn{1}{c|}{$(-K_{X})^2$} & $\rho (X)$ & $h^0(-K_X)$  \\ \hline \hline
     1 &  $M_{13}$ & $\tmor_1 \circ \tmor_1 \circ \tmor_1 \circ \tmor_1$ & (4,4) & $\frac{4}{3}$ & 6 & 1   \\ \rowcolor[gray]{0.9}
     2 &  $\pr(1,1,4)$ & $\smor_8 \circ \smor_8$  & (4,3) & $\frac{1}{3}$ & 7  & 0 \\  
     3 &  $\pr(1,1,3)$ & $\smor_7 \circ \smor_4$ & (4,3) & $\frac{4}{3}$  & 6 & 1  \\  \rowcolor[gray]{0.9}
     4 &  $\pr(1,1,3)$ & $\smor_4 \circ \smor_4$ & (5,2) & $\frac{5}{3}$  & 5 & 1   \\  
     5 &  $\pr(1,1,3)$ & $\smor_7 \circ \smor_5$ & (3,3) & 1  & 7 & 1   \\  \rowcolor[gray]{0.9}
     6 &  $\pr(1,1,3)$ & $\smor_7 \circ \smor_3$& (4,2) & $\frac{4}{3}$  & 6  & 1  \\
     7 &  $\PP$ & $\smor_4 \circ \smor_4$ & (4,2) & $\frac{4}{3}$ & 6 & 1  \\  \rowcolor[gray]{0.9}
     8 &  $\pr(1,1,3)$ & $\smor_4 \circ \smor_3$ & (5,1) & $\frac{5}{3}$  & 5 & 1   \\  
     9 &  $M_{8}$ & $\tmor_5 \circ \tmor_5 \circ \tmor_5$ & (6,0) & 2 & 4 & 1   \\ \rowcolor[gray]{0.9}
     10 & $\pr(1,1,3)$ & $\smor_7 \circ \smor_2$ & (2,3) & $\frac{5}{3}$  & 7 & 2   \\ 
     11 &  $\pr(1,1,4)$ & $\smor_8 \circ \smor_1$ & (3,2) & 2 & 6 & 2  \\  \rowcolor[gray]{0.9}
     12 &  $\pr(1,1,3)$ & $\smor_7 \circ \smor_1$ & (3,2) & 2  & 6 & 2 \\     
     13 &  $\pr(1,1,3)$ & $\smor_4 \circ \smor_1$ & (4,1) & $\frac{7}{3}$ & 5 & 2  \\    \rowcolor[gray]{0.9}
  \end{tabular}
\end{table}
}

\clearpage

{\small \rm
\begin{table}[htb]
\renewcommand{\arraystretch}{1.4} 
  \begin{tabular}{c|c|c|c|c|c|c} 
\rowcolor[gray]{1.0}    \multicolumn{1}{l|}{No.} &   \multicolumn{1}{c|}{$X_{min}$} &
      directed seq.  & \multicolumn{1}{c|}{($n_3, n_4$)} & \multicolumn{1}{c|}{$(-K_{X})^2$} & $\rho (X)$ & $h^0(-K_X)$  \\ \hline \hline \rowcolor[gray]{0.9}
     14 &  $\pr(1,1,3)$ & $\smor_3 \circ \smor_3$ & (5,0) & $\frac{5}{3}$  & 5 & 1   \\   
     15 &  $\pr(1,1,4)$ & $\smor_8$ & (2,2) & $\frac{14}{3}$ & 4 & 5 \\  \rowcolor[gray]{0.9}
     16 &  $\pr(1,1,3)$ & $\smor_7$ & (2,2) & $\frac{14}{3}$  & 4 & 5   \\
     17 &  $\PP$ & $\smor_4 \circ \smor_2$ & (2,2) & $\frac{5}{3}$ & 7 & 2 \\   \rowcolor[gray]{0.9}
     18 &  $\pr(1,1,3)$ & $\smor_4$ & (3,1) & 5  & 3 & 5  \\ 
     19 &  $\PP$ & $\smor_4 \circ \smor_1$ & (3,1) & 2 & 6 & 2  \\ \rowcolor[gray]{0.9}
     20 &  $\pr(1,1,4)$ & $\smor_1 \circ \smor_1 \circ \smor_1$ & (3,1) & 1 & 7 & 1  \\    
     21 &  $\pr(1,1,3)$ & $\smor_3 \circ \smor_1$ & (4,0) & $\frac{7}{3}$  & 5 & 2  \\ \rowcolor[gray]{0.9}
     22 &  $\pr(1,1,4)$ & $\smor_6$ & (1,2) & $\frac{16}{3}$ & 4 & 6  \\ 
     23 &  $\pr(1,1,4)$ & $\smor_3$ & (2,1) & $\frac{17}{3}$ & 3 & 6  \\ \rowcolor[gray]{0.9}
     24 &  $\pr(1,1,3)$ & $\smor_5$ & (2,1) & $\frac{14}{3}$  & 4 & 5  \\  
     25 &  $\PP$ & $\smor_4$ & (2,1) & $\frac{14}{3}$ & 4 & 5  \\  \rowcolor[gray]{0.9}
     26 &  $\pr(1,1,4)$ & $\smor_1 \circ \smor_1$ & (2,1) & $\frac{11}{3}$ & 5 & 4  \\   
     27 &  $\pr(1,1,3)$ & $\smor_3$ & (3,0) & 5  & 3 & 5   \\ \rowcolor[gray]{0.9}
     28 &  $\PP$ & $\smor_2 \circ \smor_2$ & (0,2) & 2 & 8 & 3 \\    
     29 &  $\pr(1,1,4)$ & $\smor_1$ & (1,1) & $\frac{19}{3}$ & 3 & 7  \\ \rowcolor[gray]{0.9}
     30 &  $\pr(1,1,3)$ & $\smor_2$ & (1,1) & $\frac{16}{3}$  & 4 & 6  \\  
     31 &  $\PP$ & $\smor_2 \circ \smor_1$ & (1,1) & $\frac{7}{3}$ & 7  & 3  \\    \rowcolor[gray]{0.9}
     32 &  $\pr(1,1,3)$ & $\smor_1$ & (2,0) & $\frac{17}{3}$  & 3   & 6\\ 
     33 &  $\PP$ & $\smor_1 \circ \smor_1$ & (2,0) & $\frac{8}{3}$ & 6 & 3 \\  \rowcolor[gray]{0.9}
     34 &  $\pr(1,1,4)$ & - & (0,1) & 9 & 1 & 10 \\
     35 &  $\PP$ & $\smor_2$ & (0,1) & 5 & 5 & 6  \\   \rowcolor[gray]{0.9}
     36 &  $\pr(1,1,3)$ & - & (1,0) & $\frac{25}{3}$  & 1 & 9  \\ 
     37 &  $\PP$ & $\smor_1$ & (1,0) & $\frac{16}{3}$ & 4 & 6  \\  \rowcolor[gray]{0.9}
     38 &  $\pr^2$ & - & (0,0) & 9 & 1 & 10 \\ 
     39 &  $\PP$ & - & (0,0) & 8 & 2 & 9  \\ 
  \end{tabular}
\end{table}
}

\end{thm}

\begin{notation} The notation in Table \ref{main_table} is the following:

\noindent {$\bullet$ $X_{min}$} : a minimal surface obtained from $X$ by extremal contractions.
A del Pezzo surface is {\it minimal} if it has no birational extremal contractions.
Such minimal surfaces are listed in Tables \ref{min_rank1_table} and \ref{min_rank2_table}.

\noindent {$\bullet$ $M_8$} : a minimal surface isomorphic to $xyz - w^3 = 0$ in $ \pr^3$, which has three singular points of type $A_2$.

\noindent {$\bullet$ $M_{13}$} : a minimal surface having four singular points of type $\frac{1}{5}(1,2)$.

\noindent $\bullet$ directed seq. : a {\it minimal directed sequence} defined in Theorem \ref{directed}.

\noindent {$\bullet$ $\smor_i$ and $\tmor_j$} : types of compositions of extremal contractions listed in Tables \ref{2nd_mor} and \ref{3rd_mor}.
A minimal directed sequence is expressed as these compositions.
For example, if $X$ is of No.1, then $X$ has a minimal directed sequence $X \overset{\tmor_1}{\to} X_1 \overset{\tmor_1}{\to} X_2 \overset{\tmor_1}{\to} X_3 \overset{\tmor_1}{\to} M_{13}$.
  
\noindent {$\bullet$ $n_3$} : the number of singular points of type $\frac{1}{3}(1,1)$ on $X$.

\noindent {$\bullet$ $n_4$} : the number of singular points of type $\frac{1}{4}(1,1)$ on $X$.

\noindent {$\bullet$ $(-K_X)^2$} : the anti-canonical volume of $X$.

\noindent {$\bullet$ $\rho (X)$} : the Picard number of $X$.

\noindent {$\bullet$ $h^0(-K_X) := \dim_{\C} H^0( X, \sO_X (-K_X))$} 

\end{notation}

By Theorem \ref{main_thm}, we obtain the optimal bound of the numbers of singular points on a del Pezzo surface $X$ with $\frac{1}{3}(1,1)$- and $\frac{1}{4}(1,1)$-singularities.
Indeed, we have a sequence of contractions of floating $(-1)$-curves $X \to \cdots \to S$, where $S$ is listed in Table \ref{main_table}. 
Since the numbers of singular points of each type on $X$ and $S$ are equal, we obtain the following corollary.

\begin{corollary}\label{num_of_sing}

The possibilities of $(n_3, n_4)$ are plotted by points $\bullet$ in the following figure.
Moreover, for each $(a, b)$ where a point $\bullet$ is plotted, there are some del Pezzo surfaces whose $(n_3, n_4) = (a, b)$.

\begin{figure}[htbp]
\centering\includegraphics[width=8cm, bb=0 0 700 700]{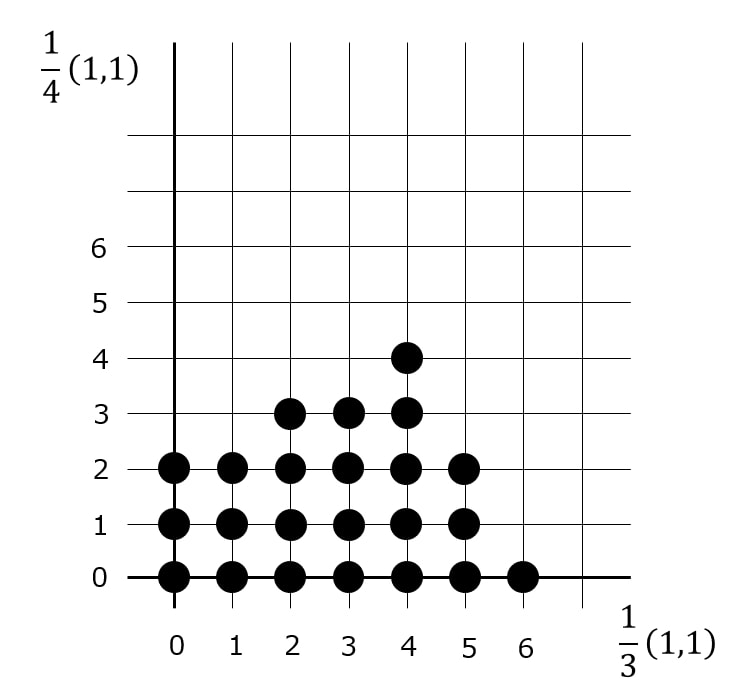}
\end{figure}

\end{corollary}

\subsection{Known results and this work} 

As mentioned above, del Pezzo surfaces with quotient singularities are important in klt minimal model program.
There are many results about classifications of such del Pezzo surfaces.
Below we quote some of them, which are strongly related to this work.

\subsubsection{Gorenstein index}

The {\it Gorenstein index}, which is the smallest positive integer $m$ such that $mK$ is a Cartier divisor, is an important invariant of del Pezzo surfaces.
Del Pezzo surfaces with small indices have been studied by many people.
Those of index one are called Gorenstein del Pezzo surfaces.
They are classified, for example, by F. Hidaka and K. Watanabe (\cite{HW81}).
In the case of index two, V. Alexeev and V. Nikulin classify them over the complex number field using K3 surface theory.
Later, Nakayama gives the complete classification of them in any characteristic (\cite{Na07}).
Those of index three are classified by K. Fujita and K. Yasutake (\cite{FY17}).
There are no complete classifications of del Pezzo surfaces in the case where the index is more than three.
We note that the index of del Pezzo surfaces classified in this paper is six if $n_3$ and $n_4$ are positive.

\subsubsection{Types of singularities}

Restricting types of singularities is an effective perspective to classify del Pezzo surfaces.
From this perspective, A. Corti and L. Heuberger classify those with $\frac{1}{3}(1,1)$-singularities (\cite{CH17}), which inspires this work.
Their work is part of a program to study mirror symmetry for del Pezzo surfaces with cyclic quotient singularities.
Their classification overlaps the one of K. Fujita and K. Yasutake (\cite{FY17}) since the index of those with $\frac{1}{3}(1,1)$-singularities is three, but they classify those by a {\it cascade} which is a relation of birational morphisms between surfaces.
This terminology is introduced in \cite{RS03}.
We are inspired by this way to classify del Pezzo surfaces.
We will explain our strategy for classification in Subsection \ref{strategy}.
We also mention that del Pezzo surfaces with only one $\frac{1}{k}(1,1)$-singularity are classified (\cite{CP17}).

\subsubsection{Picard number}
As for the Picard number, many authors are interested in del Pezzo surfaces with Picard number one, which are called {\it rank one} del Pezzo surfaces.
There are many preceding studies of rank one del Pezzo surfaces. 
For example, the optimal upper bound of the numbers of singular points and the orbifold Euler numbers (cf. Definition \ref{orb_Euler}) of them are known (\cite{Be09}, \cite{Hw14}).
Rank one del Pezzo surfaces with a unique singular point are classified (\cite{Ko99}).
In this paper, we also classify some rank one del Pezzo surfaces and use them for the classification (cf. Section \ref{minimal_sec}).

\subsection{Preliminary}
We introduce definitions of basic concepts and notation we use throughout this paper.

\begin{definition}\label{neg-curve}

Let $X$ be a normal projective surface.
A smooth rational curve $C$ whose self intersection number is $-n$ is called a $(-n)$-$curve$.
Let $\pi : Y \to X$ be the minimal resolution.
An irreducible curve $C$ on $X$ is called a {\it{quasi-$(-n)$-curve}} if its strict transform $C_Y$ on $Y$ is a $(-n)$-curve.
In particular, a quasi-$(-1)$-curve is called a {\it{quasi-line}}.
A curve $C \subset X$ is called a {\it floating $(-1)$-curve} if $C$ is a $(-1)$-curve and contained in the smooth locus $X_{\rm sm}$.
If $n \ge 1$, a $(-n)$-curve is called a {\it negative curve}.

\end{definition}

\begin{definition}

A normal projective surface is called {\it{of type $\tA$}} if it has at most $\frac{1}{3}$(1,1)- or $\frac{1}{4}$(1,1)-singularities.
A normal projective surface is called {\it{of type $\tB$}} if it has at most  $A_1$-, $A_2$-, $A_3$-, $\frac{1}{3}$(1,1)-, $\frac{1}{4}$(1,1)- and $\frac{1}{5}$(1,2)-singularities.
Here an $A_n$-singularity is a $\frac{1}{n+1}(1,n)$-singularity.

For a normal projective surface $X$ of type $\tB$, denote the singular points on $X$ by $P_1, \ldots , P_q$.
Then we set
\[
\sS(X) := \{ *_1, \ldots, *_q \} ,
\]
where $P_i$ is $*_i$-singularity for each $1 \le i \le q$.
For example, if the singular locus of a projective surface $X$ consists of one singular point of type $\frac{1}{4}(1,1)$ and two singular points of type $A_3$, then we write $\sS(X) = \{ \frac{1}{4}(1,1), A_3, A_3 \}$. 

\end{definition}

Note that a normal projective surface of type $\tA$ is also of type $\tB$.
The purpose of this paper is to classify del Pezzo surfaces of type $\tA$ with no floating $(-1)$-curves.
In the course of classification, del Pezzo surfaces of type $\tB$ play an important role (Section \ref{minimal_sec}).

\begin{definition}\label{dual_graph}

In a dual graph of curves on a projective surface, we denote a $(-1)$-curve by $\bullet$, a $(-2)$-curve by $\triangle$, a $(-3)$-curve by $\square$ and a $(-4)$-curve by $\bigcirc$.

\end{definition}

\begin{remark}

Let $P$ be an $A_1$-, $A_2$-, $A_3$-, $\frac{1}{3}(1,1)$-, $\frac{1}{4}(1,1)$- or $\frac{1}{5}(1,2)$-singularity.
The dual graphs of the exceptional curves of the minimal resolution of $P$ are the following:

{ \small
\begin{table}[htb]
\renewcommand{\arraystretch}{1.7}
  \begin{tabular}{ccc|ccc|ccc} 
$A_1$ & : & \xygraph{
    \triangle ([]!{+(0,-.4)} {-2}) 
} &
$A_2$ & : & \xygraph{
    \triangle ([]!{+(0,-.4)} {-2}) - [r]
    \triangle ([]!{+(0,-.4)} {-2}) 
} &
$A_3$ & : & \xygraph{
    \triangle ([]!{+(0,-.4)} {-2}) - [r]
    \triangle ([]!{+(0,-.4)} {-2}) - [r] 
    \triangle ([]!{+(0,-.4)} {-2}) 
} \\ \hline
$\frac{1}{3}(1,1)$ & : & 
\xygraph{
    \square ([]!{+(0,-.4)} {-3})
}  &
$\frac{1}{4}(1,1)$ & : & 
\xygraph{
    \bigcirc ([]!{+(0,-.4)} {-4})
}  &
$\frac{1}{5}(1,2)$ & : & 
\xygraph{
    \triangle ([]!{+(0,-.4)} {-2}) - [r]
    \square ([]!{+(0,-.4)} {-3}) 
}  \\
  \end{tabular} \\
\end{table}
}

\end{remark}

\begin{notation}

Let $f: Y \to X$ be a birational morphism and $C$ a curve on $X$.
Then $C_Y$ denotes the strict transform of $C$ by $f$.
For a divisor $D := \sum a_i C_i$, we set $D_Y := \sum a_i (C_i)_Y$.

\end{notation}

\begin{notation}

We denote by $\mathbb{F}_n$ the Hirzebruch surface of degree $n$.
We also denote the minimal section by $\sigma$ and a fiber by $l$.
$\sigma_{\infty}$ denotes an irreducible curve linearly equivalent to $\sigma + nl$, which is called a section at infinity.
\end{notation}

\subsection{The strategy for the proof of the main theorem}\label{strategy}

In the study of \cite{CH17}, Corti and Heuberger use the Riemann-Roch theorem and lattice theory to obtain an effective bound of the number of singular points on a del Pezzo surface.
In this paper, however, their method does not work.
Thus we use a different method.

Roughly speaking, the proof of Theorem \ref{main_thm} is divided into the following 5 steps.
Let $X$ be a del Pezzo surface with at most $\frac{1}{3}(1,1)$- and $\frac{1}{4}(1,1)$-singularities.

\noindent{\bf Step 1 : Construction of a minimal directed sequence}

One of the main idea of this paper is to introduce a sequence of (compositions of) extremal contractions
\[
X  \overset{\alpha_1}{\rightarrow} \cdots \overset{\alpha_l}{\rightarrow} 
S \overset{\beta_1}{\rightarrow} \cdots \overset{\beta_m}{\rightarrow}
T_{min} \overset{\gamma_1}{\rightarrow} \cdots \overset{\gamma_n}{\rightarrow} X_{min},
\]
which is constructed in Theorem \ref{directed}.
This is called a {\it minimal directed sequence}.
Here we call each $\alpha_i$ a first morphism, each $\beta_j$ a second morphism and each $\gamma_k$ a third morphism (Definition \ref{fst}).
A first morphism is nothing but a contraction of a floating $(-1)$-curve.
Second morphisms and third morphisms are compositions of extremal contractions and classified into 8 types and 9 types respectively as in Tables \ref{2nd_mor} and \ref{3rd_mor}.
$X_{min}$ is a minimal surface, which is classified in the next step.
$T_{min}$ is a $\smor$-minimal surface (cf. Definition \ref{2_min}), which is also introduced in this paper.
In Step 1, we prove the existence of this sequence and classify second morphisms and third morphisms.

\noindent{\bf Step 2 : Classification of minimal surfaces}

In Section \ref{contractions_sec}, we prove that a minimal surface $X_{min}$ is of type $\tB$.
We classify minimal surfaces of type $\tB$ into 19 cases in Section \ref{minimal_sec}.
In addition to standard methods to classify rank one del Pezzo surfaces, we use two ray games, which are often used in the classification of Fano 3-folds with Picard number one.

\noindent{\bf Step 3 : Determination of candidates of $X$}

From the results in Steps 1 and 2, we can list all the possibilities of minimal directed sequences in Section \ref{candidate_sec}.
The list, however, is huge.
Hence  we need to restrict these possibilities.
In this step, we first restrict the possibilities of a $\smor$-minimal del Pezzo surface $T_{min}$ into six cases.
Let $T_1$ be a del Pezzo surface such that there exists a second morphism $T_1 \to T_{min}$.
Next, by using the candidates of $T_{min}$, we restrict the possibilities of $T_1$ into 13 cases.
Similarly, by using the candidates of $T_1$, we restrict the possibilities  of a del Pezzo surface $T_2$ such that a second morphism $T_2 \to T_1$ exists.
They are restricted into 19 cases.
The possibility of a del Pezzo surface $T_3$, which has a second morphism $T_3 \to T_2$, is also restricted into one case.
Moreover, by Corollary \ref{T_>4}, for $m \ge 4$, we also see that there is no examples of $T_m$ such that $T_m \to T_{m-1} \to \cdots \to T_1 \to T_{min}$ exists.
Thus we restrict the possibilities of  $T_{min}$, $T_1, T_2$ and $T_3$ into 39 cases.
Then we see that they are nothing but the candidates of the surface $X$ which we are going to classify.

\noindent{\bf Step 4 : Construction of examples for each candidate of $X$}

In Section \ref{construct_sec}, we check the existence of each candidate of $X$.
Let $Y \to X$ be the minimal resolution.
In Step 3, we also see that how to construct $Y$ from a Hirzebruch surface by explicit blow-ups.
Then starting from the surface $Y$, we obtain the surface $X$ by contracting several negative curves and show that $-K_X$ is ample.

\noindent{\bf Step 5 : Distinction of surfaces with the same invariants}

In Table \ref{main_table}, there are four pairs of del Pezzo surfaces $X_1, X_2$ with the same number of singular points of each type, the same anti-canonical volume and the same Picard number. 
We distinguish such $X_1$ and $X_2$ by observing the configurations of negative curves on the minimal resolutions of them.  \\

We expect that this method can be applied to the other cases, for example, del Pezzo surfaces with $\frac{1}{k}(1,1)$-singularities where $k \ge 5$.

\begin{ack}
The author would like to express great gratitude to his supervisor Professor Hiromichi Takagi for his encouragement and valuable advice.
The author is also grateful to Takeru Fukuoka, who found an important fact which advanced this study in our seminar.
The author also would like to thank Professors Yoshinori Gongyo and Masanori Kobayashi for their helpful comments and suggestions.

\end{ack}

\section{Contractions between surfaces of type $\tB$}\label{contractions_sec}

In this section, we classify $K_V$-negative extremal contractions appearing in a minimal model program which starts from del Pezzo surfaces of type $\tA$. 

\subsection{Basic properties of extremal contractions}

We first introduce some basic properties of extremal contractions for minimal model programs.
In this paper, an extremal contraction means the contraction of a $K$-negative extremal ray. 

\begin{definition}

Let $V$ be a normal projective surface.
A quasi-line $C$ on $V$ passing through at least two singular points of type $\frac{1}{4}(1,1)$ is called a {\it $T$-line}.

\end{definition}
The following lemma is suggested by T. Fukuoka, which will play an important role throughout the paper.

\begin{lemma}\label{Fuku}

There is no $K_V$-negative extremal contraction $f : V \to V_1$ contracting a $T$-line.
In particular, there is no $T$-line on a del Pezzo surface.

\end{lemma}

\begin{proof}

Let $C \subset V$ be an $f$-exceptional curve. Assume that $C$ is a $T$-line by contradiction. 
Let $\pi : Y \to V$ be the minimal resolution.
We denote by $E_1, E_2$ irreducible components of the exceptional curves over the singular points of type $\frac{1}{4}(1,1)$ and by $E_i$ ($i \ge 3$) the exceptional curves over the other singular points on $V$. 
Then it holds
\[
\pi^{*} C = C_Y + \frac{1}{4} E_1 + \frac{1}{4} E_2 + \sum_{i \ge 3} a_i E_i,
\]
where $a_i \ge 0$ for $i \ge 3$. Hence we obtain
\[
-K_Y \cdot \pi^{*} C = -K_Y \cdot C_Y + (-K_Y) \cdot (\frac{1}{4} E_1 + \frac{1}{4} E_2 + \sum a_i E_i ) . 
\]
Thus we have 
\[
-K_V \cdot C  =  1 - \frac{1}{2} - \frac{1}{2} -K_Y \cdot  \sum a_i E_i   \ \le \   0 .
\]
This contradicts the fact that $f : V \to V_1$ is $K_V$-negative.

\end{proof}

\begin{lemma}\label{keep_ldP}

Let $V$ be a del Pezzo surface with at most quotient singularities and $f : V \to V_1$ a birational extremal contraction.
Then $V_1$ is also a del Pezzo surface with at most quotient singularities.

\end{lemma}

\begin{proof}

Note that a del Pezzo surface $V$ has at most quotient singularities if and only if $(V,0)$ is a klt pair.
Since $(V, 0)$ is a klt pair, $(V_1, 0)$ is also a klt pair. 
Denote the exceptional curve by $E$.
We may write
\[
K_{V} = f^{*} K_{V_1} + aE,
\]
where $a>0$.
Therefore, it is enough to show that $-K_{V_1}$ is ample.
Since $V$ is a del Pezzo surface, we see that $K_{V_1}^2 > K_V^2 >0$. 
Let $C \subset V_1$ be an irreducible curve. 
Then we have 
\[
K_{V_1} \cdot C = K_{V_1} \cdot f_{*}C_V = f^{*} K_{V_1} \cdot C_V = (K_{V} - aE) \cdot C_V  < 0 .
\]
Thus we see that $V_1$ is a del Pezzo surface with at most quotient singularities.

\end{proof}

\begin{lemma}\label{nor}

Let $V$ be a normal projective surface with at most quotient singularities.
Let $f : V \to V_1$ be a birational contraction of an extremal ray and denote the exceptional curve by $E$.
Then $E$ passes through at most two singular points.
Moreover, the intersection number between each connected component of the exceptional divisor and $E_Y$ is at most one, where $\pi : Y \to V$ is the minimal resolution.

\end{lemma}

\begin{proof}

$V_1$ is also a normal projective surface with at most quotient singularities.
Let $\pi :Y_1 \to V_1$ be the minimal resolution.
Then a birational morphism $g : Y \to Y_1$ is induced such that $\pi_1 \circ g = f \circ \pi$.
Since $Y$ and $Y_1$ are smooth, $g$ is decomposed into several blow-ups at a point.
We denote them by $\sigma_1, \ldots , \sigma_N$, where $g = \sigma_N \circ \cdots \circ \sigma_1$.

In \cite{Br67}, the configurations of the exceptional divisors over quotient singularities are determined.
For all quotient singularities, all irreducible components of its exceptional divisor are smooth and normal crossing.

From these facts, we obtain this assertion.

\end{proof}

\begin{lemma}\label{no_neg}

Let $V$ be a del Pezzo surface and $\pi : Y \to V$ the minimal resolution.
For an irreducible curve $C$ on $V$, if $C_Y$ is a negative curve, then $C_Y$ is a $(-1)$-curve.

\end{lemma}

\begin{proof}

Assume that $C_Y$ is a $(-n)$-curve.
We have $\pi^* C = C_Y + \sum a_i E_i$, where $E_i$ is the exceptional curve and $a_i \ge 0$ for each $i$.
Hence we have
\[
-K_Y \cdot \pi^* C = -K_Y \cdot C_Y + (-K_Y) \cdot \sum a_i E_i . 
\]
Since we see that $-K_V \cdot C > 0$ and $-K_Y \cdot \sum a_i E_i \le 0$, we have $n<2$.

\end{proof}

\begin{lemma}\label{sm_on_qline}

Let $V$ be a del Pezzo surface and $f : V \to V_1$ a birational extremal contraction whose center is a smooth point $P$.
Then there is no quasi-line passing through $P$.

\end{lemma}

\begin{proof}

Assume there is a quasi-line $C \subset V_1$ passing through $P$.
Denote the $f$-exceptional curve by $E$.
Let $\pi : Y \to V$ and $\pi_1 : Y_1 \to V_1$ be the minimal resolutions.
Then a birational morphism $g: Y \to Y_1$ such that $f \circ \pi = \pi_1 \circ g$ is induced.
We write $C_{Y_1}$ for the strict transform of $C$ by $\pi_1$ and $C_{V}$ for the one by $f$. 
Let $C_Y$ be the strict transform of $C_{Y_1}$ by $g$, which is also the strict transform of $C_V$ by $f$. 
Since $Y$ and $Y_1$ are smooth, $g$ is a composition of blow-ups at a point.
Since $C$ passes through $P$, there is at least one blow-up at a point on $C_{Y_1}$.
Since $C_{Y_1}$ is a $(-1)$-curve, we see that $C_Y$ is $(-n)$-curve, where $n \ge 2$.
This contradicts Lemma \ref{no_neg}.

\end{proof}

\begin{corollary}\label{float}

Let V be a del Pezzo surface and $f : V \to V_1$ a birational extremal contraction.
If there is no floating $(-1)$-curves on $V$, then there is also no floating $(-1)$-curves on $V_1$.

\end{corollary}

\begin{proof}

Assume that there exists a floating $(-1)$-curve $C$ on $V_1$.
If the center of $f$ is not in $C$, then $C_V$ is also a floating $(-1)$-curve.
This is a contradiction.
If the center of $f$ is in $C$, then it also contradicts Lemma \ref{sm_on_qline}.

\end{proof}

\subsection{Classification of extremal contractions}

In this subsection, we consider a sequence of extremal contractions which starts from a del Pezzo surface of type $\tA$ and classify such extremal contractions.
By Proposition \ref{keep_B}, we see that all del Pezzo surfaces appearing in minimal model programs are of type $\tB$.

Let $X$ be a del Pezzo surface of type $\tA$. 
By running a minimal model program, there exist a sequence of birational extremal contractions between del Pezzo surfaces,
$X =: X_0 \overset{f_1}{\rightarrow} X_1 \overset{f_2}{\rightarrow}  \cdots \overset{f_n}{\rightarrow} X_n = X_{min} $, and a minimal surface $X_{min}$.
Let $\pi_i : Y_i \to X_i$ be the minimal resolution for $1 \le i \le n$.

\begin{lemma}\label{min_Pic}

If $X_{min}$ is a minimal del Pezzo surface, it holds that $\rho (X_{min}) = 1 \ or \ 2$.

\end{lemma}

\begin{proof}

Assume $\rho (X_{min}) \ge 3$.
Let $X_{min} \to Z$  be an extremal contraction.
Since $\dim Z \le 1$, the relative Picard number $\rho (X_{min}/ Z) \ge 2$.
This is a contradiction.

\end{proof}

\begin{proposition}\label{keep_B} 

For $1 \le i \le n$, $X_i$ is a del Pezzo surface of type $\tB$.

\end{proposition}

\begin{proof}
We prove that all $X_i$ satisfy the following three conditions by induction on $i$. 
\begin{enumerate}
  \item $X_i$ is of type $\tB$; 
  \item If $X_i$ has a singular point of type $\frac{1}{5}(1,2)$, then it is produced by contracting a curve through a singular point of type $\frac{1}{3}(1,1)$ and a singular point of type $\frac{1}{4}(1,1)$; 
  \item If $X_i$ has a singular point of type $\frac{1}{4}(1,1)$, then $F_i$ is isomorphic near the point, where $F_i := f_i \circ f_{i-1} \circ \cdots \circ f_1 : X \to X_{i}$. It means that singular points of type $\frac{1}{4}(1,1)$ cannot be produced by any extremal contractions.
\end{enumerate}

Set $X_0 := X$.
Since $X$ is of type $\tA$, $X_0$ satisfies these three conditions. 
Assume that $X_i$ satisfies the three conditions.
Let us prove that $X_{i+1}$ also satisfies them.
Denote by $E$ the exceptional curve of $f_{i+1} : X_i \to X_{i+1}$.
We denote by $Q$ the point to which $E$ is contracted.
By Lemma \ref{nor}, $E$ passes through at most two singular points.
If $E$ does not pass through any singular points, then $E$ is a $(-1)$-curve.
Therefore, $X_{i+1}$ also satisfies the three conditions in this case.
If $E$ passes through only one singular point $P$, then $f_{i+1}$ is one in the following table. 

{\small
\begin{table}[htb]
\caption{$(P, Q)$}\label{sing_1_table}
\renewcommand{\arraystretch}{1.4}
  \begin{tabular}{|c||c|c|c|c|c|c|c|c|c|} \hline
$P$  &   $A_1$ & $\frac{1}{3}$(1,1) & $\frac{1}{4}$(1,1) & $A_2$ & \multicolumn{2}{c|}{$\frac{1}{5}$(1,2)} & \multicolumn{2}{c|}{$A_3$}  \\ \hline
  & & & & & a & b & a & b \\ \hline \hline
$Q$ & sm & $A_1$ & $\frac{1}{3}$(1,1) & sm & $A_1$ & $A_2$ & sm & fib \\ \hline
  \end{tabular} \\
\end{table}
}
\noindent Here if $P$ is a singular point of type $\frac{1}{5}(1,2)$ or a singular point of type $A_3$, there are two possible ways to contract $E$ respectively.
The possibilities of the dual graph of $\pi_i^{-1} (f_{i+1}^{-1} (Q) )$ are two cases respectively: 
\begin{center}
$\frac{1}{5}(1,2)$ : \ \ \ type a \xygraph{
    \bullet ([]!{+(0,-.4)} {-1}) - [r]
    \triangle ([]!{+(0,-.4)} {-2}) - [r] 
    \square ([]!{+(0,-.4)} {-3}) 
}
 \ \ \ \ 
 type b \xygraph{
    \triangle ([]!{+(0,-.4)} {-2}) - [r] 
    \square ([]!{+(0,-.4)} {-3}) - [r]
    \bullet ([]!{+(0,-.4)} {-1}) 
}
\end{center}
\begin{center}
$A_3$ : \ \ \ type a \xygraph{
	\triangle ([]!{+(.3,-.3)} {-2}) 
       		(- [d] \bullet ([]!{+(.5,0)} {-1}))
	- [r]	\triangle ([]!{+(.0,-.3)} {-2}) 
        - [r]	\triangle ([]!{+(.0,-.3)} {-2})
}
\ \ \ \ type b \xygraph{
	\triangle ([]!{+(0,-.3)} {-2}) 
	- [r]	\triangle ([]!{+(.3,-.3)} {-2}) 
       		(- [d] \bullet ([]!{+(.5,0)} {-1}))
        - [r]	\triangle ([]!{+(.0,-.3)} {-2})
} 
\end{center}
By Table \ref{sing_1_table}, we see that $X_{i+1}$ also satisfies the three conditions in this case.

From now on, we assume that $E$ passes through exactly two singular points $P_1, P_2$.
By Lemma \ref{Fuku}, $E$ is not a $T$-line.
Thus we can eliminate the case $( P_1, P_2 ) = ( \frac{1}{4}(1,1), \frac{1}{4}(1,1) )$.
We write ``/" where this case is in Table \ref{ext_cont_table}.
The other cases where ``/" is written in Table \ref{ext_cont_table} are eliminated by a contradiction to negative definiteness.
The following are the cases we must consider especially.

\noindent{\bf Case 1 : $(P_1, P_2, Q) = (\frac{1}{4}(1,1), \frac{1}{5}(1,2), \frac{1}{7}(1,3) )$} \\ 
By the assumption of induction, there exists $0 < j < i+1$ such that $f_j : X_{j-1} \to X_j$ contracts a quasi-line passing through a singular point of type $\frac{1}{3}(1,1)$ and a singular point of type $\frac{1}{4}(1,1)$.
Set $f := f_{i+1} \circ f_i \circ \cdots \circ f_j : X_{j-1} \to X_{i+1}$.
By Lemma \ref{sm_on_qline}, any exceptional curves are not contracted to smooth points.
Hence we see that the dual graph of $\pi_{j-1}^{-1}(f^{-1}(Q))$ is as follows:
\begin{center}
\  \xygraph{
    \bigcirc ([]!{+(0,-.4)} {-4}) - [r]
    \bullet ([]!{+(0,-.4)} {-1}) - [r] 
    \bigcirc ([]!{+(0,-.4)} {-4}) - [r] 
    \bullet ([]!{+(0,-.4)} {-1}) - [r] 
    \square ([]!{+(0,-.4)} {-3}) }
\end{center}
This means that there exists a $T$-line on $X_{j-1}$.
This contradicts the fact $X$ is a del Pezzo surface.
Thus this case does not occur.
 
\noindent{\bf Case 2 : $(P_1, P_2, Q) = (\frac{1}{5}(1,2), \frac{1}{5}(1,2), A_4 )$} \\ 
By the assumption of induction, there exist $0 < j < k < i+1$ such that $f_j : X_{j-1} \to X_j$ and $f_k : X_{k-1} \to X_k$ contract a quasi-line passing through a singular point of type $\frac{1}{3}(1,1)$ and a singular point of type $\frac{1}{4}(1,1)$ respectively.
Set $f := f_{i+1} \circ f_i \circ \cdots \circ f_j : X_{j-1} \to X_{i+1}$.
By Lemma \ref{sm_on_qline}, any exceptional curves are not contracted to smooth points.
Hence we see that the dual graph of $\pi_{j-1}^{-1}(f^{-1}(Q))$ is as follows:
\begin{center}
\  \xygraph{
    \square ([]!{+(0,-.4)} {-3}) - [r]
    \bullet ([]!{+(0,-.4)} {-1}) - [r]
    \bigcirc ([]!{+(0,-.4)} {-4}) - [r]
    \bullet ([]!{+(0,-.4)} {-1}) - [r] 
    \bigcirc ([]!{+(0,-.4)} {-4}) - [r] 
    \bullet ([]!{+(0,-.4)} {-1}) - [r] 
    \square ([]!{+(0,-.4)} {-3}) }
\end{center}
This means that there exists a $T$-line on $X_{j-1}$.
This also contradicts the fact $X$ is a del Pezzo surface.
Thus this case does not occur.

Thus we obtain the following Table \ref{ext_cont_table}.

{\small
\begin{table}[htb]
\caption{$(P_1, P_2, Q)$}\label{ext_cont_table}
\renewcommand{\arraystretch}{1.4}
  \begin{tabular}{|c|c||c|c|c|c|c|c|c|c|} \hline
  & $P_2$ &  $A_1$ & $A_2$ & \multicolumn{2}{c|}{$A_3$}  & $\frac{1}{3}$(1,1) & $\frac{1}{4}$(1,1) &  \multicolumn{2}{c|}{$\frac{1}{5}$(1,2)}  \\ \hline
 $P_1$ & & & & a & b & & & a & b \\ \hline \hline
 $A_1$ & &  fib & & & & & & & \\ \hline
 $A_2$ & & / & / &  & & & & & \\ \hline
 $A_3$ & a &  /  & / & / &  & & &  & \\ \hline
 $A_3$ & b &  / & / & / & /  &   &  &  &  \\ \hline
 $\frac{1}{3}$(1,1) & & sm & fib & / & / & $A_2$ & & & \\ \hline
 $\frac{1}{4}$(1,1) & & $A_1$ & sm & fib & / & $\frac{1}{5}(1,1)$ & Lem \ref{Fuku} & & \\ \hline
 $\frac{1}{5}$(1,2) & a &  /  & / & / & / & sm & $A_2$ & / & \\ \hline
 $\frac{1}{5}$(1,2) & b &  sm & / & / & / & / & Case 1 & fib & Case 2 \\ \hline
  \end{tabular} \\
\end{table}
}

By this table, we see that $X_{i+1}$ satisfies the three conditions if it exists.
Thus we see that each $X_i$ satisfies the three conditions by induction.
In particular, each $X_i$ is a del Pezzo surface of type $\tB$.

\end{proof}

\begin{notation}

Let $f : V \to V_1$ be a birational morphism of surfaces.
Then $d_{V/V_1}$ denotes the value of difference of anti-canonical volumes $K_{V_1}^2 - K_V^2$. 

\end{notation}

\begin{proposition}\label{all_ext_cont}

Let $V$ be a del Pezzo surface of type $\tB$ which is obtained from a del Pezzo surface of type $\tA$.
Let $f : V \to V_1$ be an extremal contraction.
If $\dim V_1 = 2$, that is, $f$ is birational, then $f$ is one in Table \ref{bir_ext_cont}.

{\rm \small
\begin{table}[htb]
\caption{Birational extremal contractions}\label{bir_ext_cont}
\renewcommand{\arraystretch}{1.35}
  \begin{tabular}{|c|c|c|c|c|} \hline
    \multicolumn{1}{|l|}{} & \multicolumn{1}{c|}{From}
      & \multicolumn{1}{c|}{To} & \multicolumn{1}{c|}{$d_{V/V_1}$} & configurations  \\ \hline \hline
    $\sB_0$ & - & sm pt & 1  & $\bullet$ \\ \hline
    $\sB_{1}$ & $A_1$ & sm pt & 2 & $\bullet - \triangle$ \\
    $\sB_{2}$ & $A_2$ & sm pt & 3 & $\bullet - \triangle - \triangle$  \\
    $\sB_{3}$ & $A_3$ & sm pt & 4 & $\bullet - \triangle - \triangle - \triangle$  \\
    $\sB_{4}$ & $\frac{1}{3}$(1,1) & $A_1$  &  $\frac{2}{3}$ & $\bullet - \square$ \\ 
    $\sB_{5}$ & $\frac{1}{4}$(1,1) & $\frac{1}{3}$(1,1)  &  $\frac{1}{3}$ & $\bullet - \bigcirc$ \\ 
    $\sB_{6}$ & $\frac{1}{5}$(1,2) & $A_1$ &  $\frac{8}{5}$ & $\bullet - \triangle - \square$ \\ 
    $\sB_{7}$ & $\frac{1}{5}$(1,2) & $A_2$ & $\frac{3}{5}$ & $\bullet - \square - \triangle$ \\ \hline
    $\sB_{8}$ & $A_1$ and $\frac{1}{3}$(1,1) & sm pt  &  $\frac{8}{3}$ & $\triangle - \bullet - \square$ \\
    $\sB_{9}$ & $A_1$ and $\frac{1}{5}$(1,2) & sm pt &  $\frac{18}{5}$ & $\triangle - \bullet - \square - \triangle$  \\ 
    $\sB_{10}$ & $A_2$ and $\frac{1}{4}$(1,1) & sm pt  &  3 & $\triangle -  \triangle - \bullet - \bigcirc$ \\
    $\sB_{11}$ & $\frac{1}{3}$(1,1) and $\frac{1}{5}$(1,2)  & sm pt  &  $\frac{49}{15}$ & $\square - \bullet - \triangle - \square$ \\
    $\sB_{12}$ & $A_1$ and $\frac{1}{4}$(1,1) & $A_1$ &  1 & $\triangle - \bullet - \bigcirc$ \\
    $\sB_{13}$ & $\frac{1}{3}$(1,1) and $\frac{1}{3}$(1,1)  & $A_2$  & $\frac{1}{3}$ & $\square - \bullet - \square$ \\
    $\sB_{14}$ & $\frac{1}{3}$(1,1) and $\frac{1}{4}$(1,1)  & $\frac{1}{5}$(1,2)  & $\frac{1}{15}$ & $\square - \bullet - \bigcirc$ \\
    $\sB_{15}$ & $\frac{1}{5}$(1,2) and $\frac{1}{4}$(1,1)  &$A_2$ &  $\frac{3}{5}$ & $\square - \triangle - \bullet - \bigcirc$ \\
    $\sB_{16}$ & \ $\frac{1}{5}$(1,2) and $\frac{1}{3}$(1,1) \  &$A_3$  & $\frac{4}{15}$ & $\triangle - \square - \bullet - \square$ \\ \hline 
  \end{tabular}
\end{table}
}

If $\dim V_1 = 1$, that is, $f$ is a $\pr^1$-fibration, then $f$ is one in Table \ref{fibers}.

{\rm \small
\begin{table}[htb]
\caption{Non birational extremal contractions ($\pr^1$-fibration)}\label{fibers}
\renewcommand{\arraystretch}{1.3}
  \begin{tabular}{|c|c|c|c|c|} \hline
    \multicolumn{1}{|l|}{} & \multicolumn{1}{c|}{From}  & configurations \\ \hline \hline
    $\sC_1$ & -  &  \\
    $\sC_2$ & $A_3$ & \xygraph{
	\triangle ([]!{+(0,-.3)} ) 
	- [r]	\triangle ([]!{+(.3,-.3)} ) 
       		(- [d] \bullet ([]!{+(.5,0)}))
        - [r]	\triangle ([]!{+(.0,-.3)} )
}  \\
    $\sC_3$ & $A_1$ and $A_1$ & \hspace{16mm} \xygraph{
	\triangle ([]!{+(0,-.3)} ) 
	- [r]	\bullet ([]!{+(.3,-.3)} ) 
        - [r]	\triangle ([]!{+(.0,-.3)} )
} \hspace{16mm} \\
  \end{tabular}
\end{table}
}

{\rm \small
\begin{table}[htb]
\renewcommand{\arraystretch}{1.4}
  \begin{tabular}{|c|c|c|c|c|} 
    \multicolumn{1}{|l|}{No.} & \multicolumn{1}{c|}{From}  & configurations \\ \hline \hline
    $\sC_4$ & $\frac{1}{3}$(1,1) and $A_2$ & \xygraph{
	\square ([]!{+(0,-.3)}) 
	- [r]	\bullet ([]!{+(.3,-.3)}) 
        - [r]	\triangle ([]!{+(.0,-.3)})
        - [r]	\triangle ([]!{+(.0,-.3)})
} \\ 
    $\sC_5$ & $\frac{1}{4}$(1,1) and $A_3$ & \xygraph{
	\bigcirc ([]!{+(0,-.3)}) 
	- [r]	\bullet ([]!{+(.3,-.3)}) 
        - [r]	\triangle ([]!{+(.0,-.3)})
        - [r]	\triangle ([]!{+(.0,-.3)})
        - [r]	\triangle ([]!{+(.0,-.3)})
}  \\ 
    $\sC_6$ & $\frac{1}{5}$(1,2) and $\frac{1}{5}$(1,2) & \xygraph{
	\triangle ([]!{+(0,-.3)}) 
        - [r]	\square ([]!{+(.0,-.3)})
	- [r]	\bullet ([]!{+(.3,-.3)}) 
        - [r]	\triangle ([]!{+(.0,-.3)})
        - [r]	\square ([]!{+(.0,-.3)})
} \\ \hline
  \end{tabular}
\end{table}
}

Here `` {\rm From}" means singular points the exceptional curve passing through and `` {\rm To}" means a singular point to which the exceptional curve contracted.
The meaning of $\bullet$, $\triangle$, $\square$ and $\bigcirc$ is defined in Definition \ref{dual_graph}.

\end{proposition}

\begin{proof}

The assertion follows from Tables \ref{sing_1_table} and \ref{ext_cont_table}.

\end{proof}

Moreover, observing the proof of Proposition \ref{keep_B}, we see that we can use the same notation even for non del Pezzo surfaces.
The following lemma is needed when we play two ray games (Subsection \ref{2_ray}).

\begin{lemma}\label{non_del}

Let $Y$ be a rational surface of rank two with at most $A_1, A_2, A_3$, $\frac{1}{3}(1,1)$ and $\frac{1}{4}(1,1)$-singularities.
Let $\psi : Y \to Z$ be an extremal contraction.
Then $\psi$ is of one type in $\sB_0$, $\sB_1$, $\sB_2$, $\sB_3$, $\sB_4$, $\sB_5$, $\sB_8$, $\sB_{10}$, $\sB_{12}$, $\sB_{13}$, $\sB_{14}$, $\sC_1$, $\sC_2$, $\sC_3$, $\sC_4$ and $\sC_5$ in Tables \ref{bir_ext_cont} and \ref{fibers}.

\end{lemma}

\begin{proof}

By Lemma \ref{nor}, we see that the exceptional curve passes through at most two singular points.
Since $Y$ does not have singular points of type $\frac{1}{5}(1,2)$, we see the candidates easily.

\end{proof}

\begin{notation}

Let $\varphi_i$ be an extremal contraction of type $\sB_{s_i}$ for $1 \le i \le n$ and $\varphi = \varphi_1 \circ \cdots \circ \varphi_n$.
Then we call $\varphi$ of type $\sB_{s_1} \circ \cdots \circ \sB_{s_n}$.

\end{notation}

\subsection{Compositions of extremal contractions}

\begin{definition}\label{2_min}

Let $X$ be a del Pezzo surface of type $\tA$.
$X$ is called {\it $\smor$-minimal} if for any composition of birational extremal contractions $\varphi : X \to X_1$, the center of $\varphi$ is contained in $\Sing X_1$.

\end{definition}

$\smor$-minimal surfaces play important roles in Section \ref{candidate_sec}.
There are six classes of $\smor$-minimal surfaces, which is proved in Proposition \ref{T_min}.
By definition, we obtain the following lemma which justifies the name $\smor$-minimal.

\begin{lemma}\label{A-min}

Let $T_{min}$ be a $\smor$-minimal del Pezzo surface and $f : T_{min} \to U_1$ a composition of extremal contractions. 
Then $U_1$ does not have any birational extremal contractions $\psi : U_1 \to Z$ whose center is a smooth point.

\end{lemma}

\begin{proof}

Assume that there is a composition of birational extremal contractions $\psi : U_1 \to Z$ such that its center is a smooth point.
Then $\psi \circ f : T_{min} \to Z$ is a birational contraction whose center is a smooth point.
This contradicts the definition of $\smor$-minimal.

\end{proof}

\begin{definition}\label{fst}
An extremal contraction of type $\sB_0 =: \fmor_1$ is called a {\it first morphism}. 
A composition of extremal contractions $\varphi$ is called a {\it second morphism} (resp. {\it third morphism}) if it satisfies the following two conditions {\rm (1)}, {\rm (2)} (resp.   two conditions {\rm (1)}, {\rm (3)}):  \\
  {\rm (1)} Each irreducible component of the exceptional curves passes through at least one singular point and the singular points which they pass through are only of type $\frac{1}{3}(1,1)$ or $\frac{1}{4}(1,1)$;  \\
  {\rm (2)} The center of $\varphi$ is a smooth point; \\
  {\rm (3)} The center of $\varphi$ is a singular point. 

\end{definition}

\begin{definition}

Let $S$ be a del Pezzo surface of type $\tA$ with no floating $(-1)$-curves.
A {\it $\smor$-sequence} from $S$ is a sequence of second morphisms
\[
S =: X_0 \overset{\beta_1}{\rightarrow} X_1 \overset{\beta_2}{\rightarrow} \cdots \overset{\beta_m}{\rightarrow} X_m =: T_{min}
\]
such that $T_{min}$ is $\smor$-minimal.

\end{definition}

\begin{lemma}\label{2nd_part}

Let $S$ be a del Pezzo surface of type $\tA$ with no floating $(-1)$-curves.
There exists a $\smor$-sequence from $S$.

\end{lemma}

\begin{proof}

We may assume that there exists a composition of extremal contractions $\beta_1$ whose center is a smooth point since otherwise $S$ is $\smor$-minimal and $S = T_{min}$.
Then $\beta_1$ satisfies the condition (2) in Definition \ref{fst}.
We also see that $\beta_1$ satisfies the condition (1) in Definition \ref{fst} since $S$ is of type $\tA$ and has no floating $(-1)$-curves.
Thus $\beta_1$ is a second morphism.
Then $X_1$ is also a del Pezzo surface of type $\tA$ with no floating $(-1)$-curves by Lemma \ref{sm_on_qline}.
Since $X_1$ satisfies the same assumption as $S$, we can repeat such contractions as many times as possible.
Since $\rho (S)$ is finite and $\rho(S) > \rho(X_1)$, a sequence of second morphisms is finite.
Therefore, we obtain a sequence $S \overset{\beta_1}{\rightarrow} \cdots \overset{\beta_m}{\rightarrow} T_{min}$ and a $\smor$-minimal surface $T_{min}$.

\end{proof}

\begin{definition}

Let $T_{min}$ be a $\smor$-minimal del Pezzo surface.
A { \it $\tmor$-sequence} from $T_{min}$ is a sequence of third morphisms
\[
T_{min}  =: X_0 \overset{\gamma_1}{\rightarrow} X_{1} \overset{\gamma_2}{\rightarrow} \cdots \overset{\gamma_n}{\rightarrow} X_n =: X_{min}
\]
satisfying the following conditions: \\
{\rm (1)} $X_{min}$ is minimal and of type $\tB$; \\
{\rm (2)} The centers of $\gamma_1, \ldots , \gamma_n$ are distinct singular points on $X_{min}$.

\end{definition}

\begin{lemma}\label{3rd_part}

Let $T_{min}$ be a $\smor$-minimal del Pezzo surface.
There exists a $\tmor$-sequence from $T_{min}$.

\end{lemma}

\begin{proof}

Take a sequence of extremal contractions $T_{min} \overset{f_1}{\rightarrow} \cdots \overset{f_N}{\rightarrow} X_{min} $ and set $f := f_N \circ \cdots \circ f_1$.
Denote the connected components of the $f$-exceptional divisor by $\Gamma_1, \ldots , \Gamma_n$.
Since $T_{min}$ is $\smor$-minimal, each $\Gamma_i$ is contracted to a singular point $P_k$ on $X_{min}$.
Then we denote the contraction of $\Gamma_k$ by $\gamma_k$ and obtain a sequence
\[
T_{min}  =: X_0 \overset{\gamma_1}{\rightarrow} X_{1} \overset{\gamma_2}{\rightarrow} \cdots \overset{\gamma_n}{\rightarrow} X_n =: X_{min} .
\]
Then we can confirm that each $\gamma_k$ is a third morphism since $T_{min}$ is $\smor$-minimal.

\end{proof}

\begin{proposition}\label{mor_list}
Let $T$ and $U$ be del Pezzo surfaces.

A second morphism $\varphi : T \to T_1$ is one in the following list:

{\rm \small
\begin{table}[htb]
\caption{Second morphisms}\label{2nd_mor}
\renewcommand{\arraystretch}{1.35}
  \begin{tabular}{|c|c|c|c|c|} \hline
    \multicolumn{1}{|l|}{Name}  & \multicolumn{1}{c|}{Compositions} & \multicolumn{1}{c|}{($\frac{1}{3}$(1,1), $\frac{1}{4}$(1,1))} &  \multicolumn{1}{c|}{$d_{T/{T_1}}$}  \\ \hline \hline
    $\smor_1$ &  $\sB_{1} \circ \sB_4$ & (1,0) & $\frac{8}{3}$  \\ 
    $\smor_2$ &  $\sB_{1} \circ \sB_4 \circ \sB_5$ & (0,1) & 3 \\ 
    $\smor_3$ &  $\sB_{8} \circ \sB_4$ & (2,0) & $\frac{10}{3}$  \\ 
    $\smor_4$ &  $\sB_{10} \circ \sB_{13}$ & (2,1) & $\frac{10}{3}$  \\
    $\smor_5$ &  $\sB_8 \circ \sB_5 \circ \sB_4$ & (1,1) & $\frac{11}{3}$  \\
    $\smor_6$ &  $\sB_8 \circ \sB_4 \circ \sB_5$ & (1,1) & $\frac{11}{3}$   \\
    $\smor_7$ &  $\sB_{10} \circ \sB_{13} \circ \sB_5$ & (1,2) & $\frac{11}{3}$  \\
    $\smor_8$ &  $\sB_8 \circ \sB_{12} \circ \sB_4$ & (2,1) & $\frac{13}{3}$  \\   \hline
  \end{tabular}
\end{table}
}

A third morphism $\varphi : U \to U_1$ is one in the following list:

{\rm \small
\begin{table}[htb]
\caption{Third morphisms}\label{3rd_mor}
\renewcommand{\arraystretch}{1.4}
  \begin{tabular}{|c|c|c|c|c|c|} \hline
    \multicolumn{1}{|l|}{Name} & \multicolumn{1}{c|}{To}
      & \multicolumn{1}{c|}{Compositions} & \multicolumn{1}{c|}{($\frac{1}{3}$(1,1), $\frac{1}{4}$(1,1))} &  \multicolumn{1}{c|}{$d_{U/{U_1}}$} & \ \ configurations \ \ \\ \hline \hline
    $\tmor_1$ & $\frac{1}{5}$(1,2) &  $\sB_{14}$ & (1,1) & $\frac{1}{15}$  &  $\bigcirc - \bullet - \square$ \\ \hline
    $\tmor_2$ & $\frac{1}{3}$(1,1) &  $\sB_{5}$ & (0,1) & $\frac{1}{3}$  & $\bigcirc - \bullet$ \\ 
  \end{tabular}
\end{table}
}

{\rm \small
\begin{table}[htb]
\renewcommand{\arraystretch}{1.4}
  \begin{tabular}{|c|c|c|c|c|c|} 
    \multicolumn{1}{|l|}{Name} & \multicolumn{1}{c|}{To}
      & \multicolumn{1}{c|}{Compositions} & \multicolumn{1}{c|}{($\frac{1}{3}$(1,1), $\frac{1}{4}$(1,1))} &  \multicolumn{1}{c|}{$d_{U/{U_1}}$} & configurations \\ \hline \hline
    $\tmor_3$ & $A_3$ &  $\sB_{16} \circ \sB_{14}$ & (2,1) & $\frac{1}{3}$ & $\square - \bullet - \bigcirc - \bullet - \square$ \\ \hline
    $\tmor_4$ & $A_2$ &  $\sB_{15} \circ \sB_{14}$ & (1,2) & $\frac{2}{3}$ & $\bigcirc - \bullet - \square - \bullet - \bigcirc$ \\
    $\tmor_5$ & $A_2$ &  $\sB_{13}$ & (2,0) & $\frac{1}{3}$  & $\square - \bullet - \square$ \\ 
    $\tmor_6$ & $A_2$ &  $\sB_{13} \circ \sB_{5}$ & (1,1) & $\frac{2}{3}$ & $\square - \bullet - \bigcirc - \bullet$ \\  \hline
    $\tmor_7$ & $A_1$ &  $\sB_{4}$ & (1,0) & $\frac{2}{3}$ &  $\square - \bullet$ \\ 
    $\tmor_8$ & $A_1$ &  $\sB_{4} \circ \sB_5$ & (0,1) & 1  &  $\bullet - \bigcirc - \bullet$ \\ 
    $\tmor_9$ & $A_1$ &  $\sB_{12} \circ \sB_4$ & (1,1) & $\frac{5}{3}$  & $\bigcirc - \bullet - \square - \bullet$ \\ \hline
  \end{tabular}
\end{table}
}

Here the exceptional curves of a second morphism of type $\smor_i$ (resp. $\tmor_j$ ) is called a $\smor_i$-line pair (resp. $\tmor_j$-line pair).
{ \rm  ``($\frac{1}{3}$(1,1), $\frac{1}{4}$(1,1))"} means the numbers of singular points contracted by the morphisms.
{ \rm ``configurations"} means the dual gragh of total transform of $\tmor_j$-line pair by the minimal resolution of $U$.
The ones of second morphisms are listed in Corollary \ref{config}.

\end{proposition}

\begin{proof}

We first consider third morphisms.
Let $\varphi : T \to T_1$ be a third morphism.
Denote by $P$ the singular point to which $\varphi$ contracts curves.
By definition, $\varphi$ is decomposed into several birational cotractions of extremal rays Table \ref{bir_ext_cont}.
Assume that $\varphi$ is of type $\sB_{i_s} \circ \cdots \circ \sB_{i_1}$.
By Lemma \ref{sm_on_qline}, there is no contraction contracting a curve to a smooth point on a quasi-line.
Hence we consider only contractions of extremal rays whose centers are singular points.
Thus we see that $i_j = 4, 5, 6, 7, 12, 13, 14, 15$ or 16 for all $\sB_{i_j}$. 

\noindent{\bf Case 1 : $P$ is a singular point of type $\frac{1}{4}(1,1)$} \\
We see that singular points of type $\frac{1}{4}(1,1)$ cannot be produced by any contractions by Proposition \ref{keep_B}.

\noindent{\bf Case 2 :  $P$ is a singular point of type $\frac{1}{3}(1,1)$} \\ 
A singular point of type $\frac{1}{3}(1,1)$ on a minimal surface is produced only by a contraction of type $\sB_{5}$ if it is produced by some contractions.
We also see that singular points of type $\frac{1}{4}(1,1)$ cannot be produced by any contractions.
Hence $\varphi$ is of type  $\sB_5$.
We denote this type by $\tmor_2$.

\noindent{\bf Case 3 :  $P$ is a singular point of type $\frac{1}{5}(1,2)$} \\ 
A singular point of type $\frac{1}{5}(1,2)$ on a minimal surface is produced only by a contraction of type $\sB_{14}$.
A contraction of type $\sB_{14}$ needs one singular point of type $\frac{1}{3}(1,1)$ and one singular point of type $\frac{1}{4}(1,1)$.
If the singular point of type $\frac{1}{3}(1,1)$ is produced by some contractions, the contraction is of type $\sB_5$.
If $\varphi$ is of type $\sB_{14} \circ \sB_5$, however, then $\Exc (\varphi )$ has a $T$-line. 
Hence $\varphi$ is of type $\sB_{14}$.
We denote this type by $\tmor_1$.

\noindent{\bf Case 4 :  $P$ is a singular point of type $A_3$} \\ 
A singular point of type $A_3$ on a minimal surface is produced only by a contraction of type $\sB_{16}$.
A contraction of type $\sB_{16}$ needs one singular point of type $\frac{1}{5}(1,2)$ and one singular point of type $\frac{1}{3}(1,1)$.
Singular points of type $\frac{1}{5}(1,2)$ do not exist on a del Pezzo surface of type $\tA$ and they are produced only by a contraction of type $\sB_{14}$.
If the singular point of type $\frac{1}{3}(1,1)$ is produced by an extremal contraction, then we see that $\Exc (\varphi )$ has a $T$-line.
Hence $\varphi$ is of type $\sB_{16} \circ \sB_{14}$.
We denote this type by $\tmor_3$.

\noindent{\bf Case 5 :  $P$ is a singular point of type $A_2$} \\ 
The type of last extremal contraction $\sB_{i_s}$ must be $\sB_{13}$ or $\sB_{15}$.
In the same manner we see that there are four possible types of $\varphi$, $\sB_{15} \circ \sB_{7}$, $\sB_{13}$, $\sB_{13} \circ \sB_5$ or $\sB_{15} \circ \sB_{14}$.
We see that $\sB_{15} \circ \sB_{7} = \sB_{13} \circ \sB_5$.
Hence $\varphi$ is of type $\sB_{13}$, $\sB_{13} \circ \sB_5$ or $\sB_{15} \circ \sB_{14}$.
We denote them by $\tmor_4, \tmor_5, \tmor_6$ respectively.

\noindent{\bf Case 6 :  $P$ is a singular point of type $A_1$} \\ 
The last type of extremal contraction $\sB_{i_s}$ must be $\sB_{4}, \sB_6$ or $\sB_{12}$.
A contraction of type $\sB_4$ needs one singular point of type $\frac{1}{3}(1,1)$.
If the singular point of type $\frac{1}{3}(1,1)$ is produced by some contractions, then the contraction is of type $\sB_5$.
We denote the type $\sB_4$ by $\tmor_7$ and $\sB_4 \circ \sB_5$ by $\tmor_8$.
A contraction of type $\sB_6$ needs one singular point of type $\frac{1}{5}(1,2)$.
Singular points of type $\frac{1}{5}(1,2)$ must be produced by a contraction of type $\sB_{14}$.
We denote $\sB_6 \circ \sB_{14}$ by $\tmor_9$.
A contraction of type $\sB_{12}$ needs one singular point of type $\frac{1}{4}(1,1)$ and one singular point of type $A_1$.
Then the singular point of $A_1$ is produced only by a contraction of type $\sB_4$. 
Otherwise, it is produced by a contraction of type $\sB_6$ or $\sB_{12}$ and then $\Exc (\sB_{12} \circ \sB_6)$ and $\Exc (\sB_{12} \circ \sB_{12})$ have a $T$-line, which is a contradiction. 
Then we also see that $\sB_{12} \circ \sB_4 = \tmor_9$.

By these considerations, we obtain Table \ref{3rd_mor}.

Next, we consider second morphisms.
Let $\varphi : U \to U_1$ be a second morphism.
By definition, $\varphi$ is decomposed into several birational extremal contractions in Table \ref{bir_ext_cont}.
We may assume that $\varphi$ is of type $\sB_{i_t} \circ \cdots \circ \sB_{i_1}$.
Since the center of $\varphi$ is a smooth point, we see that the center of the last extremal contraction is also a smooth point.
Thus we see that the type $\sB_{i_t}$ of the last extremal contraction is one of the seven types $\sB_1, \sB_2, \sB_3, \sB_8, \sB_9, \sB_{10}$ and $\sB_{11}$.

\noindent{\bf Case 7 :  $\sB_{i_t}$ = $\sB_1$} \\
An extremal contraction of type $\sB_1$ contracts a curve passing through a singular point of type $A_1$.
Therefore, this case can be reduced to how to produce a singular point of type $A_1$ (Case 6).
Thus we see that candidates of the type of $\varphi$ is one of $\sB_1 \circ \sB_4$, $\sB_1 \circ \sB_4 \circ \sB_5$ and $\sB_1 \circ \sB_{12} \circ \sB_4$.
We denote them by $\smor_1, \smor_2, \smor_5$ respectively.

\noindent{\bf Case 8 :  $\sB_{i_t}$ = $\sB_2$} \\
An extremal contraction of type $\sB_2$ contracts a curve passing through a singular point of type $A_2$.
Therefore, this case can be reduced to how to produce a singular point of type $A_2$ (Case 5).
Thus we see that candidates of the type of $\varphi$ is one of $\sB_2 \circ \sB_{13}$, $\sB_2 \circ \sB_{13} \circ \sB_5$ and $\sB_2 \circ \sB_{15} \circ \sB_{14}$.
We denote the type $\sB_2 \circ \sB_{13}$ by $\smor_3$ and the type $\sB_2 \circ \sB_{15} \circ \sB_{14}$ by $\smor_7$.
A remarkable point is that there are two possible types of $\sB_2 \circ \sB_{13} \circ \sB_5$.
Moreover, one of the two types is the same type as $\smor_5$.
Thus we denote the other type by $\smor_6$.

\noindent{\bf Case 9 :  $\sB_{i_t}$ = $\sB_3$} \\
An extremal contraction of type $\sB_3$ contracts a curve passing through a singular point of type $A_3$.
Therefore, this case can be reduced to how to produce a singular point of type $A_3$ (Case 4).
Thus we see that the candidate of the type of $\varphi$ is $\sB_{3} \circ \sB_{16} \circ \sB_{14}$.
We denote it by $\smor_8$.

\noindent{\bf Case 10 :  $\sB_{i_t}$ = $\sB_8$} \\
An extremal contraction of type $\sB_8$ contracts a curve passing through a singular point of type $A_1$ and a singular point of type $\frac{1}{3}(1,1)$.
Therefore, this case can be reduced to how to produce a singular point of type $A_1$ and a singular point of type $\frac{1}{3}(1,1)$ (Case 2 and Case 6).
Hence there are six candidates of the type of $\varphi$.
In them, the ones which does not have any $T$-lines are the three cases $\sB_8 \circ \sB_4$, $\sB_8 \circ \sB_4 \circ \sB_5$ and $\sB_8 \circ \sB_{12} \circ \sB_4$.
Then we see that of type $\sB_8 \circ \sB_4$ is the same type as $\smor_3$ and $\sB_8 \circ \sB_{12} \circ \sB_4$ is the same as $\smor_8$.
A remarkable point is that there are two possible types of $\sB_8 \circ \sB_4 \circ \sB_5$.
Then we see that one is the type $\smor_5$ and the other is the type $\smor_6$. 

\noindent{\bf Case 11 :  $\sB_{i_t}$ = $\sB_9$} \\
An extremal contraction of type $\sB_9$ contracts a curve passing through a singular point of type $A_1$ and a singular point of type $\frac{1}{5}(1,2)$.
Therefore, this case can be reduced to how to produce a singular point of type $A_1$ and a singular point of type $\frac{1}{5}(1,2)$ (Case 3 and Case 6).
As in previous cases, we see that the possible case is only $\sB_{9} \circ \sB_{14} \circ \sB_{4}$.
This type is the same as $\smor_8$.

\noindent{\bf Case 12 :  $\sB_{i_t}$ = $\sB_{10}$} \\
An extremal contraction of type $\sB_{10}$ contracts a curve passing through a singular point of type $A_2$ and a singular point of type $\frac{1}{4}(1,1)$.
Therefore, this case can be reduced to how to produce a singular point of type $A_2$ (Case 5).
As in previous cases, we see that the possible cases are $\sB_{10} \circ \sB_{13}$ and $\sB_{10} \circ \sB_{13} \circ \sB_{5}$.
The type $\sB_{10} \circ \sB_{13} \circ \sB_{5}$ is the same as $\smor_7$.
We denote $\sB_{10} \circ \sB_{13}$ by $\smor_4$.

\noindent{\bf Case 13 :  $\sB_{i_t}$ = $\sB_{11}$} \\
An extremal contraction of type $\sB_{11}$ contracts a curve passing through a singular point of type $\frac{1}{3}(1,1)$ and a singular point of type $\frac{1}{5}(1,2)$ (Case 2 and Case 3).
As in previous cases, we see that the possible cases are $\sB_{11} \circ \sB_{14}$ and $\sB_{11} \circ \sB_{14} \circ \sB_{5}$.
We see that they are the same types as $\smor_4$ and $\smor_7$ respectively.

By these considerations, we obtain Table \ref{2nd_mor}.

\end{proof}

\subsection{Minimal directed sequences}
In this subsection, we define a direction for $\smor$-sequences and $\tmor$-sequences, which is the essential ingredient in this paper.

\begin{notation}

We prepare notation in order to define a direction for $\smor$-sequences and $\tmor$-sequences.

\noindent{\bf $\bullet$ Sets ${\rm Mor}_{\smor}(S)$ and ${\rm Mor}_{\tmor}(T_{min})$}

Let $S$ be a del Pezzo surface of type $\tA$ with no floating $(-1)$-curves.
By Lemma \ref{sm_on_qline}, we see that the centers of all second morphisms in a $\smor$-sequence are disjoint.
Thus we can change the order of second morphisms in a $\smor$-sequence.
Hence we consider only the well-ordered set ${\rm Mor}_{\smor}(S) := \{ S =: X_0 \overset{\smor_{a_1}}{\to} X_1 \cdots \overset{\smor_{a_m}}{\to} X_m =T_{min} \ | \ a_i \le a_j \ {\rm for \ any} \ i < j \}$.
For a $\smor$-minimal del Pezzo surface $T_{min}$, we can also change the order of third morphisms in a $\tmor$-sequence.
Hence we can define ${\rm Mor}_{\tmor}(T_{min}) := \{ T_{min} =: X_0 \overset{\tmor_{b_1}}{\to} X_1 \cdots \overset{\tmor_{b_n}}{\to} X_n =X_{min}  \ | \ b_i \le b_j \ {\rm for \ any} \ i < j \}$.
Note that ${\rm Mor}_{\smor}(S)$ and ${\rm Mor}_{\tmor}(T_{min})$ are finite sets since $S$ and $T_{min}$ are del Pezzo surfaces.

\noindent{\bf $\bullet$ Ordered sets $D_{\smor}$ and $D_{\tmor}$}

Set $D_{\smor} := \{ (a_1, \ldots , a_m ) \in \bigsqcup_{k \in \N} \{ 1, \ldots , 8 \}^k \ | \ a_i \le a_j \ {\rm for \ any} \ i < j \} $ and
 $D_{\tmor} := \{ (b_1, \ldots , b_n ) \in \bigsqcup_{k \in \N} \{ 1, \ldots , 9 \}^k \ | \ b_i \le b_j \ {\rm for \ any} \ i < j \}$, where $\bigsqcup$ is the notation of disjoint union and $\N$ is the set of positive integers.
We define a total order $\prec$ for $D_{\smor}$ and $D_{\tmor}$ as follows.
For $(a_1, \ldots , a_m ) , (b_1, \ldots , b_n) \in D_{i}$, 
$(a_1, \ldots , a_m ) \prec (b_1, \ldots , b_n)$ if and only if they satisfy

(1) $m > n$, or 

(2) $m = n$ and $(a_1, \ldots , a_m) \le_{lex} (b_1, \ldots , b_m )$ , where $\le_{lex}$ is the lexicographical order in $\N^m$.

\noindent{\bf $\bullet$ Maps $N_{\smor}$ and $N_{\tmor}$}

We can define the following maps, $N_{\smor}, N_{\tmor}$:
\[
\begin{array}{cccc}
N_{\smor} \  : \ &  {\rm Mor}_{\smor}(S)                   & \longrightarrow & D_{\smor}                     \\[1pt]
& \rotatebox{90}{$\in$} &                 & \rotatebox{90}{$\in$} \\[-4pt]
& \{ S =: X_0 \overset{\smor_{a_1}}{\to} X_1 \cdots \overset{\smor_{a_m}}{\to} X_m =T_{min} \}                     & \longmapsto     & (a_1, \ldots , a_m)
\end{array}
\]

\[
\begin{array}{cccc}
N_{\tmor} \  : \ &  {\rm Mor}_{\tmor}(T_{min})                   & \longrightarrow & D_{\tmor}                     \\[1pt]
& \rotatebox{90}{$\in$} &                 & \rotatebox{90}{$\in$} \\[-4pt]
& \{ T_{min} =: X_0 \overset{\tmor_{b_1}}{\to} X_1 \cdots \overset{\tmor_{b_n}}{\to} X_n =X_{min} \}                     & \longmapsto     & (b_1, \ldots , b_n)
\end{array}
\]
Note that each $N_i$ is not necessary injective for $i \in \{ \smor, \tmor \}$.
${\rm Im} N_i \subset D_i$ is finite for $i \in \{ \smor, \tmor \}$.
Hence there are the minimal elements $s \in {\rm Im} N_{\smor}$ and $t \in {\rm Im} N_{\tmor}$.

\end{notation}

\begin{definition}\label{dir_seq}

A $\smor$-sequence (resp. $\tmor$-sequence) whose image by $N_{\smor}$ (resp. $N_{\tmor}$) is the minimal element $s \in {\rm Im} N_{\smor}$ (resp. $t \in {\rm Im} N_{\tmor}$) is called a {\it minimal directed $\smor$-sequence} of $S$ (resp. {\it minimal directed $\tmor$-sequence} $T_{min}$).

\end{definition}

\begin{thm}\label{directed}

Let $X$ be a del Pezzo surface of type $\tA$.
Then there exists a sequence of first morphisms $\alpha_i$, second morphisms $\beta_j$ and third morphisms $\gamma_k$
\[
X  \overset{\alpha_1}{\rightarrow} \cdots \overset{\alpha_l}{\rightarrow} 
S \overset{\beta_1}{\rightarrow} \cdots \overset{\beta_m}{\rightarrow}
T_{min} \overset{\gamma_1}{\rightarrow} \cdots \overset{\gamma_n}{\rightarrow} X_{min} 
\]
satisfying the following four conditions: \\
{\rm (1)} $S$ is a del Pezzo surface of type $\tA$ with no floating $(-1)$-curves, $T_{min}$ is a $\smor$-minimal del Pezzo surface and $X_{min}$ is a minimal surface. \\
{\rm (2)} For $1 \le i \le l$, $\alpha_i$ is of type $\fmor_1$. \\
{\rm (3)} $S \overset{\beta_1}{\rightarrow} \cdots \overset{\beta_m}{\rightarrow} T_{min}$ is a minimal directed $\smor$-sequence of $S$. \\
{\rm (4)} $T_{min} \overset{\gamma_1}{\rightarrow} \cdots \overset{\gamma_n}{\rightarrow} X_{min}$ is a minimal directed $\tmor$-sequence of $T_{min}$. 

We call this ordered sequence a {\it minimal directed sequence}.

\end{thm}

\begin{proof}

By Corollary \ref{float}, we can obtain a sequence $X  \overset{\alpha_1}{\rightarrow} \cdots \overset{\alpha_l}{\rightarrow} S$, where $\alpha_i$ is of type $\fmor_1$ and $S$ is a del Pezzo surface with no floating $(-1)$-curves.
By Definition \ref{dir_seq}, there exists a minimal directed $\smor$-sequence $S \overset{\beta_1}{\rightarrow} \cdots \overset{\beta_m}{\rightarrow} T_{min}$.
By Definition \ref{dir_seq}, there also exists a minimal directed $\tmor$-sequence $T_{min} \overset{\gamma_1}{\rightarrow} \cdots \overset{\gamma_n}{\rightarrow} X_{min}$.

\end{proof}

For Section \ref{minimal_sec} and \ref{candidate_sec}, we prepare the following corollaries.
They follow from Proposition \ref{mor_list}.

\begin{corollary}\label{bound_vol}

Let $X_{min}$ be a minimal surface of type $\tB$.
Assume that $X_{min}$ is obtained from a del Pezzo surface of type $\tA$, that is, there exist a del Pezzo surface $X$ of type $\tA$ and a sequence of first morphisms, second morphisms and third morphisms $X \to \cdots \to X_{min}$.
Denote the numbers of singular points of type $\frac{1}{5}(1,2)$, $A_1$, $A_2$ and $A_3$ by $a$, $b$, $c$ and $d$ respectively.
Then we have 
\[
(-K_{X_{min}})^2 > \frac{1}{15}a + \frac{1}{3}b + \frac{1}{3}c + \frac{2}{3}d .
\]

\end{corollary}

\begin{corollary}\label{config}

Let $\varphi : T \to T_1$ be of type $\smor_1, \ldots , \smor_7$ or $\smor_8$.
Let $\pi : Y \to T$ and $\pi_1 : Y_1 \to T_1$ be the minimal resolutions.
Then a birational morphism $g := \sigma_N \circ \cdots \circ \sigma_1 : Y \to Y_1$ is induced such that $\pi_1 \circ g = \varphi  \circ \pi$, where $\sigma_1, \ldots , \sigma_N$ are blow-ups at a smooth point.
Denote the exceptional curve of $\sigma_i$ by $E_i$.
The dual graphs of $E_1, \ldots , E_N$ on $Y$ are follwing:

{\rm \small
\begin{table}[htb]
\caption{Exceptional curves of second morphisms}\label{dual_graph_table}
\renewcommand{\arraystretch}{1.4}
  \begin{tabular}{|c|cc|} \hline
    \multicolumn{1}{|l|}{No.}   &  \multicolumn{2}{c|}{configurations} \\ \hline \hline
    $\smor_1$ & \xygraph{
    \bullet ([]!{+(0,-.4)} {E_{i}}) - [r]
    \square ([]!{+(0,-.4)} {E_1}) - [r] 
    \bullet ([]!{+(0,-.4)} {E_{j}}) 
} & $\{ i, j \} = \{ 2,3 \}$ \\ 
    $\smor_2$  & \xygraph{
    \bullet ([]!{+(0,-.3)} {E_{i}}) - [r]
    \bigcirc ([]!{+(.3,-.3)} {E_1}) (
        - [d] \bullet ([]!{+(.5,0)} {E_{j}}),
        - [r] \bullet ([]!{+(.2,-.3)} {E_{k}})
} & $\{ i, j, k \} = \{ 2,3,4 \}$ \\ 
    $\smor_3$ &  \xygraph{
    \square ([]!{+(0,-.4)} {E_1}) - [r]
    \bullet ([]!{+(0,-.4)} {E_{i}}) - [r] 
    \square ([]!{+(0,-.4)} {E_2}) - [r]
    \bullet ([]!{+(0,-.4)} {E_{j}})
} & $\{ i, j \} = \{ 3,4 \}$ \\ 
    $\smor_4$ &  \xygraph{
    \bigcirc ([]!{+(0,-.4)} {E_1}) - [r]
    \bullet ([]!{+(0,-.4)} {E_{4,5}}) - [r] 
    \square ([]!{+(0,-.4)} {E_3}) - [r] 
    \bullet ([]!{+(0,-.4)} {E_{4,5}}) - [r]
    \square ([]!{+(0,-.4)} {E_2})
}  & $\{ i, j \} = \{ 4,5 \}$ \\
    $\smor_5$ & \xygraph{
    \bullet ([]!{+(0,-.4)} {E_{i}}) - [r]
    \bigcirc ([]!{+(0,-.4)} {E_1}) - [r] 
    \bullet ([]!{+(0,-.4)} {E_{j}}) - [r] 
    \square ([]!{+(0,-.4)} {E_2}) - [r]
    \bullet ([]!{+(0,-.4)} {E_{k}})
} & $\{ i, j, k \} = \{ 3,4,5 \}$ \\
    $\smor_6$ & \xygraph{
    \square ([]!{+(0,-.4)} {E_1}) - [r]
    \bullet ([]!{+(0,-.4)} {E_{i}}) - [r]
    \bigcirc ([]!{+(.3,-.3)} {E_2}) (
        - [d] \bullet ([]!{+(.5,0)} {E_{j}}),
        - [r] \bullet ([]!{+(.2,-.4)} {E_{k}})
} & $\{ i, j, k \} = \{ 3,4,5 \}$ \\
    $\smor_7$ &  \xygraph{
    \bigcirc ([]!{+(0,-.4)} {E_1}) - [r]
    \bullet ([]!{+(0,-.4)} {E_{i}}) - [r] 
    \square ([]!{+(0,-.4)} {E_3}) - [r] 
    \bullet ([]!{+(0,-.4)} {E_{j}}) - [r] 
    \bigcirc ([]!{+(0,-.4)} {E_2}) - [r]
    \bullet ([]!{+(0,-.4)} {E_{k}})
} & $\{ i, j, k \} = \{ 4,5,6 \}$  \\
    $\smor_8$ & \xygraph{
    \square ([]!{+(0,-.4)} {E_1}) - [r]
    \bullet ([]!{+(0,-.4)} {E_{i}}) - [r] 
    \bigcirc ([]!{+(0,-.4)} {E_2}) - [r] 
    \bullet ([]!{+(0,-.4)} {E_{j}}) - [r] 
    \square ([]!{+(0,-.4)} {E_3}) - [r]
    \bullet ([]!{+(0,-.4)} {E_{k}})
}  & $\{ i, j, k \} = \{ 4,5,6 \}$ \\   \hline
  \end{tabular}
\end{table}
}

\end{corollary}

\section{Minimal surfaces}\label{minimal_sec}

In this section, we classify minimal surfaces which can be obtained by extremal contractions from del Pezzo surfaces of type $\tA$. 
By Proposition \ref{keep_B}, we  know that such surfaces are of type $\tB$. 

\begin{thm}\label{min_rank1}

Let $X$ be a rank one minimal del Pezzo surface of type $\tB$ obtained from one of type $\tA$.
Then $X$ is one of the surfaces in Table \ref{min_rank1_table}.

{\small \rm
\begin{table}[htb]
\caption{Rank one minimal surfaces of type $\tB$}\label{min_rank1_table}
\renewcommand{\arraystretch}{1.4}
  \begin{tabular}{|c|c|c|c|}  \hline
    \multicolumn{1}{|l|}{Name} & \multicolumn{1}{c|}{$\sS(M_i)$}
      & \multicolumn{1}{c|}{description} & \multicolumn{1}{c|}{$(-K)^2$}  \\ \hline \hline
    $M_{1}$ & $\frac{1}{5}$(1,2), $A_3, A_2$ & $\pr$(3,4,5) & $\frac{12}{5}$ \\
    $M_{2}$ & $\frac{1}{5}$(1,2), $A_2$ & $\pr$(1,3,5) & $\frac{27}{5}$ \\ 
    $M_{3}$ & $\frac{1}{5}$(1,2), $A_1$ & $\pr$(1,2,5) & $\frac{32}{5}$  \\ \hline
    $M_{4}$ & $\frac{1}{3}$(1,1), $A_3$ & $\pr$(1,3,4) & $\frac{16}{3}$ \\
    $M_{5}$ & $\frac{1}{3}$(1,1) & $\pr$(1,1,3) & $\frac{25}{3}$ \\ \hline
    $M_{6}$ & $\frac{1}{4}$(1,1) & $\pr$(1,1,4) & 9 \\ \hline
    $M_{7}$ & $A_3, A_3, A_1$ & cf. Remark \ref{explicit} & 2  \\
    $M_{8}$ & $A_2, A_2, A_2$ & cf. Remark \ref{explicit}  & 3  \\
    $M_{9}$ & $A_3, A_1, A_1$ & cf. Remark \ref{explicit} & 4 \\
    $M_{10}$ & $A_2, A_1$ & $\pr$(1,2,3) & 6  \\
    $M_{11}$ & $A_1$ & $\pr$(1,1,2) & 8 \\  \hline
    $M_{12}$ & - & $\pr^2$ & 9 \\ \hline
  \end{tabular}
\end{table}
}

\end{thm}

\begin{proof}

If $X$ is smooth, then we see that $X \cong \pr^2$.
Assume that $X$ is singular.
If the index of $X$ is two, thent $X \cong \pr(1,1,4)$ by Theorem \ref{min_index2}.
If $X$ is Gorenstein, then we see that $X$ is one in Table \ref{min_Gor_table}.
In particular, by Corollary \ref{bound_vol}, we see that candidates of $X$ are five cases in Table \ref{min_Gor_table}.
We may assume that $X$ has at least one singular point of type $\frac{1}{3}(1,1)$ or $\frac{1}{5}(1,2)$.
By Proposition \ref{candi_1/3_1/5}, candidates of $X$ are 19 cases.
By Corollary \ref{bound_vol}, Proposition \ref{non_exi_7_10}, Proposition \ref{non_exi_11_16}, Proposition \ref{non_exi_17_18} and Proposition \ref{non_exi_19} , we see that candidates of $X$ are five cases.

\end{proof}

\begin{remark}

We can confirm that the isomorphic class of each of $M_1, \ldots , M_{12}$ is unique respectively.
The uniqueness of $M_2, M_5, M_{10}$ and $M_{11}$ is used in this section. 
The uniqueness of $M_{10}$ and $M_{11}$ is proved in \cite{Ye02}.
The uniqueness of $M_2$ and $M_5$ is also proved by Lemma \ref{unique_1/n} and Lemma \ref{unique_1/5_A2}.

\end{remark}

\begin{thm}\label{min_rank2}

Let $X$ be a rank two minimal del Pezzo surface of type $\tB$ obtained from one of type $\tA$.
Then $X$ is one of the surfaces in Table \ref{min_rank2_table}.

{\small \rm
\begin{table}[htb]
\caption{Rank two minimal surfaces of type $\tB$}\label{min_rank2_table}
\renewcommand{\arraystretch}{1.4}
  \begin{tabular}{|c|c|c|} \hline
    \multicolumn{1}{|l|}{Name} & \multicolumn{1}{c|}{$\sS(M_i)$}
      & \multicolumn{1}{c|}{$(-K)^2$}   \\ \hline \hline
    $M_{13}$ & $\frac{1}{5}(1,2), \frac{1}{5}(1,2), \frac{1}{5}(1,2), \frac{1}{5}(1,2)$ & $\frac{8}{5}$   \\ \hline
    $M_{14}$ & $\frac{1}{4}(1,1), \frac{1}{4}(1,1), A_3, A_3$  & 2     \\ \hline
    $M_{15}$ & $\frac{1}{4}$(1,1), $A_3, A_3$ & 2   \\  \hline
    $M_{16}$ & $\frac{1}{3}(1,1), \frac{1}{3}(1,1), A_2, A_2$ & $\frac{8}{3}$   \\ \hline
    $M_{17}$ &  $A_3, A_3$  & 2      \\ \hline
    $M_{18}$ & $A_1, A_1, A_1, A_1$ & 4  \\ \hline
    $M_{19}$ & - ($\PP$)& 8  \\ \hline
  \end{tabular}
\end{table}
}

\end{thm}

\begin{proof}

By Proposition \ref{candi_rank2} and Corollary \ref{bound_vol}, we see that candidates of $X$ is eight cases.
Thus by Lemma \ref{nonexi_8} and Lemma \ref{nonexi_9}, we obtain Table \ref{min_rank2_table}.

\end{proof}

\subsection{Minimal surfaces of rank one}

First of all, we will classify rank one surfaces. 

\subsubsection{Known results of rank one surfaces}

By using some known results, we can determine rank one del Pezzo surfaces of type $\tB$ when their index is two or they are Gorenstein.
Rank one del Pezzo surfaces of index two are already classified by Kojima \cite{Ko03}.

\begin{thm}\label{min_index2}\cite[Kojima]{Ko03}
If X is a rank one del Pezzo surface of index two and of type $\tB$, then X $\cong$ $\pr(1,1,4) $.

\end{thm}

Rank one Gorenstein del Pezzo surfaces are also already classified by Qiang \cite{Ye02} explicitly. 

\begin{thm}\label{min_Gor}\cite[Qiang]{Ye02}
If X is a rank one Gorenstein del Pezzo surface of type $\tB$, then X is one in the surfaces in Table \ref{min_Gor_table}.

{\small \rm
\begin{table}[htb]
\caption{Rank one Gorenstein del Pezzo surfaces of type $\tB$}\label{min_Gor_table}
\renewcommand{\arraystretch}{1.4}
  \begin{tabular}{|c|c|c|c|} \hline
    \multicolumn{1}{|l|}{No.} & \multicolumn{1}{c|}{$\sS(X)$}
      & \multicolumn{1}{c|}{$(-K_{X})^2$} & \multicolumn{1}{c|}{iso. class} \\ \hline \hline
    1 & $A_3, A_3, A_1, A_1$ &  1 & 1 \\
    2 & $A_2, A_2, A_2, A_2$ &  1 & 1 \\
    3 & $A_3, A_3, A_1$ &  2 & 1 \\
    4 & $A_2, A_2, A_2$ &  3 & 1 \\
    5 & $A_3, A_1, A_1$ &  4 & 1 \\
    6 & $A_2, A_1$ & 6 & 1 \\
    7 & $A_1$ & 8 & 1 \\  \hline
  \end{tabular}
\end{table}
}

\end{thm}

\begin{remark}\label{explicit}

We see that a surface of No.6 is $\pr(1,2,3)$ and a surface of No.7 is $\pr(1,1,2)$.
A surface of No.3 can be expressed by the following equation:
\[
xy(z^2 - xy) = w^2 \  {\rm in} \  \pr(1,1,1,2),
\]
where $\deg x,y,z = 1$ and $\deg w = 2$.
A surface of No.4 can be expressed by the following equation:
\[
xyz - w^3 = 0 \ {\rm in} \  \pr^3.
\] 
A surface of No.5 can be expressed by the following equation:
\[
\begin{cases}
\ xy - z^2 = 0 \\
\ zv - w^2 = 0 \ \ \ \ \ {\rm in} \  \pr^4.
\end{cases}
\]

\end{remark}

By Theorem \ref{min_index2} and \ref{min_Gor}, it remains to classify only the cases that surfaces having at least a singular point of type $\frac{1}{5}(1,2)$ or $\frac{1}{3}(1,1)$.
The upper bound of the number of singular points on a rank one del Pezzo surface is known as in the following theorem.

\begin{thm}\label{num_sing_pt_rank1}\cite[Theorem 1.1]{Be09}
A rank one del Pezzo surface with at most quotient singularities has at most four singular points. 

\end{thm}

\subsubsection{Lattice theory}
The following lemma is useful for eliminating impossible cases.

\begin{lemma}\cite[Lemma 3.3]{HK11}\label{lattice_square}
Let X be a rank one normal projective surface with quotient singularities and assume $K_X$ is not numerically trivial. 
Let $\pi : Y \to X $ be the minimal resolution. 
Then $H^2(Y, \mathbb{Z})_{free}$ is a unimodular lattice.
Let $R \subset H^2(Y, \mathbb{Z})_{free} := H^2(Y, \mathbb{Z} )/ ({\rm torsion \ part } )$ be a sublattice spanned by irreducible components of the exceptional divisors.
Then $\det (X) := | \det (R + \langle K_Y \rangle) |$ is a square number.

\end{lemma}

\begin{proof}

Let $E_1, \ldots ,E_r$ be the irreducible components of the exceptional divisors of $\pi$.
We may write
\[
K_Y = \pi^{*} K_X + \sum_{i=1}^{r} a_i E_i ,
\]
where $a_i \ge 0 $ for $1 \le i \le n$. 
 Let $v_1, \ldots , v_{r+1} \in H^2(Y, \mathbb{Z})_{free}$ be a basis as a lattice. 
Let ${\bf{v}} = (v_1 \cdots  v_{r+1})$ and $\bf{e}$ $= \left( K_Y \ E_1 \cdots E_r \right)$.
Since $H^2(Y, \mathbb{Z})_{free}$ is a unimodular, $\det ({}^t {\bf{v}} {\bf{v}}) =1$.
Since $ K_Y, E_1, \ldots , E_r$ are elements of $H^2(Y, \mathbb{Q} )$, there exists $A \in M_{r+1}(\mathbb{Q})$ such that ${\bf{e}} = {\bf{v}} A$.
Then we can compute $ \det (X)$ as follows:
\begin{eqnarray*}
   \det (X) &=&  \det ({}^{t}{\bf{e}} \bf{e})  \\
   &=& \det ({}^t ({\bf{v}} A) {\bf{v}} A) \\
   &=& \det ({}^t A {}^t {\bf{v}} {\bf{v}} A) \\
   &=& \det {}^t A \cdot \det ({}^t {\bf{v}} {\bf{v}}) \cdot \det A \\
   &=& (\det A)^2,
\end{eqnarray*} 
which is a square of an integer.

\end{proof}

\begin{remark}

Del Pezzo surfaces with at most quotient singularities are rational.
This is proved in \cite{Na07} for example.
Thus, for the minimal resolution $\pi : Y \to X$, $Y$ is a smooth rational surface. 

\end{remark}

\begin{remark}\label{Noeth_RR}

Let $Y$ be a smooth rational surface.
Then we see that 
\[
K_Y^2 + \rho(Y) = 10 .
\]
Let $X$ be a del Pezzo surface of type $\tA$.
Then we have
\[
K_X^2 + \rho (X) + \frac{2}{3} n_3 = 10
\]
by the Noether formula.
We also have
\[
h^0(X, -K_X ) = K_X^2 + 1 - \frac{n_3}{3}
\]
by the Riemann-Roch theorem.

\end{remark}

The following lemma tells us how to calculate $\det (X)$.

\begin{lemma}\label{lattice_cal}

Let the notation be as in the proof of Lemma \ref{lattice_square}.
It holds that
\[
\det (X) = \det (E_i \cdot E_j)_{ij} \cdot \left( 9 - r - \left( \sum a_i E_i  \right)^2 \right) .
\]

\end{lemma}

\begin{proof}
We can obtain this relation by the following relations:
\[
\begin{cases}
\ \det (X) = \det R \cdot K_X^2 = \det (E_i \cdot E_j)_{ij} \cdot K_X^2 ,\\
\ K_X^2 = K_Y^2 - \left( \sum a_i E_i  \right)^2 , \\
\ 1 + r = \rho (Y) , \\
\ K_Y^2 + \rho(Y) = 10 \ \ \  ({\rm Remark}  \ \ref{Noeth_RR}).
\end{cases}
\]
\end{proof}

\begin{proposition}\label{candi_1/3_1/5}

Let $X$ be a rank one del Pezzo surface of type $\tB$.
If $X$ has at least one singular point of type $\frac{1}{3}(1,1)$ or $\frac{1}{5}(1,2)$, then $X$ is one of the 19 cases in Table \ref{19candi}.

{\small \rm
\begin{table}[htb]
\caption{}\label{19candi}
\renewcommand{\arraystretch}{1.4}
  \begin{tabular}{|c|c|c|c|c|}  \hline
    \multicolumn{1}{|l|}{No.} & \multicolumn{1}{c|}{$\sS (X)$} & \multicolumn{1}{c|}{$(-K_{X})^2$}  & $ \det (X)$  \\ \hline \hline
    1 & $\frac{1}{5}$(1,2), $A_3, A_3, A_1$ & $\frac{2}{5}$ & 64   \\
    2 & $\frac{1}{5}$(1,2), $A_3, A_2$ &  $\frac{12}{5}$ & 144  \\ 
    3 & $\frac{1}{5}$(1,2), $A_2$  & $\frac{27}{5}$ & 81  \\  
    4 & $\frac{1}{5}$(1,2), $A_1$ & $\frac{32}{5}$ & 64 \\
    5 & $\frac{1}{3}$(1,1), $A_3$  & $\frac{16}{3}$ & 64  \\  
    6 & $\frac{1}{3}$(1,1)  & $\frac{25}{3}$ & 25  \\  
    7 & $\frac{1}{4}$(1,1), $\frac{1}{3}$(1,1) & $\frac{16}{3}$ & 100  \\
    8 & $\frac{1}{4}$(1,1), $\frac{1}{4}$(1,1), $\frac{1}{3}$(1,1) & $\frac{25}{3}$ & 400  \\
    9 & $\frac{1}{4}$(1,1), $\frac{1}{4}$(1,1), $\frac{1}{4}$(1,1), $\frac{1}{3}$(1,1) & $\frac{25}{3}$ & 1600  \\
    10 & $\frac{1}{3}$(1,1), $\frac{1}{3}$(1,1), $\frac{1}{3}$(1,1), $A_1$ & 6 & 324  \\
    11 & $\frac{1}{5}$(1,2), $\frac{1}{4}$(1,1), $\frac{1}{4}$(1,1), $A_1$  & $\frac{32}{5}$ & 1024   \\ 
    12 & $\frac{1}{5}$(1,2), $\frac{1}{3}$(1,1), $\frac{1}{3}$(1,1), $\frac{1}{3}$(1,1) & $\frac{27}{5}$ & 729   \\
    13 & $\frac{1}{5}$(1,2), $\frac{1}{4}$(1,1), $A_1$ & $\frac{32}{5}$ & 256   \\
    14 & $\frac{1}{3}$(1,1), $A_3, A_2, A_2$ & $\frac{4}{3}$ & 144  \\ 
    15 & $\frac{1}{5}$(1,2), $\frac{1}{4}$(1,1), $A_3, A_2$ &  $\frac{12}{5}$ & 576   \\ 
    16 & $\frac{1}{5}$(1,2), $\frac{1}{4}$(1,1), $\frac{1}{4}$(1,1), $A_2$ &  $\frac{27}{5}$ & 1296 \\
    17 & $\frac{1}{5}$(1,2), $\frac{1}{4}$(1,1), $A_2$ & $\frac{27}{5}$ & 324   \\
    18 & $\frac{1}{4}$(1,1), $A_3$, $\frac{1}{3}$(1,1)  &  $\frac{16}{3}$ & 256  \\  
    19 & $\frac{1}{4}$(1,1), $\frac{1}{4}$(1,1), $A_3$, $\frac{1}{3}$(1,1) & $\frac{16}{3}$ & 1024   \\ \hline
  \end{tabular}
\end{table}
}

\end{proposition}

\begin{proof}

By Theorem \ref{num_sing_pt_rank1}, the number of singular points on $X$ is at most four.
For all combinations of singularities, we calculate $\det (X)$ by using Lemma \ref{lattice_cal}.
We see that $\det (X)$ is a square number by Lemma \ref{lattice_square}.
The candidate whose $\det (X)$ is a square number is one of the 19 cases in Table \ref{19candi}.

\end{proof}

The surface of No.1 does exist.
By Corollary \ref{bound_vol}, we, however, see that we cannot obtain such a surface from a del Pezzo surface of type $\tA$.
We will prove non-existence of cases from No.7 to No.19.
We first prepare a lemma for the next subsection. 

\begin{lemma}\label{rat_pic_le_5}

Let $Y$ be a smooth rational surface.
Assume that $\rho(Y) \le 5$ and for a negative curve $C$ on $Y$, the inequality $-4 \le C^2 \le -1$ holds.
Denote the numbers of $(-3)$-curves and $(-4)$-curves on $Y$ by $N_3, N_4$ respectively.
Then it holds that $N_3 \le 2$ and $N_4 \le 1$.

\end{lemma}

\begin{proof}

It is enough to show only the case $\rho(Y) = 5$.
A rank two smooth rational surface is a Hirzebruch surface. 
Hence we have a sequence of blow-ups at points, $Y \overset{\tau_3}{\rightarrow} Y_2 \overset{\tau_2}{\rightarrow} Y_1 \overset{\tau_1}{\rightarrow} \mathbb{F}_m$.
By assumption, we see $m \le 4$.
Set $f := \tau_1 \circ \tau_2 \circ \tau_3$.
Denote the exceptional curves of $\tau_1, \tau_2, \tau_3$ by $E_1, E_2, E_3$ respectively.
Then we have $\Pic Y = \mathbb{Z} [(\tau_2 \circ \tau_3)^* E_1] \oplus \mathbb{Z} [\tau_3^* E_2] \oplus \mathbb{Z} [E_3] \oplus \mathbb{Z} [f^* \sigma] \oplus \mathbb{Z} [f^* l]$, where $\sigma$ is the minimal section and $l$ is a fiber.
Considering the configuration of negative curves on $Y_1$, we may assume $m=1, 3$ by choosing blow-downs.
Set $e_1 := (\tau_2 \circ \tau_3)^* E_1$, $e_2 := \tau_3^* E_2$, $e_3 := E_3$, $e_4 := f^* \sigma$ and $e_5 := f^* l$.
Then we see that $e_1^2 = e_2^2 = e_3^2 = -1, e_4^2 = -m, e_5^2 = 0$ and $e_4 \cdot e_5 = 1, e_i \cdot e_j = 0$ for the rest. 
Let $C \in \Pic Y$ be a $(-n)$-curve and we set $C \sim \sum^3_{i=1} a_i e_i + xe_4 + ye_5$ with some integer $a_i, x, y$.
Since $f_* C$ is effective, we have $x , y \ge 0$.
Since it holds that $-K_Y \cdot C = 2-n$ and $C^2 = -n$, we obtain the two equations by computing intersection numbers:
\begin{equation*}
( \bigstar ) \ \ 
\begin{cases}
\ A := a_1 + a_2 + a_3  =  (m-2) x -2 y + 2 - n \\
\ B := a_1^2 + a_2^2 + a_3^2  =  -mx^2 + 2xy + n \ .
\end{cases}
\end{equation*}
By the Cauchy-Schwarz inequality, we have $3B \ge A^2$, that is,
\[
3(-mx^2 + 2 xy + n) \ge ((m-2) x -2y + 2 - n)^2  .
\]
Hence we have
\begin{eqnarray*}
  0 & \ge & (m^2 -m +4) x^2 + 4y^2 - (4m -2)xy  \\
  & & - \ 2(m-2)(n-2) x + 4(n-2) y  + n^2 -7n + 4  .
\end{eqnarray*}
This method is used in \cite{Ha77}.

\noindent{\bf Case 1-1 : $m = 1$ and $n=3$} \\
If $C$ is a $(-3)$-curve, then by the inequality, we have 
\[
8 \ge 4x^2 + 4y^2 - 2xy + 2x + 4y  .
\]
Hence we have
\[
31 \ge 3 (x-y)^2 + (3x + 1)^2 + 6(y+1)^2 .
\]
Thus we see that $(x,y) = (0,1)$, $(0,0)$ or $(1,0)$.
If $(x,y) = (0,1)$, we have $A = -3$ and $B = 3$ by $(\bigstar)$.
Hence we have $(a_1, a_2, a_3) = (-1, -1, -1)$.
Set $C_1 := -e_1 - e_2 - e_3 + e_5 $.
If $(x,y) = (0,0)$, we have $A = -1$ and $B = 3$ by $(\bigstar)$.
Hence we have $(a_1, a_2, a_3) = (1, -1, -1)$, $(-1, 1, -1)$ or $(-1, -1, 1)$.
Set $C_2 := e_1 - e_2 - e_3$, $C_3 := -e_1 + e_2 - e_3$ and $C_4 := -e_1 - e_2 + e_3$.
If $(x,y) = (1,0)$, we have $A = -2$ and $B = 2$ by $(\bigstar)$.
Hence we have $(a_1, a_2, a_3) = (0, -1, -1)$, $(-1, 0, -1)$ or $(-1, -1, 0)$.
Set $C_5 := - e_2 - e_3 + e_4$, $C_6 := -e_1 - e_3 + e_4$ and $C_7 := -e_1 - e_2 + e_4$.

In summary, we have the seven candidates of $(-3)$-curve $C_1, \ldots , C_7$.
What we will prove is that $N_3 \le 2$.
The following Table \ref{1_3} is the intersection numbers between the candidates.

{\small
\begin{table}[htb]
\caption{Intersection numbers $C_i \cdot C_j$ when $(m,n) = (1,3)$}\label{1_3}
\renewcommand{\arraystretch}{1.4}
  \begin{tabular}{|l||c|c|c|c|c|c|c|} \hline
 & $C_1$ &  $C_2$ &  $C_3$ &  $C_4$ &  $C_5$ &  $C_6$ &  $C_7$   \\ \hline \hline
$C_1 := -e_1 - e_2 - e_3 + e_5 $ & -3 & & & & & & \\ \hline
$C_2 := e_1 - e_2 - e_3$ & -1 & -3 & & & & & \\ \hline
$C_3 := -e_1 + e_2 - e_3$ & -1 & 1 & -3 & & & & \\ \hline
$C_4 := -e_1 - e_2 + e_3$ & -1 & 1 & 1 & -3 & & & \\ \hline
$C_5 := - e_2 - e_3 + e_4$ & -1 & -2 & 0 & 0 & -3 & & \\ \hline
$C_6 := -e_1 - e_3 + e_4$ & -1 & 0 & -2 & 0 & -2 & -3 & \\ \hline
$C_7 := -e_1 - e_2 + e_4$ & -1 & 0 & 0 & -2 & -2 & -2 & -3 \\ \hline
  \end{tabular}
\end{table}
} 
Note that $C \cdot D \ge 0$ for distinct irreducible curves $C, D$.  

If $C_1$ is an irreducible curve, we see that the other $C_i$'s are not irreducible curves since $C_1 \cdot C_i < 0$ as in Table \ref{1_3}.
Therefore, $N_3 = 1 \le 2$ in this case.
From now on we may assume that $C_1$ is not an irreducible curve.

Since $C_2, C_3$ and $C_4$ are spanned by exceptional divisors, they are exceptional curves if they are irreducible curves.
The exceptional curve which can be a $(-3)$-curve is only the strict transform of $E_1$ by $\tau_2 \circ \tau_3$.
Hence we see that $C_3$ and $C_4$ cannot be $(-3)$-curves.
We may assume that $C_3$ and $C_4$ are not irreducible curves.
If $C_2$ is an irreducible curve, $C_5$ is not an irreducible curve since $C_2 \cdot C_5 = -2 < 0$.
Since $C_6 \cdot C_7 < 0$, at least one of them is not an irreducible curve.
Therefore, we see $N_3 \le 2$. 
Hence we may assume that $C_2, C_3$ and $C_4$ are also not irreducible curves. 

If one of $C_5, C_6$ and $C_7$ is an irreducible curve, the other cases are not irreducible curves.
Hence we see $N_3 = 1 \le 2$.

\noindent{\bf Case 1-2 : $m = 1$ and $n=4$}  \\
By the inequality, we have
\[
8 \ge 4x^2 + 4y^2 - 2xy + 4x + 8y  .
\]
Hence we have
\[
44 \ge 3 (x-y)^2 + (3x + 12)^2 + (3y+4)^2 .
\]
Thus we see that $(x,y) = (0, 0)$ or $(1,0)$.
If $(x,y) = (0,0)$, we have $A = -2$ and $B = 4$ by $(\bigstar)$.
Thus we have $(a_1, a_2, a_3) = (-2, 0, 0), (0, -2, 0)$ or $(0,0,-2)$.
We see, however, that neither $-2e_1$, $-2e_2$ nor $-2e_3$ cannot be realized as an irreducible curve.
Hence this case is impossible.
If $(x,y) = (1,0)$, we have $A = -3$ and $B = 3$ by $(\bigstar)$.
Thus we have $(a_1, a_2, a_3) = (-1, -1, -1)$.
Hence we see that there is at most one candidate.
Therefore, we see that $N_4 \le 1$.

\noindent{\bf Case 2-1 : $m = 3$ and $n=3$}  \\
By the inequality, we have
\begin{eqnarray*}
8 &\ge& 10x^2 + 4y^2 - 10xy - 2x + 4y  .
\end{eqnarray*}
Hence we have
\[
189 \ge 5 (3y - 5x)^2 + (5x - 3)^2 + 15(y+2)^2 .
\]
Thus we see that $(x,y) = (0,0)$, $(1,0)$, $(0,1)$ or $(1,1)$.
If $(x,y) = (0,0)$, we have $A = -1$ and $B = 3$ by $(\bigstar)$.
Thus we have $(a_1, a_2, a_3) = (1, -1, -1)$, $(-1, 1, -1)$ or $(-1, -1, 1)$.
Set $C_1 := e_1 - e_2 - e_3 $, $C_2 := -e_1 + e_2 - e_3$ and $C_3 := -e_1 - e_2 + e_3$.
If $(x,y) = (1,0)$, we have $A = 0$ and $B = 0$ by $(\bigstar)$.
Thus we have $(a_1, a_2, a_3) = (0, 0, 0)$.
Set $C_4 := e_4$.
If $(x,y) = (0,1)$, we have $A  = -3$ and $B = 3$ by $(\bigstar)$.
Thus we have $(a_1, a_2, a_3) = (-1,-1,-1)$.
Set $C_5 := - e_1 - e_2 - e_3 + e_5$.
If $(x,y) = (1,1)$, we have $A = -2$ and $B = 2$ by $(\bigstar)$.
Thus we have $(a_1, a_2, a_3) = (0,-1,-1), (-1,0,-1)$ or $(-1,-1,0)$.
Set $C_6 := -e_2 - e_3 + e_4 + e_5$, $C_7 := -e_1 - e_3 + e_4 + e_5$ and $C_8 := -e_1 - e_2 + e_4 + e_5$.

Hence we have the eight candidates of $(-3)$-curve $C_1, \ldots , C_8$.
The following Table \ref{3_3} is the intersection numbers between the candidates.

{\small
\begin{table}[htb]
\caption{Intersection numbers $C_i \cdot C_j$ when $(m,n) = (3,3)$}\label{3_3}
\renewcommand{\arraystretch}{1.4}
  \begin{tabular}{|l||c|c|c|c|c|c|c|c|} \hline
 & $C_1$ &  $C_2$ &  $C_3$ &  $C_4$ &  $C_5$ &  $C_6$ &  $C_7$ & $C_8$  \\ \hline \hline
$C_1 := e_1 - e_2 - e_3 $          & -3 & & & & & & & \\ \hline
$C_2 := -e_1 + e_2 - e_3$         & 1  & -3 & & & & & & \\ \hline
$C_3 := -e_1 - e_2 + e_3$         & 1  & 1  & -3  & & & & & \\ \hline
$C_4 := e_4 $                         & 0  & 0   & 0  & -3 & & & & \\ \hline
$C_5 := - e_1 - e_2 - e_3 + e_5$ & -1 & -1 & -1 & 1 & -3 & & & \\ \hline
$C_6 := -e_2 - e_3 + e_4 + e_5$  & -2 & 0  & 0   & -2 & -1 & -3 & & \\ \hline
$C_7 := -e_1 - e_3 + e_4 + e_5$  & 0  & -2 & 0   & -2 & -1 & -2 & -3 &  \\ \hline
$C_8 := -e_1 - e_2 + e_4 + e_5$  & 0  & 0 & -2   & -2 & -1 & -2 & -3 & -3  \\ \hline  
  \end{tabular}
\end{table}
}

Since $C_1, C_2$ and $C_3$ are spanned by exceptional divisors, they are exceptional curves if they are irreducible curves.
The exceptional curve which can be a $(-3)$-curve is only the strict transform of $E_1$ by $\tau_2 \circ \tau_3$.
Hence we see that $C_2$ and $C_3$ cannot be $(-3)$-curves.
We may assume that $C_2$ and $C_3$ are not irreducible curves.
If $C_1$ is an irreducible curve, then $C_5$ and $C_6$ are not irreducible curves.
Then if $C_4$ is also an irreducible curve, $C_7$ and $C_8$ are not irreducible curves.
Hence we see $N_3 \le 2$.
If $C_4$ is not an irreducible curve, we also have $N_3 \le 2$ since at least one of $C_7$ and $C_8$ is not an irreducible curve.
Hence we may assume that $C_1$ is not an irreducible curve.

If $C_4$ is an irreducible curve, $C_6, C_7$ and $C_8$ are not irreducible curves.
Hence we have $N_3 \le 2$.
We may assume that $C_4$ is also not an irreducible curve.

If $C_5$ is an irreducible curve, $C_6, C_7$ and $C_8$ are not irreducible curves.
Hence $N_3 = 1 \le 2$.
We may assume that $C_5$ is also not an irreducible curve.

If one of $C_6, C_7$ and $C_8$ is not an irreducible curve, then the other candidates are not irreducible curves.
Hence we see that $N_3 \le 2$.

\noindent{\bf Case 2-2 : $m = 3$ and $n=4$} \\
By the inequality, we have  
\begin{eqnarray*}
8 &\ge& 10x^2 + 4y^2 - 10xy - 4x + 8y  .
\end{eqnarray*}
Hence we have
\[
396 \ge 5 (3y - 5x)^2 + (5x - 6)^2 + 15(y+4)^2 .
\]
Thus we see that $(x,y) = (0,0)$, $(1,0)$ or $(1,1)$.
If $(x,y) = (0,0)$, we have $A = -2$ and $B = 4$ by $(\bigstar)$.
Thus we have $(a_1, a_2, a_3) = (-2,0,0), (0,-2,0)$ or $(0,0,-2)$.
We see, however, that neither $-2e_1$, $-2e_2$ nor $-2e_3$ cannot be realized as an irreducible curve.
Hence this case is impossible.
If $(x,y) = (1,0)$, we have $A = -1$ and $B = 1$ by $(\bigstar)$.
Thus we have $(a_1, a_2, a_3) = (-1,0,0), (0,-1,0)$ or $(0,0,-1)$.
Set $C_1 := -e_1 + e_4$, $C_2 := -e_2 + e_4$ and $C_3 := -e_3 + e_4$.
If $(x,y) = (1,1)$, we have $A = -3$ and $B = 3$ by $(\bigstar)$.
Thus we have $(a_1, a_2, a_3) = (-1,-1,-1)$.
Set $C_4 := -e_1 - e_2 - e_3 + e_4 + e_5$.

Hence we have the eight candidates of $(-4)$-curve $C_1, \ldots , C_4$.
The following Table \ref{3_4} is the intersection numbers between the candidates.

{\small
\begin{table}[htb]
\caption{Intersection numbers $C_i \cdot C_j$ when $(m,n) = (3,4)$}\label{3_4}
\renewcommand{\arraystretch}{1.4}
  \begin{tabular}{|l||c|c|c|c|} \hline
 & $C_1$ &  $C_2$ &  $C_3$ &  $C_4$   \\ \hline \hline
$C_1 := -e_1 + e_4$ & -4 & & &  \\ \hline
$C_2 := -e_2 + e_4$ & -3 & -4 & &  \\ \hline
$C_3 := -e_3 + e_4$ & -3 & -3 & -4 &  \\ \hline
$C_4 := -e_1 - e_2 - e_3 + e_4 + e_5$ & -3 & -3 & -3 & -4  \\ \hline
  \end{tabular}
\end{table}
}

Since $C_i \cdot C_j < 0$ for all $i \neq j$, we see that $N_4 \le 1$.
 
\end{proof}

\begin{proposition}\label{non_exi_7_10}

There are no examples of the case $No.7$, $No.8$, $No.9$ nor $No.10$.

\end{proposition}

\begin{proof}

Let $X$ be a del Pezzo surface of type $\tB$.
Let $\pi : Y \to X$ be the minimal resolution.
Then $Y$ is a smooth rational surface.
Assume that $X$ is one of the surfaces of No.8, No.9 and No.10. 
We see that $\rho (Y) \le 5$.
For each case, we see that the number of $(-3)$-curves or the number of $(-4)$-curves contradicts Lemma \ref{rat_pic_le_5}.
Therefore, there are no examples of the case No.8, No.9 nor No.10.
Next, assume that $X$ is of No.7.
Then $Y$ has one $(-3)$-curve and one $(-4)$-curve.
Let $\tau : Y' \to Y$ be the blow-up at a general point on the $(-3)$-curve.
$Y'$ has two $(-4)$-curves and $\rho (Y' ) = 4$.
This contradicts Lemma \ref{rat_pic_le_5}.
Therefore, there is no example of No.7.
\end{proof}

\subsubsection{Two ray games}\label{2_ray}

To eliminate the possibilities of from No.11 to No.19, we play two ray games in this subsection.

\begin{definition}

Let $X$ be a normal projective surface. 
Let $\varphi : Y \to X$ be a contraction of an irreducible curve $C$ on $Y$ and assume that $C$ passes through only one singular point $P$.
If $\varphi$ contracts $C$ to a singular point $x \in X$, then $\varphi$ is called an {\it extraction of $x$}.

\end{definition}

\begin{remark}\label{extractions}

In general, there are several possibilities of extractions for a singular point.

In this paper, the {\it extraction of $\frac{1}{5}(1,2)$} means $\varphi : Y \to X$ where $x \in X$ is a singular point of type $\frac{1}{5}(1,2)$ and $P$ is a singular point of type $A_1$.
Then we have $C^2 = -\frac{5}{2}$ and $K_Y \cdot C = 1$.
The {\it extraction of $A_3$} means $\varphi : Y \to X$ where $x \in X$ is a singular point of type $A_3$ and $P$ is a singular point of type $A_2$.
Then we have $C^2 = -\frac{4}{3}$ and $K_Y \cdot C = 0$.
The {\it extraction of $A_2$} means $\varphi : Y \to X$ where $x \in X$ is a singular point of type $A_2$ and $P$ is a singular point of type $A_1$.
Then we have $C^2 = -\frac{3}{2}$ and $K_Y \cdot C = 0$.
In particular, the extraction of $A_1$ is the minimal resolution of a singular point of type $A_1$.

\end{remark}

\begin{lemma}\label{two_ray_game}

Let $X$ be a rank one del Pezzo surface with at least one singular point $P$ of type $\frac{1}{5}(1,2)$ (resp. $A_3$, $A_2$).
Let $\varphi : Y \to X$ be the extraction of $\frac{1}{5}(1,2)$ (resp. $A_3$, $A_2$).
Then there exists a $K_Y$-negative extremal contraction, which we call $\psi\colon Y\to Z$.
Moreover, if $\dim Z = 2$, $Z$ is a del Pezzo surface.

\end{lemma}

\begin{proof}

There exists a curve $D$ such that it does not pass through $P$.
Since $D$ does not pass through $P$, we have $\varphi^* D = D_Y$.
Hence $-K_Y \cdot D_Y = \varphi_{*}(-K_Y) \cdot D = -K_X \cdot D > 0$.
This means $K_Y$ is not nef. 
Therefore, there is a $K_Y$-negative extremal contraction and we denote it by $\psi : Y \to Z$. 
Since $\psi$ is $K_Y$-negative, we see $Z \neq X$.

Assume that $\dim Z = 2$.
Denote the $\psi$-exceptional curve by $E$.
Since $K_Y \cdot D_Y > 0$ and $K_Y \cdot E < 0$, $D_Y$ and $E$ are distinct curves.
Thus we see that $D_Y \cdot E \ge 0$.
Then we have
\[
K_Y = \psi^{*} K_Z + aE ,
\]
where $a \ge 0$.
Since $\rho (Z) = 1$, we see that $-K_Z$ is ample or $K_Z$ is nef.
We have $-K_Z \cdot \psi_{*} D_Y =  \psi^{*}(-K_Z) \cdot D_Y = -K_Y \cdot D_Y + aE \cdot D_Y > 0$.
Therefore, $Z$ is a del Pezzo surface.

\end{proof}

\begin{proposition}\label{non_exi_11_16}

There are no examples of the cases $No.11$, $No.12$, $No.13$, $No.14$, $No.15$ nor $No.16$.

\end{proposition}

\begin{proof}

Let $X$ be a del Pezzo surface.

Assume that $X$ is of No.11. 
Let $\varphi : Y \to X$ be the extraction of $\frac{1}{5}(1,2)$.
By Lemma \ref{two_ray_game}, there is a $K_Y$-negative contraction $\psi : Y \to Z$.
We see that $\sS (Y) = \{ \frac{1}{4}(1,1), \frac{1}{4}(1,1), A_1, A_1 \} $.
By Lemma \ref{non_del}, candidates of $\psi : Y \to Z$ are the following.

{\small
\begin{table}[htb]
\renewcommand{\arraystretch}{1.4}
  \begin{tabular}{|c|c|c|c|c|c|} \hline
    \multicolumn{1}{|l|}{No.} & $\psi$ & \multicolumn{1}{c|}{From}
      & \multicolumn{1}{c|}{To} & \multicolumn{1}{c|}{$\sS (Z)$} & $ \det (Z) $ \\ \hline \hline
    11-1 & \ $\sB_0$ \ & \hspace{7mm} - \hspace{7mm} & \ \ sm pt \ \ & $\frac{1}{4}(1,1)$, $\frac{1}{4}(1,1)$, $A_1, A_1$ & 448 \\ 
  \end{tabular}
\end{table}
}
{\small
\begin{table}[htb]
\renewcommand{\arraystretch}{1.4}
  \begin{tabular}{|c|c|c|c|c|c|} 
    \multicolumn{1}{|l|}{No.} & $\psi$ & \multicolumn{1}{c|}{From}
      & \multicolumn{1}{c|}{To} & \multicolumn{1}{c|}{$\sS (Z)$} & $ \det (Z) $ \\ \hline \hline
    11-2 & $\sB_1$ & $A_1$ & sm pt &  $\frac{1}{4}(1,1)$, $\frac{1}{4}(1,1)$, $A_1$ & 256  \\
    11-3 & $\sB_5$ & $\frac{1}{4}(1,1)$ & $\frac{1}{3}(1,1)$ & $\frac{1}{4}(1,1)$, $\frac{1}{3}(1,1)$, $A_1, A_1$ & 304  \\
    11-4 & $\sB_{12}$ & $A_1$, $\frac{1}{4}(1,1)$ & $A_1$ & $\frac{1}{4}(1,1)$, $A_1, A_1$ & 112  \\
    11-5 & $\sC_3$ & $A_1, A_1$ & pt on $\pr^1$ &  &  \\  \hline 
  \end{tabular}
\end{table}
}
The cases 11-1, 11-3 and 11-4 are eliminated by Lemma \ref{lattice_square}, since their values of $\det (Z)$ are not square numbers. 
In the case 11-2, $Z$ is a del Pezzo surface of index two.
By Theorem \ref{min_index2}, we see that this case is impossible.
In the case 11-5, the rest two singular points of type $\frac{1}{4}(1,1)$ must be contracted to a point of $\pr^1$. 
By Table \ref{fibers}, however, it is impossible.
Thus we see that the all cases are impossible.
Therefore, we see that there is no del Pezzo surface of No.11.

Next, assume that $X$ is of No.12. 
Let $\varphi : Y \to X$ be the extraction of $\frac{1}{5}(1,2)$.
By Lemma \ref{two_ray_game}, there is a $K_Y$-negative contraction $\psi : Y \to Z$. 
We see that $\sS (Y) = \{ \frac{1}{3}(1,1), \frac{1}{3}(1,1), \frac{1}{3}(1,1), A_1 \} $.
By Lemma \ref{non_del}, candidates of $\psi : Y \to Z$ are following.

{\small
\begin{table}[htb]
\renewcommand{\arraystretch}{1.4}
  \begin{tabular}{|c|c|c|c|c|c|} \hline
    \multicolumn{1}{|l|}{No.} & $\psi$  & \multicolumn{1}{c|}{From}
      & \multicolumn{1}{c|}{To} & \multicolumn{1}{c|}{$\sS (Z)$} & $ \det (Z) $ \\ \hline \hline
    12-1 & $\sB_0$ &  - & sm pt & $\frac{1}{3}(1,1)$, $\frac{1}{3}(1,1)$, $\frac{1}{3}(1,1)$, $A_1$ & 324 \\ 
    12-2 & $\sB_1$ & $A_1$ & sm pt &  $\frac{1}{3}(1,1)$, $\frac{1}{3}(1,1)$, $\frac{1}{3}(1,1)$ & 189  \\
    12-3 & $\sB_4$ & $\frac{1}{3}(1,1)$ & $A_1$ & $\frac{1}{3}(1,1)$, $\frac{1}{3}(1,1)$, $A_1, A_1$ & 204  \\
    12-4 & $\sB_8$ & $A_1$, $\frac{1}{3}(1,1)$ & sm pt & $\frac{1}{3}(1,1)$, $A_1$ & 69  \\
    12-5 & $\sB_{13}$ & $\frac{1}{3}(1,1)$, $\frac{1}{3}(1,1)$ & $A_2$ & $\frac{1}{3}(1,1)$, $A_2, A_1$ & 96 \\  \hline 
  \end{tabular}
\end{table}
}

The cases 12-2, 12-3, 12-4 and 12-5 are eliminated by Lemma \ref{lattice_square}. 
In the case 12-1, $Z$ is of No.10.
Hence it is impossible by Proposition \ref{non_exi_7_10}.
Thus we see that the all cases are impossible.
Therefore, we see that there is no del Pezzo surface of No.12.

Assume that $X$ is of No.13. 
Let $\varphi : Y \to X$ be the extraction of $\frac{1}{5}(1,2)$.
By Lemma \ref{two_ray_game}, there is a $K_Y$-negative contraction $\psi : Y \to Z$. 
We see that $\sS (Y) = \{ \frac{1}{4}(1,1), A_1, A_1 \} $.
By Lemma \ref{non_del}, candidates of $\psi : Y \to Z$ are following.

{\small
\begin{table}[htb]
\renewcommand{\arraystretch}{1.4}
  \begin{tabular}{|c|c|c|c|c|c|} \hline
    \multicolumn{1}{|l|}{No.} & $\psi$ & \multicolumn{1}{c|}{From}
      & \multicolumn{1}{c|}{To} & \multicolumn{1}{c|}{$\sS (Z)$} & $ \det (Z) $ \\ \hline \hline
    13-1 & \ $\sB_0$ \ & \hspace{7mm} - \hspace{7mm} & sm pt & $\frac{1}{4}(1,1)$, $A_1, A_1$ & 112 \\ 
  \end{tabular}
\end{table}
}

{\small
\begin{table}[htb]
\renewcommand{\arraystretch}{1.4}
  \begin{tabular}{|c|c|c|c|c|c|} 
    \multicolumn{1}{|l|}{No.} & $\psi$ & \multicolumn{1}{c|}{From}
      & \multicolumn{1}{c|}{To} & \multicolumn{1}{c|}{$\sS (Z)$} & $ \det (Z) $ \\ \hline \hline
    13-2 & $\sB_1$ & $A_1$ & sm pt &  $\frac{1}{4}(1,1)$, $A_1$ & 64  \\
    13-3 & $\sB_5$ & $\frac{1}{4}(1,1)$ & $\frac{1}{3}(1,1)$ & $\frac{1}{3}(1,1)$, $A_1, A_1$ & 76  \\  
    13-4 & $\sB_{10}$ & $A_1$, $\frac{1}{4}(1,1)$ & $A_1$ & $A_1, A_1$ & 28  \\
    13-5 & $\sC_3$ & $A_1, A_1$ & pt on $\pr^1$ &  &  \\ \hline 
  \end{tabular}
\end{table}
}

The cases 13-1, 13-3 and 13-4 are eliminated by Lemma \ref{lattice_square}. 
In the 13-5, the rest singular point $\frac{1}{4}(1,1)$ must be contracted to a point of $\pr^1$. 
By Table \ref{fibers}, it is impossible.
In the case 13-2, $Z$ is a del Pezzo surface of index two.
By Theorem \ref{min_index2}, we see that this case is impossible.
Thus we see that the all cases are impossible.
Therefore, we see that there is no del Pezzo surface of No.13.

Assume that $X$ is of No.14. 
Let $\varphi : Y \to X$ be the extraction of $A_3$.
By Lemma \ref{two_ray_game}, there is a $K_Y$-negative contraction $\psi : Y \to Z$. 
We see that $\sS (Y) = \{ \frac{1}{3}(1,1), A_2, A_2, A_2 \} $.
By Lemma \ref{non_del}, candidates of $\psi : Y \to Z$ are following.

{\small
\begin{table}[htb]
\renewcommand{\arraystretch}{1.4}
  \begin{tabular}{|c|c|c|c|c|c|} \hline
    \multicolumn{1}{|l|}{No.} & $\psi$ & \multicolumn{1}{c|}{From}
      & \multicolumn{1}{c|}{To} & \multicolumn{1}{c|}{$\sS (Z)$} & $ \det (Z) $ \\ \hline \hline
    14-1 & $\sB_0$ & - & sm pt & $\frac{1}{3}(1,1)$, $A_2, A_2, A_2$ & 189 \\ 
    14-2 & $\sB_2$ & $A_2$ & sm pt &  $\frac{1}{3}(1,1)$, $A_2, A_2$ & 117  \\
    14-3 & $\sB_4$ & $\frac{1}{3}(1,1)$ & $A_1$ & $A_2, A_2, A_2, A_1$ & 108  \\
    14-4 & $\sC_4$ & $\frac{1}{3}(1,1)$, $A_2$ & pt on $\pr^1$ &  &  \\  \hline 
  \end{tabular}
\end{table}
}

The cases 14-1, 14-2 and 14-3 are eliminated by Lemma \ref{lattice_square}. 
In the case 14-4, the rest two singular points (both are of type $A_2$) must be contracted to a point of $\pr^1$.
By Table \ref{fibers}, it is impossible.
Thus we see that the all cases are impossible.
Therefore, we see that there is no del Pezzo surface of No.14.

Assume that $X$ is of No.15. 
Let $\varphi : Y \to X$ be the extraction of $\frac{1}{5}(1,2)$.
By Lemma \ref{two_ray_game}, there is a $K_Y$-negative contraction $\psi : Y \to Z$. 
We see that $\sS (Y) = \{ \frac{1}{4}(1,1), A_3, A_2, A_1 \} $.
By Lemma \ref{non_del}, candidates of $\psi : Y \to Z$ are following.

{\small
\begin{table}[htb]
\renewcommand{\arraystretch}{1.4}
  \begin{tabular}{|c|c|c|c|c|c|} \hline
    \multicolumn{1}{|l|}{No.} & $\psi$ & \multicolumn{1}{c|}{From}
      & \multicolumn{1}{c|}{To} & \multicolumn{1}{c|}{$\sS (Z)$} & $ \det (Z) $ \\ \hline \hline
    15-1 & $\sB_0$ & \hspace{7mm}  - \hspace{7mm}  & \ \ sm pt \ \ & $\frac{1}{4}(1,1)$, $A_3, A_2, A_1$ & 288 \\ 
    15-2 & $\sB_1$ & $A_1$ & sm pt & $\frac{1}{4}(1,1)$, $A_3, A_2$ & 192 \\ 
    15-3 & \ $\sB_2$ \ & $A_2$ & sm pt & $\frac{1}{4}(1,1)$, $A_3, A_1$ & 160 \\
  \end{tabular}
\end{table}
}
{\small
\begin{table}[htb]
\renewcommand{\arraystretch}{1.4}
  \begin{tabular}{|c|c|c|c|c|c|}
    \multicolumn{1}{|l|}{No.} & $\psi$ & \multicolumn{1}{c|}{From}
      & \multicolumn{1}{c|}{To} & \multicolumn{1}{c|}{$\sS (Z)$} & $ \det (Z) $ \\ \hline \hline
    15-4 & $\sB_3$ & $A_3$ &  sm pt   & $\frac{1}{4}(1,1)$, $A_2, A_1$ & 144 \\
    15-5 & $\sB_5$ & $\frac{1}{4}(1,1)$ & $\frac{1}{3}(1,1)$ & $\frac{1}{3}(1,1)$, $A_3, A_2, A_1$ & 168 \\
    15-6 & $\sB_{10}$ & $\frac{1}{4}(1,1)$, $A_2$ & sm pt & $A_3, A_1$ & 40 \\
    15-7 & $\sB_{12}$ & $\frac{1}{4}(1,1)$, $A_1$ & $A_1$ & $A_3, A_2, A_1$ & 72  \\
    15-8 & $\sC_5$ & $\frac{1}{4}(1,1)$, $A_3$ & pt on $\pr^1$ &  &  \\   \hline 
  \end{tabular}
\end{table}
}

The cases 15-1, 15-2, 15-3, 15-5, 15-6 and 15-7 are eliminated by Lemma \ref{lattice_square}. 
In the case 15-4, $Z$ is a del Pezzo surface of index two.
By Theorem \ref{min_index2}, we see that this case is impossible.
In the case 15-8, the rest one singular point of type $A_1$ and one singular point of type $A_2$ must be contracted to a point of $\pr^1$.
By Table \ref{fibers}, it is impossible.
Thus we see that the all cases are impossible.
Therefore, we see that there is no del Pezzo surface of No.15.

Assume that $X$ be is of No.16. 
Let $\varphi : Y \to X$ be the extraction of $\frac{1}{5}(1,2)$.
By Lemma \ref{two_ray_game}, there is a $K_Y$-negative contraction $\psi : Y \to Z$. 
We see that $\sS (Y) = \{ \frac{1}{4}(1,1), \frac{1}{4}(1,1), A_2, A_1 \} $.
By Lemma \ref{non_del}, candidates of $\psi : Y \to Z$ are following.

{\small
\begin{table}[htb]
\renewcommand{\arraystretch}{1.4}
  \begin{tabular}{|c|c|c|c|c|c|} \hline
    \multicolumn{1}{|l|}{No.} & $\psi$ & \multicolumn{1}{c|}{From}
      & \multicolumn{1}{c|}{To} & \multicolumn{1}{c|}{$\sS (Z)$} & $ \det (Z)$ \\ \hline \hline
    16-1 & $\sB_0$ & - & sm pt & $\frac{1}{4}(1,1)$, $\frac{1}{4}(1,1)$, $A_2, A_1$ & 576 \\ 
    16-2 & $\sB_1$ & $A_1$ & sm pt & $\frac{1}{4}(1,1)$, $\frac{1}{4}(1,1)$, $A_2$ & 336 \\
    16-3 & $\sB_2$ & $A_2$ & sm pt & $\frac{1}{4}(1,1)$, $\frac{1}{4}(1,1)$, $A_1$ & 256 \\
    16-4 & $\sB_5$ & $\frac{1}{4}(1,1)$ & $\frac{1}{3}(1,1)$ & $\frac{1}{4}(1,1)$, $\frac{1}{3}(1,1)$, $A_2, A_1$ & 384 \\ 
    16-5 & $\sB_{10}$ & $\frac{1}{4}(1,1)$, $A_2$ & sm pt & $\frac{1}{4}(1,1)$, $A_1$  & 64 \\ 
    16-6 & $\sB_{12}$ & $\frac{1}{4}(1,1)$, $A_1$ & $A_1$ &  $\frac{1}{4}(1,1)$, $A_2, A_1$ & 144  \\    \hline 
  \end{tabular}
\end{table}
}

The cases 16-2 and 16-4 are eliminated by Lemma \ref{lattice_square}. 
The case 16-3 is also eliminated by Lemma \ref{rat_pic_le_5}.
In the cases 16-1, 16-5 and 16-6, $Z$ is a del Pezzo surface of index two.
By Theorem \ref{min_index2}, we see that these cases are impossible.
Thus we see that the all cases are impossible.
Therefore, we see that there is no del Pezzo surface of No.16.

\end{proof}

\begin{lemma}\label{unique_1/n}

If $X$ is a rank one del Pezzo surface with only one singular point of type $\frac{1}{n}(1,1)$, then $X \cong \pr(1,1,n)$.

\end{lemma}

\begin{proof}

Let $\alpha : Y \to X$ be the minimal resolution.
Then $Y$ is a smooth rational surface and $\rho(Y) = 2$. 
Hence we have $Y \cong \mathbb{F}_n$.

\end{proof}

\begin{lemma}\label{unique_1/5_A2}

If X is a rank one del Pezzo surface with only one singular point of type $\frac{1}{5}(1,2)$ and one $A_2$ singular point, then $X \cong \pr(1,3,5)$.

\end{lemma}

\begin{proof}

Let $\varphi : Y \to X$ be the extraction of $\frac{1}{5}(1,2)$.
Then there is a $K_Y$-negative contraction $\psi : Y \to Z$ by Lemma \ref{two_ray_game}.
Since $\sS(Y) = \{ A_1, A_2 \}$, the candidates of types of $\psi$ is of type $\sB_0$, $\sB_1$ or $\sB_2$.
We will prove that $\psi$ is of type $\sB_2$.
If $\psi$ is of type $\sB_1$, then $Z$ is a del Pezzo surface with only one singular point of type $A_2$.
Such a surface, however, does not exist since $\det (Z) = 21$.
If $\psi$ is of type $\sB_0$, then $Z \cong \pr(1,2,3)$.
Since $\psi$ is the blow-up at a smooth point and $-K_Z$ is very ample, $-K_Y$ is nef and big.
We have, however, $-K_Y \cdot C = -1$ by Remark \ref{extractions}.
This is a contradiction.

Thus we see $\psi$ is of type $\sB_2$ and $Z \cong \pr(1,1,2)$.
Moreover, we also see that $\psi$ is constructed by blow-ups three times on the strict transform of a ruling line.
Therefore, $\psi$ depends only on a choice of a smooth point.
We see that $C$ on $Y$ is the strict transform of the ruling line on $\pr(1,1,2)$.
Since smooth points on $\pr(1,1,2)$ are transitive by the action of Aut$\pr(1,1,2)$, we see that $Y$ is unique.
Hence we see that how to construct del Pezzo surfaces with only one singular point of type $\frac{1}{5}(1,2)$ and one singular point of type $A_2$ is unique.
$\pr(1,3,5)$ can be also obtained in the same way.
Thus we see that $X \cong \pr(1,3,5)$.

\end{proof}

\begin{proposition}\label{non_exi_17_18}

There are no examples of the case $No.17$ nor $No.18$.

\end{proposition}

\begin{proof}

Let $X$ be a del Pezzo surface.
Assume that $X$ is of No.17. 
Let $\varphi : Y \to X$ be the extraction of $\frac{1}{5}(1,2)$.
By Lemma \ref{two_ray_game}, there is a $K_{Y}$-negative contraction $\psi : Y \to Z$. 
We see that $\sS (Y) = \{ \frac{1}{4}(1,1), A_2, A_1 \}$.
By Lemma \ref{non_del}, candidates of $\psi : Y \to Z$ are following. 

{\small
\begin{table}[htb]
\renewcommand{\arraystretch}{1.4}
  \begin{tabular}{|c|c|c|c|c|c|} \hline
    \multicolumn{1}{|l|}{No.} & $\psi$ & \multicolumn{1}{c|}{From}
      & \multicolumn{1}{c|}{To} & \multicolumn{1}{c|}{$\sS (Z)$} & $ \det (Z) $ \\ \hline \hline
    17-1 & $\sB_0$ & - & sm pt & $\frac{1}{4}(1,1)$, $A_2, A_1$ & 144 \\ 
    17-2 & $\sB_1$ & $A_1$ & sm pt & $\frac{1}{4}(1,1)$, $A_2$ & 84 \\
    17-3 & $\sB_2$ & $A_2$ & sm pt & $\frac{1}{4}(1,1)$, $A_1$ & 64 \\ 
    17-4 & $\sB_5$ & $\frac{1}{4}(1,1)$ & $\frac{1}{3}(1,1)$ & $\frac{1}{3}(1,1)$, $A_2, A_1$ & 96 \\ 
    17-5 & $\sB_{10}$ & $\frac{1}{4}(1,1)$, $A_2$ & sm pt & $A_1$  & 16 \\ 
    17-6 & $\sB_{12}$ & $\frac{1}{4}(1,1)$, $A_1$ & $A_1$ & $A_2, A_1$ & 36  \\ \hline 
  \end{tabular}
\end{table}
}

The cases 17-2 and 17-4 are eliminated by Lemma \ref{lattice_square}. 
In the cases 17-1 and 17-3, $Z$ is a del Pezzo surface of index two.
By Theorem \ref{min_index2}, we see these cases are impossible.
The rest two cases, 17-6 and 17-5, need more considerations.

Assume that $Z$ is of the case 17-6.
We see that $\sS (Z) = \{ A_2, A_1 \}$.
By Theorem \ref{min_Gor}, we have $Z \cong \pr(1,2,3)$.
Let $L := \mathcal{O}_Z(1)$ and $C_Z := \psi_{*} C$, where $C$ is the exceptional curve of $\varphi$.
Then we see that $-K_Z \sim 6L$ and there is an integer $n \in \mathbb{Z}_{\ge 0}$ such that $C_Z \sim nL$.
Denote the exceptional curve of $\psi$ by $E$.
Then we have $K_Y = \varphi^{*} K_X - \frac{2}{5} C$ and $K_Y = \psi^{*} K_Z + 2E$ .
By using these relations, we obtain $E^2 = -\frac{1}{4}$ and $K_Y \cdot E = -\frac{1}{2}$.
We can set $\psi^{*} C_Z = C +\alpha E$.
Then we have
\[
-K_Y \cdot \psi^{*} C_Z =-K_Y \cdot C + (-K_Y) \cdot \alpha E  \ .
\]
Since $-K_Y \cdot \psi^{*} C_Z = -K_Z \cdot C_Z = n$ and $-K_Y \cdot C = -1$, we obtain $\alpha = 2(n+1)$.
Thus we have $\psi^* C_Z - 2(n+1)E = C$.
Since $( \psi^{*} C_Z -2(n+1)E )^2 = \frac{1}{6}n^2 + 4(n+1)^2$ and $C^2 = -\frac{5}{2}$, we have $5n^2 + 12n -9 = 0$.
Thus we see that $n = \frac{3}{5}$ or $-3$. 
This is a contradiction.

Assume that $Z$ is of the case 17-5.
Then we see that $Z \cong \pr(1,1,2)$ by Theorem \ref{min_Gor} (or Lemma \ref{unique_1/n}).
Let $L := \mathcal{O}_Z(1)$ and $C_Z := \psi_{*} C \sim nL$, where $C$ is the exceptional curve of $\varphi$.
Then we see that $-K_Z \sim 4L$ and there is an integer $n \in \mathbb{Z}_{\ge 0}$ such that $C_Z \sim nL$.
Denote the exceptional curve of $\psi$ by $E$.
Then we have $K_Y = \varphi^{*} K_X - \frac{2}{5} C$ and $K_Y = \psi^{*} K_Z + 6E$.
By using these relations, we  obtain $E^2 = -\frac{1}{12}$ and $K_Y \cdot E = -\frac{1}{2}$.
We can set $\psi^{*} C_Z = C +\alpha E$.
Then we have
\[
-K_Y \cdot \psi^{*} C_Z = -K_Y \cdot C + (-K_Y) \cdot \alpha E .
\]
Since $-K_Y \cdot \psi^{*} C_Z = -K_Z \cdot C_Z = 2n$ and $-K_Y \cdot C = -1$, we obtain $\alpha = 2(2n+1)$.
Thus we have $\psi^* C_Z -2(2n + 1)E = C$.
Since $(\psi^* C_Z -2(2n + 1)E)^2 = \frac{1}{4}n^2 + 4(2n +1)^2$ and $C^2 = -\frac{5}{2}$, we have $(5n-13)(n+1) = 0$.
Thus we see that $n = \frac{13}{5}$ or $-1$. 
This is a contradiction.
Thus we see that the all cases are impossible.
Therefore, we see that there is no del Pezzo surface of No.17. 

Next, assume that $X$ is of No.18. 
Let $\varphi : Y \to X$ be the extraction of $A_3$.
By Lemma \ref{two_ray_game}, there is a $K_{Y}$-negative contraction $\psi : Y \to Z$. 
We see that $\sS (Y) = \{ \frac{1}{4}(1,1), \frac{1}{3}(1,1), A_2 \}$.
By Lemma \ref{non_del}, candidates of $\psi : Y \to Z$ are following.
{\small
\begin{table}[htb]
\renewcommand{\arraystretch}{1.4}
  \begin{tabular}{|c|c|c|c|c|c|} \hline
    \multicolumn{1}{|l|}{No.} & $\psi$ & \multicolumn{1}{c|}{From}
      & \multicolumn{1}{c|}{To} & \multicolumn{1}{c|}{$\sS (Z)$} & $ \det (Z) $ \\ \hline \hline
    18-1 & \ $\sB_0$ \ &\hspace{9mm}  - \hspace{9mm}  &  \ \ sm pt \ \ & $\frac{1}{4}(1,1)$, $\frac{1}{3}(1,1)$, $A_2$ & 228 \\ 
    18-2 & $\sB_2$ & $A_2$ & sm pt & $\frac{1}{4}(1,1)$, $\frac{1}{3}(1,1)$ & 100 \\
    18-3 & $\sB_4$ & $\frac{1}{3}(1,1)$ & $A_1$ & $\frac{1}{4}(1,1)$, $A_2, A_1$ & 144 \\
  \end{tabular}
\end{table}
}
{\small
\begin{table}[htb]
\renewcommand{\arraystretch}{1.4}
  \begin{tabular}{|c|c|c|c|c|c|} 
    \multicolumn{1}{|l|}{No.} & $\psi$ & \multicolumn{1}{c|}{From}
      & \multicolumn{1}{c|}{To} & \multicolumn{1}{c|}{$\sS (Z)$} & $ \det (Z) $ \\ \hline \hline
    18-4 & $\sB_5$ & $\frac{1}{4}(1,1)$ & $\frac{1}{3}(1,1)$ & $\frac{1}{3}(1,1)$, $\frac{1}{3}(1,1)$, $A_2$ & 153 \\  
    18-5 & $\sB_{10}$ & $\frac{1}{4}(1,1)$, $A_2$ & sm pt & $\frac{1}{3}(1,1)$ & 25  \\
    18-6 & $\sB_{14}$ & $\frac{1}{4}(1,1)$, $\frac{1}{3}(1,1)$ & $\frac{1}{5}(1,2)$ & $\frac{1}{5}(1,2)$, $A_2$ & 81 \\
    18-7 & $\sC_4$ & $\frac{1}{3}(1,1)$, $A_2$ & pt on $\pr^1$ &   &  \\  \hline 
  \end{tabular}
\end{table}
}

The cases 18-1 and 18-4 are eliminated by Lemma \ref{lattice_square}. 
In the case 18-7, the rest singular point of type $\frac{1}{4}(1,1)$ must be contracted to a point of $\pr^1$. 
By Table \ref{fibers}, it is impossible.
In the case 18-3, $Z$ is a del Pezzo surface of index two.
By Theorem \ref{min_index2}, we see that this case is impossible.
The case 18-2 is the same as the case No.6.
Hence this case is impossible by Proposition \ref{non_exi_7_10}.
The rest two cases, 18-5 and 18-6, need more considerations.

Assume that $Z$ is of the case 18-5.
Then we see that $Z \cong \pr(1,1,3)$ by Lemma \ref{unique_1/n}.
Let $L := \mathcal{O}_Z(1)$ and $C_Z := \psi_{*} C$, where $C$ is the exceptional curve of $\varphi$.
Then we can see that $-K_Z \sim 5L$ and there is an integer $n \in \mathbb{Z}_{\ge 0}$ such that $C_Z  \sim nL$
Denote the exceptional curve of $\psi$ by $E$.
Then we have $K_Y = \varphi^{*} K_X$ and $K_Y = \psi^{*} K_Z + 6E$.
By using these relations, we obtain relations, $E^2 = -\frac{1}{12}$ and $K_Y \cdot E = -\frac{1}{2}$.
We set $\psi^{*} C_Z = C +\alpha E$.
Then we have
\[
- K_Y \cdot \psi^{*} C_Z = - K_Y \cdot C + (- K_Y) \cdot \alpha E \ .
\]
Since $- K_Y \cdot \psi^{*} C_Z = -K_Z \cdot C_Z = \frac{5}{3}n$ and $-K_Y \cdot C = 0$, we obtain $\alpha = \frac{10}{3}n$.
Thus we have $\psi^* C_Z - \frac{10}{3}n E = C$.
Since $(\psi^* C_Z - \frac{10}{3}n E)^2 = \frac{1}{3}n^2 + \frac{100}{9}n^2 \cdot \frac{-1}{12}$ and $C^2 = -\frac{4}{3}$, we have $16n^2 = 36$.
Thus we see that $n = \pm \frac{3}{2}$.
This is a contradiction.

Assume that $Z$ is of the case 18-6.
By Lemma \ref{unique_1/5_A2}, we see that $Z \cong \pr(1,3,5)$.
Let $L := \mathcal{O}_Z(1)$ and $C_Z := \psi_{*} C$, where $C$ is the exceptional curve of $\varphi$.
Then wesee that $-K_Z \sim 9L$ and there is an integer $n \in \mathbb{Z}_{\ge 0}$ such that $C_Z \sim nL$.
Denote the exceptional curve of $\psi$ by $E$.
Then we have $K_Y = \varphi^{*} K_X$ and $K_Y = \psi^{*} K_Z + \frac{2}{5}E$.
By using these relations, we obtain relations $E^2 = -\frac{5}{12}$ and $K_Y \cdot E = -\frac{1}{6}$.
We set $\psi^{*} C_Z = C +\alpha E$.
Then we have
\[
- K_Y \cdot \psi^{*} C_Z = - K_Y \cdot C + (- K_Y) \cdot \alpha E \ .
\]
Since $- K_Y \cdot \psi^{*} C_Z = -K_Z \cdot C_Z = \frac{3}{5}n$ and $-K_Y \cdot C = 0$, we obtain $\alpha = \frac{18}{5}n$.
Thus we have $\psi^{*} C_Z - \frac{18}{5}nE  = C$.
Since $(\psi^{*} C_Z - \frac{18}{5}nE)^2 = \frac{1}{15}n^2 + \frac{324}{25}n^2 \cdot \frac{-5}{12}$ and $C^2 = -\frac{4}{3}$, we have $320 n^2 = 80$.
Thus we see that $n = \pm \frac{1}{2}$.
This is a contradiction. 
Thus we see that the all cases are impossible.
Therefore, we see that there is no del Pezzo surface of No.18. 

\end{proof}

\begin{proposition}\label{non_exi_19}

There is no example of the case $No.19$.

\end{proposition}

\begin{proof}

Let $X$ be a del Pezzo surface.
Assume that $X$ is of No.19.
Let $\varphi : Y \to X$ be the extraction of $A_3$.
By Lemma \ref{two_ray_game}, there is a $K_{Y}$-negative contraction $\psi : Y \to Z$. 
We see that $\sS (Y) = \{ \frac{1}{4}(1,1), \frac{1}{4}(1,1), \frac{1}{3}(1,1), A_2 \}$.
By Lemma \ref{non_del}, candidates of $\psi : Y \to Z$ are following.

{\small
\begin{table}[htb]
\renewcommand{\arraystretch}{1.4}
  \begin{tabular}{|c|c|c|c|c|c|} \hline
    \multicolumn{1}{|l|}{No.} & $\psi$ & \multicolumn{1}{c|}{From}
      & \multicolumn{1}{c|}{To} & \multicolumn{1}{c|}{$\sS (Z)$} & $ \det (Z) $ \\ \hline \hline
    19-1 & $\sB_0$ & - & sm pt & $\frac{1}{4}(1,1)$, $\frac{1}{4}(1,1)$, $\frac{1}{3}(1,1)$, $A_2$ & 912 \\ 
    19-2 & $\sB_2$ & $A_2$ & sm pt & $\frac{1}{4}(1,1)$, $\frac{1}{4}(1,1)$, $\frac{1}{3}(1,1)$ & 400 \\
    19-3 & $\sB_4$ & $\frac{1}{3}(1,1)$ & $A_1$ & $\frac{1}{4}(1,1)$, $\frac{1}{4}(1,1)$, $A_2, A_1$ & 576 \\
    19-4 & $\sB_5$ & $\frac{1}{4}(1,1)$ & $\frac{1}{3}(1,1)$ & $\frac{1}{4}(1,1)$, $\frac{1}{3}(1,1)$, $\frac{1}{3}(1,1)$, $A_2$ & 612 \\
    19-5 & $\sB_{10}$ & $\frac{1}{4}(1,1)$, $A_2$ & sm pt &  $\frac{1}{4}(1,1)$, $\frac{1}{3}(1,1)$ & 100  \\
    19-6 & $\sB_{14}$ & $\frac{1}{4}(1,1)$, $\frac{1}{3}(1,1)$ & sm pt & $\frac{1}{5}(1,2)$, $\frac{1}{4}(1,1)$, $A_2$  & 324 \\
    19-7 & $\sC_4$ & $\frac{1}{3}(1,1)$, $A_2$ &  pt on $\pr^1$ &  &  \\      \hline 
  \end{tabular}
\end{table}
}

The cases 19-1 and 19-4 are eliminated by Lemma \ref{lattice_square}. 
The case 19-2 is the same as the case No.8.
The case 19-5 is the same as the case No.7.
Hence these two cases are impossible by Proposition \ref{non_exi_7_10}.
The case 19-6 is the same as the case No.17.
Hence this case is impossible by Proposition \ref{non_exi_17_18}.
In the case 19-7, the rest two singular points of type $\frac{1}{4}(1,1)$ must be contracted to a point of $\pr^1$. 
By Table \ref{fibers}, it is impossible.
In the case 19-3, $Z$ is a del Pezzo surface of index two.
By Theorem \ref{min_index2}, we see that this case is impossible.
Thus we see that the all cases are impossible.
Therefore, we see that there is no del Pezzo surface of No.19.

\end{proof}

\subsection{Minimal surfaces of rank two}

In this subsection, rank two minimal del Pezzo surfaces of type $\tB$ are classified.  
A rank two minimal surface has two distinct $\pr^1$-fibrations $\pi_1, \pi_2$.

\begin{definition}\label{orb_Euler}

Let $X$ be a normal projective surface. The {\it orbifold Euler number} of $X$ is defined as
\[
e_{orb}(X) := e_{top}(X) - \sum_{x \in {\rm Sing} X} \frac{ \# \pi_{X,x} -1}{\# \pi_{X,x}}  ,
\]
where $e_{top}(X)$ is the topological Euler number of $X$ and $\pi_{X,x}$ is the local fundamental group of $X$ at $x \in X$.

\end{definition}

\begin{thm}\cite[Theorem 3.2]{HK11}\label{Miyaoka-Yau}
Let X be a normal projective surface with quotient singularities such that $-K_X$ is nef. Then 
\[
0 \le e_{orb}(X) .
\]

\end{thm}

\begin{corollary}\label{rank2_ineq}
Let X be a rank two del Pezzo surface with at most quotient singularities. 
Then
\[
\sum_{x \in {\rm Sing} X} \frac{ \# \pi_{X,x} -1}{\# \pi_{X,x}} \le 4 .
\]

In particular, the number of singular points on $X$ is at most eight.

\end{corollary}

\begin{remark}\label{l_f_g}

A local fundamental group $\pi_{X,x}$ is described more detail in \cite{Br67}.
Here we collect some of such descriptions.
$\# \pi_{X,x} = 1$ if and only if $x \in X$ is a smooth point.
$\# \pi_{X,x} = 2,3,4,3,4$ and 5, respectively for $x$ is one of $A_1$-, $A_2$-, $A_3$-, $\frac{1}{3}(1,1)$-, $\frac{1}{4}(1,1)$- and $\frac{1}{5}(1,2)$-singularities.

\end{remark}

\begin{lemma}\label{two_fib}

Let X be a rank two minimal surface of type $\tB$. 
Then the following hold:
\begin{align*}
\sharp \{ x \in X | \sharp \pi_{X,x} = 2 \} & \in  \{ 0, 4,6,8 \}, \\
\sharp \{ x \in X | \sharp \pi_{X,x} = 3 \} & \in \{ 0, 4,6 \}, \\
\sharp \{ x \in X | \sharp \pi_{X,x} = 5 \} & \in \{ 0, 4 \}.
\end{align*}

\end{lemma}

\begin{proof}

It follows from Table \ref{fibers}, Corollary \ref{rank2_ineq} and Remark \ref{l_f_g}.

\end{proof}

\begin{proposition}\label{candi_rank2}

Let $X$ be a rank two minimal surface of type $\tB$.
$X$ is one in the following table:

{\small \rm
\begin{table}[htb]
\renewcommand{\arraystretch}{1.3}
  \begin{tabular}{|c|c|c|c|} \hline
    \multicolumn{1}{|l|}{No.} & \multicolumn{1}{c|}{$\sS (X)$}
      & \multicolumn{1}{c|}{$(-K_{X})^2$}     \\ \hline \hline
    1 & $\frac{1}{5}(1,2), \frac{1}{5}(1,2), \frac{1}{5}(1,2), \frac{1}{5}(1,2)$ & $\frac{8}{5}$   \\
    2 & $\frac{1}{4}$(1,1), $\frac{1}{4}$(1,1), $A_3, A_3$  & 2   \\
    3 & $\frac{1}{4}$(1,1), $A_3, A_3$ & 2    \\
    4 & $\frac{1}{3}$(1,1), $\frac{1}{3}$(1,1), $A_2, A_2$ & $\frac{8}{3}$  \\
    5 &  $A_3, A_3$  & 2  \\ 
    6 & $A_1, A_1, A_1, A_1$ & 4  \\
    7 & $A_1, A_1, A_1, A_1, A_1, A_1$ & 2  \\
    8 & $A_3, A_1, A_1, A_1, A_1$ & 1   \\
    9 & $A_3$ & 5   \\ \hline
  \end{tabular}
\end{table}
}

\end{proposition}

\begin{proof}

Let $\pi : Y \to X$ be the minimal resolution.
Then we have
\[
K_Y = \pi^*K_X + \sum^r a_i E_i,
\]
where each $E_i$ is the exceptional curve and $a_i \le 0$.
Thus we have
\[
K_Y^2 = K_X^2 + (\sum^r a_i E_i)^2 .
\]
By Remark \ref{Noeth_RR},
\[
K_X^2 = 8 - r - (\sum^r a_i E_i)^2 .
\]
Candidates satisfying $K_X^2 > 0$ and Lemma \ref{two_fib} are the nine cases.

\end{proof}

A del Pezzo surface of No.7 does exist.
We, however, see that this surface cannot be obtained from a del Pezzo surface of type $\tA$ by Corollary \ref{bound_vol}. 

Gorenstein del Pezzo surfaces of rank two are also classified by Qiang (\cite{Ye02}).
The following lemma holds.

\begin{lemma}\cite[Qiang]{Ye02}\label{nonexi_8}
There is no rank two del Pezzo surface which has one singular point of type $A_3$ and four singular points of type $A_1$.

\end{lemma}

\begin{lemma}\label{nonexi_9}

There is no del Pezzo surface of rank two which has one singular point of type $A_3$ and has two distinct $\pr^1$-fibrations $\pi_1, \pi_2$.

\end{lemma}

\begin{proof}

Assume there exists such a surface $X$ by contradiction.
Let $l_1, l_2$ be general fibers of $\pi_1, \pi_2$ respectively and set $l_1 \cdot l_2 = d \in \mathbb{Z}$.
Since $\rho(X) = 2$, we can write $-K_X = al_1 + bl_2$.
Then we have
$2 = -K_X \cdot l_1 = (al_1 + bl_2) \cdot l_1 = bd$ and $2 = -K_X \cdot l_2 = (al_1 + bl_2) \cdot l_2 = ad$.
Thus we see
\[
 (-K_X)^2 = (al_1 + bl_2)^2 = 2abd = \frac{8}{d} .
\]
Since $X$ is of rank two and has only one singular point of type $A_3$, we can see $(-K_X)^2 = 5$.
This, however, contradicts the fact that $d$ is an integer. 

\end{proof}

\section{Candidates of del Pezzo surfaces of type $\tA$}\label{candidate_sec}

In the previous two sections, extremal contractions and minimal surfaces are classified.
By using these results, we determine candidates of del Pezzo surfaces of type $\tA$ in this section.
The existence of each candidate is proved in the next section.

Recall that for a del Pezzo surface $X$ of type $\tA$, there exists a minimal directed sequence of first morphisms, second morphisms and third morphisms (Theorem \ref{directed}) 
\[
X \to \cdots \to S \to \cdots \to T_i \to \cdots \to T_{min} \to \cdots \to U_n = X_{min} .
\]
The purpose of this paper is to classify $X$ in case $X = S$.
We will prove the following theorem.

\begin{thm}\label{candi_S}

Let $S$ be a del Pezzo surface of type $\tA$ with no floating $(-1)$-curves.
Then $S$ is one of the surfaces in Table \ref{main_table}.

\end{thm}

\begin{proof}

The assertion follows from Table \ref{T_min_table}, \ref{T_1_table}, \ref{T_2_table} and \ref{T_3_table}.

\end{proof}

In order to prove this theorem, we first classify $T_{min}$.
Next, we classify $T_i$, where $S =: T_m \to \cdots \to T_2 \to T_1 \to T_{min}$ is a directed $\smor$-sequence.
By observing the anti-canonical volume $(-K_X)^2$, we see that there is no $T_m$ for $m \ge 4$ (Corollary \ref{T_>4}).
Hence it suffices to classify $T_1, T_2$ and $T_3$.

\subsection{Candidates of $T_{min}$}\label{T}

We first determine the candidates of a $\smor$-minimal del Pezzo surface $T_{min}$.

\begin{definition}

Let $X$ be a normal projective surface.
Let $\pi : Y \to X$ be the minimal resolution.
Let $C \subset X$ be a quasi-line.
Assume that $C$ passes through exactly two singular points $P_1, P_2$.
If $P_1$ is a singular point of type $\frac{1}{3}(1,1)$ and $P_2$ is one of type $A_1$, then $C$ is called an {\it $S_1$-line}.
The dual graph of the total transform with reduced structure of an $S_1$-line on $Y$ is the following. 
\begin{center}
$\pi^{*}(S_1$-line$)_{\rm red}$ : \ \  \xygraph{
	\square ([]!{+(.0,-.3)} {-3}) 
	- [r]	\bullet ([]!{+(.0,-.3)} {-1}) 
        - [r]	\triangle ([]!{+(.0,-.3)} {-2}) }
\end{center}
If $P_1$ is a singular point of type $\frac{1}{4}(1,1)$ and $P_2$ is one of type $A_2$, then $C$ is called an {\it $S_2$-line}.
The dual graph of the total transform with reduced structure of an $S_2$-line on $Y$ is the following. 
\begin{center}
$\pi^{*}(S_2$-line$)_{\rm red}$ : \ \  \xygraph{
	\bigcirc ([]!{+(.0,-.3)} {-4}) 
	- [r]	\bullet ([]!{+(.0,-.3)} {-1}) 
        - [r]	\triangle ([]!{+(.0,-.3)} {-2}) 
        - [r]	\triangle ([]!{+(.0,-.3)} {-2}) }
\end{center}
Let $D \subset X$ be a different curve from $C$.
Assume that $C \cap D = \{ P_1 \}$ and $D$ passes through exactly two singular points $P_1, P_3$.
If $P_1$ is a singular point of type $\frac{1}{4}(1,1)$, $P_2$ is one of type $\frac{1}{3}(1,1)$ and $P_3$ is one of type $A_1$, then $C \cup D $ is called an {\it $S_3$-line pair}.
The dual graph of the total transform with reduced structure of $S_3$-line pair on $Y$ is the following.
\begin{center}
$\pi^{*}(S_3$-line pair$)_{\rm red}$ : \ \  \xygraph{
	\square ([]!{+(.0,-.3)} {-3}) 
	- [r]	\bullet ([]!{+(.0,-.3)} {-1}) 
        - [r]	\bigcirc ([]!{+(.0,-.3)} {-4}) 
        - [r]	\bullet ([]!{+(.0,-.3)} {-1}) 
        - [r]	\triangle ([]!{+(.0,-.3)} {-2}) }
\end{center}

\end{definition}

\begin{lemma}\label{cont_to_sm}

Let $X$ be a $\smor$-minimal del Pezzo surface.
Take a $\tmor$-sequence $X =: U_0 \to \cdots \to U_i \to \cdots \to U_n =: X_{min}$.
For $0 \le i \le n$, $U_i$ does not have neither $S_1$-lines, $S_2$-lines nor $S_3$-line pairs. 

\end{lemma}

\begin{proof}

Since $S_1$-lines, $S_2$-lines and $S_3$-line pairs are contracted to smooth points, the assertion holds by Lemma \ref{A-min}.

\end{proof}

\begin{proposition}\label{T_min}

Candidates of $\smor$-minimal del Pezzo surfaces are the following:
{\small \rm
\begin{table}[htb]
\caption{Candidates of $\smor$-minimal del Pezzo surfaces}\label{T_min_table}
\renewcommand{\arraystretch}{1.4}
  \begin{tabular}{|c|c|c|c|c|c|c|c|c|}  \hline
    \multicolumn{1}{|l|}{No.} & \multicolumn{2}{c|}{$X_{min}$} & directed sequence
      & \multicolumn{1}{c|}{($n_3, n_4$)} & \multicolumn{1}{c|}{$(-K_{T_{min}})^2$} & $\rho (T_{min})$   \\ \hline \hline
    1 & $M_{13}$ & $\frac{1}{5}(1,2) \times 4 $ & $\tmor_1 \circ \tmor_1 \circ \tmor_1 \circ \tmor_1$ & (4,4) & $\frac{4}{3}$ & 6   \\ 
    2  & $M_8$ & $A_2 \times 3$ &  $\tmor_5 \circ \tmor_5 \circ \tmor_5$ & (6,0) & 2 & 4  \\
    3 & $M_6$ & $\frac{1}{4}$(1,1) & - & (0,1) & 9 & 1   \\
    4 & $M_5$ & $\frac{1}{3}$(1,1) & - & (1,0) & $\frac{25}{3}$  & 1   \\ 
    5 & $M_{12}$ & $\pr^2$ & - & (0,0) & 9 & 1  \\
    6 & $M_{19}$ & $\PP$ & - & (0,0) & 8 & 2  \\ \hline
  \end{tabular}
\end{table}
}

\end{proposition}

\begin{proof}

Let $T$ be a $\smor$-minimal del Pezzo surface.
In order to classify $\smor$-minimal surfaces, for each $X_{min}$ in Tables \ref{min_rank1_table} and \ref{min_rank2_table}, we check possible minimal directed $\tmor$-sequences $T := U_0 \to \cdots \to U_i \to \cdots \to U_n = X_{min}$, where each $U_i$ is a del Pezzo surface of type $\tB$.
For $0 \le i \le n$, let $\pi_i : Y_i \to U_i$ be the minimal resolution.
Then a sequence $Y_0 \to Y_1 \to \cdots \to Y_n =: Y_{min}$ is induced.
Note that for a minimal directed $\tmor$-sequence, we can change the order of the third morphisms since the centers of the third morphisms are disjoint.
In this proof, a sequence obtained by changing third morphisms in a minimal directed $\tmor$-sequence is also called a minimal directed $\tmor$-sequence. 

\noindent{\bf Case 1 :  $X_{min} = M_1 \ ( = \pr(3,4,5))$ } 

Note that $\sS(X_{min}) = \{ \frac{1}{5}(1,2), A_3, A_2 \}$.
Since $X_{min}$ has a singular point of type $\frac{1}{5}(1,2)$ and one of type $A_3$, we obtain a minimal directed $\tmor$-sequence $T \overset{\tmor_i}{\rightarrow} U_1 \overset{\tmor_1}{\rightarrow}  U_2 \overset{\tmor_3}{\rightarrow} X_{min}$, where $4 \le i \le 6$.
Then we can find the following configuration of negative curves on $Y_1$.
\begin{center}
 \  \xygraph{
	\bigcirc ([]!{+(.0,-.3)} {-4}) 
	- [r]	\bullet ([]!{+(.0,-.3)} {-1}) 
        - [r]	\square ([]!{+(.0,-.3)} {-3}) 
        - [r]	\bullet ([]!{+(.0,-.3)} {-1}) 
        - [r]	\triangle ([]!{+(.0,-.3)} {-2})
        - [r]	\triangle ([]!{+(.0,-.3)} {-2})
        - [r]	\bullet ([]!{+(.0,-.3)} {-1})
        - [r]	\square ([]!{+(.0,-.3)} {-3}) 
        - [r]	\bullet ([]!{+(.0,-.3)} {-1}) 
        - [r]	\bigcirc ([]!{+(.0,-.3)} {-4})  }
\end{center}
If $i=5$ or 6, then we can find a $\smor_4$-line pair on $T$.
Hence this case is impossible by Lemma \ref{A-min}.
If $i=4$, then we can obtain a sequence $T \overset{\tmor_1}{\to} U_1 \overset{\tmor_1}{\to} U_2 \overset{\tmor_1}{\to} U_3 \overset{\tmor_1}{\to} M_{13}$ by observing negative curves on $Y$ more carefully.
This contradicts the minimal directedness.

\noindent{\bf Case 2 :  $X_{min} = M_2 \ ( = \pr(1,3,5))$} 

Note that $\sS(X_{min}) = \{ \frac{1}{5}(1,2), A_2 \}$.
We know that $Y_{min}$ has the following negative curves.
\begin{center}
 \  \xygraph{
	\triangle ([]!{+(.0,-.3)} {-2}) 
        - [r]	\square ([]!{+(.0,-.3)} {-3}) 
        - [r]	\bullet ([]!{+(.0,-.3)} {-1}) 
        - [r]	\triangle ([]!{+(.0,-.3)} {-2})
        - [r]	\triangle ([]!{+(.0,-.3)} {-2}) }
\end{center}
Since $X_{min}$ has a singular point of type $\frac{1}{5}(1,2)$, we see that $T \overset{\tmor_i}{\rightarrow} U_1 \overset{\tmor_1}{\rightarrow} X_{min}$, where $i = 4, 5$ or 6.
Thus we can find an $S_2$-line on $U_1$.
Hence this case is impossible by Lemma \ref{cont_to_sm}.

\noindent{\bf Case 3 :  $X_{min} = M_3 \ ( = \pr(1,2,5))$} 

Note that $\sS(X_{min}) = \{ \frac{1}{5}(1,2), A_1 \}$.
We know that $Y_{min}$ has the following negative curves.
\begin{center}
 \  \xygraph{
	\square ([]!{+(.0,-.3)} {-3}) 
        - [r]	\triangle ([]!{+(.0,-.3)} {-2}) 
        - [r]	\bullet ([]!{+(.0,-.3)} {-1})
        - [r]	\triangle ([]!{+(.0,-.3)} {-2}) }
\end{center}
Since $X_{min}$ has a singular point of type $\frac{1}{5}(1,2)$, we have $T \overset{\tmor_i}{\rightarrow} U_1 \overset{\tmor_1}{\rightarrow} X_{min}$, where $i = 7, 8$ or 9.
Thus we can find an $S_1$-line on $U_1$.
Hence this case is impossible by Lemma \ref{cont_to_sm}.

\noindent{\bf Case 4 :  $X_{min} = M_4 \ ( = \pr(1,3,4))$} 

Note that $\sS(X_{min}) = \{ \frac{1}{3}(1,1), A_3 \}$.
We know that $Y_{min}$ has the following negative curves.
\begin{center}
 \  \xygraph{
	\square ([]!{+(.0,-.3)} {-3}) 
        - [r]	\bullet ([]!{+(.0,-.3)} {-1})
        - [r]	\triangle ([]!{+(.0,-.3)} {-2})
        - [r]	\triangle ([]!{+(.0,-.3)} {-2})
        - [r]	\triangle ([]!{+(.0,-.3)} {-2}) }
\end{center}
We can find $\smor_4$-line pair on $U_1$ where $U_1 \overset{\tmor_3}{\to} X_{min}$. 
Hence this case is impossible.

\noindent{\bf Case 5 :  $X_{min} = M_5 \  (= \pr(1,1,3))$} 

$\pr(1,1,3)$ is a $\smor$-minimal surface.
This case is No.4 in Table \ref{T_min_table}.
If the singular point of type $\frac{1}{3}(1,1)$ is produced by $\tmor_2$, that is, $T \overset{\tmor_2}{\to} \pr(1,1,3)$, then we can find a floating $(-1)$-curve on $T$.
This is a contradiction.

\noindent{\bf Case 6 :  $X_{min} = M_6 \ ( = \pr(1,1,4))$} 

Since any singular points of type $\frac{1}{4}(1,1)$ cannot be produced by extremal contractions in Table \ref{bir_ext_cont}, we see that $T = X_{min} = \pr(1,1,4)$ in this case.
This case is No.3 in Table \ref{T_min_table}.

\noindent{\bf Case 7 :  $X_{min} = M_7$} 

Note that $\sS(X_{min}) = \{ A_1, A_3, A_3 \}$.
We know that $Y_{min}$ has the following negative curves.
\begin{center}
 \  \xygraph{
	\triangle ([]!{+(.0,-.3)} {-2}) 
        - [r]	\triangle ([]!{+(.0,-.3)} {-2}) 
        - [r]	\triangle ([]!{+(.0,-.3)} {-2}) 
        - [r]	\bullet ([]!{+(.0,-.3)} {-1}) 
        - [r]	\triangle ([]!{+(.0,-.3)} {-2}) 
        - [r]	\bullet ([]!{+(.0,-.3)} {-1}) 
        - [r]	\triangle ([]!{+(.0,-.3)} {-2})
        - [r]	\triangle ([]!{+(.0,-.3)} {-2}) 
        - [r]	\triangle ([]!{+(.0,-.3)} {-2}) }
\end{center}
Since $X_{min}$ has two singular points of type $A_3$, we have $T \overset{\tmor_i}{\to} U_1 \overset{\tmor_3}{\to} U_2 \overset{\tmor_3}{\to} X_{min}$, where $7 \le i \le 9$.
Thus we can find an $S_3$-line pair on $U_2$.
Hence this case is impossible.

\noindent{\bf Case 8 :  $X_{min} = M_8$} 

Since $\sS(X_{min}) = \{ A_2, A_2, A_2 \}$, we have a sequence $T \overset{\tmor_i}{\to} U_1 \overset{\tmor_j}{\to} U_2 \overset{\tmor_k}{\to} X_{min}$, where $4 \le i, j, k \le 6$.
If $k = 4$ or 6, then we can find an $S_2$-line on $U_2$.
Thus we may assume $i = j = k = 5$. 
This case is No.2 in Table \ref{T_min_table}.

\noindent{\bf Case 9 :  $X_{min} = M_9$} 

Note that $\sS(X_{min}) = \{ A_1, A_1, A_3 \}$.
We know that $Y_{min}$ has the following negative curves.
\begin{center}
 \  \xygraph{
	\triangle ([]!{+(.0,-.3)} {-2}) 
        - [r]	\bullet ([]!{+(.0,-.3)} {-1})
        - [r]	\triangle ([]!{+(.0,-.3)} {-2})
        - [r]	\triangle ([]!{+(.0,-.3)} {-2})
        - [r]	\triangle ([]!{+(.0,-.3)} {-2}) 
        - [r]	\bullet ([]!{+(.0,-.3)} {-1}) 
        - [r]	\triangle ([]!{+(.0,-.3)} {-2})   }
\end{center}
Since $X_{min}$ has a singular point of type $A_3$, we have $T \overset{\tmor_i}{\to} U_1 \overset{\tmor_j}{\to} U_2 \overset{\tmor_3}{\to} X_{min}$, where $7 \le i, j \le 9$.
Thus we can find an $S_1$-line on $U_2$.
Hence this case is impossible.

\noindent{\bf Case 10 :  $X_{min} = M_{10} \ ( = \pr(1,2,3))$} 

Since $\sS(X_{min}) = \{ A_1, A_2 \}$, we have a sequence $T \overset{\tmor_i}{\to} U_1 \overset{\tmor_j}{\to} X_{min}$, where we can assume that $7 \le i \le 9$ and $4 \le j \le 6$.
We know that $Y_{min}$ has the following negative curves.
\begin{center}
 \ \  \xygraph{
	\triangle ([]!{+(.0,-.3)} {-2}) 
	- [r]	\bullet ([]!{+(.0,-.3)} {-1}) 
        - [r]	\triangle ([]!{+(.0,-.3)} {-2}) 
        - [r]	\triangle ([]!{+(.0,-.3)} {-2}) }
\end{center}
If $j = 4$, then we can find an $S_3$-line pair on $U_1$. 
If $j = 5$, then we can find an $S_1$-line on $U_1$. 
If $j = 6$, we see that there are two candidates of $U_1 \overset{\tmor_6}{\to} X_{min}$ by above configuration. 
Then we can find $S_1$-line or $S_3$-line pair on $U_1$ for these two cases.
Hence this case is impossible.

\noindent{\bf Case 11 :  $X_{min} = M_{11} \ ( = \pr(1,1,2))$} 

Since $\sS(X_{min}) = \{ A_1 \}$, we have a sequence $T \overset{\tmor_i}{\to} X_{min}$, where $7 \le i \le 9$.
If $i = 7$ or 8, then we can find a floating $(-1)$-curve on $T$.
If $i = 9$, then we can find a $\smor_1$-line pair on $T$.
Hence this case is impossible.

\noindent{\bf Case 12 :  $X_{min} = M_{12} \ ( = \pr^2)$}

Since $X_{min} = \pr^2$ is nonsingular and minimal, we see that $T_{min} = X_{min}$.  
This case is No.5 in Table \ref{T_min_table}.

\noindent{\bf Case 13 :  $X_{min} = M_{13}$} 

Since $\sS(X_{min}) = \{ \frac{1}{5}(1,2), \frac{1}{5}(1,2), \frac{1}{5}(1,2), \frac{1}{5}(1,2) \}$, there is only one possibility of a $\tmor$-sequence, $T \overset{\tmor_1}{\to} U_1 \overset{\tmor_1}{\to} U_2 \overset{\tmor_1}{\to} U_3 \overset{\tmor_1}{\to} X_{min}$.
This case is No.1 in Table \ref{T_min_table}.

\noindent{\bf Case 14 :  $X_{min} = M_{14}$} 

Since $\sS(X_{min}) = \{ \frac{1}{4}(1,1), \frac{1}{4}(1,1), A_3, A_3 \}$, there is only one possibility of a $\tmor$-sequence, $T \overset{\tmor_3}{\to} U_1 \overset{\tmor_3}{\to} X_{min}$.
Then we can find a sequence $T \overset{\tmor_1}{\to} U_1 \overset{\tmor_1}{\to} U_2 \overset{\tmor_1}{\to} U_3 \overset{\tmor_1}{\to} M_{13}$.
This is a contradiction to the minimal directedness.
Hence this case is impossible.

\noindent{\bf Case 15 :  $X_{min} = M_{15}$} 

Since $\sS(X_{min}) = \{ \frac{1}{4}(1,1), A_3, A_3 \}$, there is only one possibility of a $\tmor$-sequence, $T \overset{\tmor_3}{\to} U_1 \overset{\tmor_3}{\to} X_{min}$.
We know that $Y_{min}$ has the following negative curves.
\begin{center}
\ \ \xygraph{
    \bigcirc ([]!{+(0,-.4)} {-4}) - [r]
    \bullet ([]!{+(0,-.4)} {-1}) - [r]
    \triangle ([]!{+(0,-.4)} {-2}) - [r]
    \triangle ([]!{+(.4,-.3)} {-2}) (
        - [d] \bullet ([]!{+(.5,0)} {-1}),
        - [r] \triangle ([]!{+(.2,-.4)} {-2}) }
\end{center}
Then we can find a $\smor_7$-line pair on $U_1$.
Hence this case is impossible.

\noindent{\bf Case 16 :  $X_{min} = M_{16}$} 

Note that $\sS(X_{min}) = \{ \frac{1}{3}(1,1), \frac{1}{3}(1,1), A_2, A_2 \}$.
We know that $Y_{min}$ has the following negative curves.
\begin{center}
 \ \  \xygraph{
	\square ([]!{+(.0,-.3)} {-3}) 
	- [r]	\bullet ([]!{+(.0,-.3)} {-1}) 
        - [r]	\triangle ([]!{+(.0,-.3)} {-2}) 
        - [r]	\triangle ([]!{+(.0,-.3)} {-2})
	- [r]	\bullet ([]!{+(.0,-.3)} {-1})
	- [r]	\square ([]!{+(.0,-.3)} {-3}) }
\end{center}
If at least one of the singular points of type $\frac{1}{3}(1,1)$ is produced by the third morphism of type $\tmor_2$, then we can find an $S_2$-line on $U_1$, where $U_1 \overset{\tmor_2}{\to} X_{min}$.
Hence we may assume a minimal directed $\tmor$-sequence of $T$ is $T \overset{\tmor_i}{\to} U_1 \overset{\tmor_j}{\to} X_{min}$, where $4 \le i, j \le 6$.
If $j=6$, then we can find a $\smor_4$-line pair on $U_1$.
Hence we may assume that $4 \le i,j \le 5$.  
If $i=j=4$, then we can find a sequence $T \overset{\tmor_1}{\to} U_1 \overset{\tmor_1}{\to} U_2 \overset{\tmor_1}{\to} U_3 \overset{\tmor_1}{\to} M_{13}$.
This is a contradiction to the minimal directedness.
If $i = 4$ and $j = 5$, then we can find a $\smor_4$-line pair on $T$.
These cases are impossible.
If $i = j = 5$, then we can find a sequence $T \overset{\tmor_3}{\to} U_1 \overset{\tmor_3}{\to} U_2 \overset{\tmor_3}{\to} M_8$.
This is also a contradiction to the minimal directedness.

\noindent{\bf Case 17 :  $X_{min} = M_{17}$} 

Note that $\sS(X_{min}) = \{ A_3, A_3 \}$. 
We know that $Y_{min}$ has the following negative curves:
\begin{center}
 \ \  \xygraph{
	\triangle ([]!{+(.0,-.3)} {-2}) 
        - [r]	\triangle ([]!{+(.0,-.3)} {-2}) 
        - [r]	\triangle ([]!{+(.0,-.3)} {-2})
	- [r]	\bullet ([]!{+(.0,-.3)} {-1}) 
        - [r]	\triangle ([]!{+(.0,-.3)} {-2}) 
        - [r]	\triangle ([]!{+(.0,-.3)} {-2})
        - [r]	\triangle ([]!{+(.0,-.3)} {-2})  }
\end{center}
Thus there is one possibility of a $\tmor$-sequence, $T \overset{\tmor_3}{\to} U_1 \overset{\tmor_3}{\to} X_{min}$.
Then we can find a $\smor_4$-line pair on $T$.
Hence this case is impossible.

\noindent{\bf Case 18 :  $X_{min} = M_{18}$} 

Note that $\sS(X_{min}) = \{ A_1, A_1, A_1, A_1 \}$. 
We know that $Y_{min}$ has the following negative curves:
\begin{center}
 \ \  \xygraph{
	\triangle ([]!{+(.0,-.3)} {-2}) 
	- [r]	\bullet ([]!{+(.0,-.3)} {-1}) 
        - [r]	\triangle ([]!{+(.0,-.3)} {-2}) 
	- [r]	\bullet ([]!{+(.0,-.3)} {-1}) 
        - [r]	\triangle ([]!{+(.0,-.3)} {-2})
	- [r]	\bullet ([]!{+(.0,-.3)} {-1}) 
        - [r]	\triangle ([]!{+(.0,-.3)} {-2})  }
\end{center}
Thus we have a sequence $T \overset{\tmor_i}{\to} U_1 \overset{\tmor_j}{\to} U_2 \overset{\tmor_k}{\to}  U_3 \overset{\tmor_l}{\to} X_{min}$, where $7 \le i, j, k, l \le 9$.
If $l = 7$, then we can find an $S_1$-line on $U_3$.
Thus we may assume that $8 \le i,j,k,l \le 9$.
We, however, see that $(-K_T)^2 \le 0$.
This contradicts the fact that $T$ is a del Pezzo surface. 

\noindent{\bf Case 19 :  $X_{min} = M_{19}$} 

Since $X_{min} = \PP$ is nonsingular and minimal, we see that $T_{min} = X_{min}$.  
This case is No.6 in Table \ref{T_min_table}. 

\end{proof}

\begin{remark}

From now on we write $M_5, M_6, M_{12}$ and $M_{19}$ for $\pr(1,1,3)$, $\pr(1,1,4)$, $\pr^2$ and $\PP$.

\end{remark}

\subsection{Candidates of $T_m$}

Let $T_m$ be a del Pezzo surface such that the length of its minimal directed $\smor$-sequence is $m$.
We determine candidates of a surface $T_m$.





We prepare some lemmas.
Let $\varphi : U \to U_1$ be a second morphism.
We use the same notation as in Corollary \ref{config}.

\begin{definition}

Let $X$ be a del Pezzo surface of type $\tA$ and $P \in X$ a smooth point.
If there is a quasi-0-curve passing through $P$ and a singular point of type $\frac{1}{4}(1,1)$, then we say that $P$ satisfies a condition ({\bf P}). 

\end{definition}

\begin{lemma}\label{quasi-0}

Let $U$ be a del Pezzo surface of type $\tA$ and $\varphi : U \to U_1$ a second morphism.
Let $P$ be the image of $\varphi$-exceptional curve.
If $P$ satisfies a condition {\rm ({\bf P})}, then $\varphi$ is of type $\smor_1$, $\smor_3$, $\smor_6$ or $\smor_8$.

\end{lemma}

\begin{proof}

Denote by $C$ the quasi-0-curve passing through $P$ and a singular point of type $\frac{1}{4}(1,1)$.
We write $C_U$ for the strict transform of $C$ by $\varphi$ and $C_Y$ for the strict transform of $C_U$ by $\pi$.
Then $C_Y$ is a $(-1)$-curve and $C_Y \cdot E_1 = 1$, where $E_1$ is the one in Corollary \ref{config}.
If $\varphi$ is of type $\smor_2$, $\smor_4$, $\smor_5$ or $\smor_7$, then $E_1$ is a $(-4)$-curve.
Thus $C_U$ is a $T$-line.
This is a contradiction.

\end{proof}

\begin{lemma}\label{P^2}

Let $X$ be a del Pezzo surface of type $\tA$ with no floating $(-1)$-curves and $\varphi : X \to X_1$ a second morphism.
If $X_1 \cong \pr^2$, then $\varphi$ is of type $\smor_3$, $\smor_4$, $\smor_6$, $\smor_7$ or $\smor_8$.

\end{lemma}

\begin{proof}

If $\varphi$ is of type $\smor_1, \smor_2$ or $\smor_5$, then we can confirm that $X$ has a floating $(-1)$-curve.
Thus $\varphi$ is of type $\smor_3$, $\smor_4$, $\smor_6$, $\smor_7$ or $\smor_8$.

\end{proof}

\begin{lemma}\label{P(1,1,4)_P1}

If $X$ is $\pr(1,1,4)$, any point of $X \backslash \Sing X$ satisfies {\rm ({\bf P})}. 

\end{lemma}

\begin{proof}

Let $Y \to X$ be the minimal resolution.
Since $Y \cong \mathbb{F}_4$, we see that any point of $X \backslash \Sing X$ satisfies {\rm ({\bf P})}.

\end{proof}

\subsubsection{Candidates of $T_1$}

We will determine candidates of $T_1$.
Then $T_1$ has a minimal directed $\smor$-sequence $T_1 \overset{\varphi}{\to} T_{min}$.

\begin{lemma}\label{T_1_dir}

If $\varphi$ is of type $\smor_i$, then $T_1$ has no $\smor_j$-line pair where $j < i$.

\end{lemma}

\begin{proof}

Assume that $T_1$ has such a $\smor_j$-line pair.
Then there is a minimal directed $\smor$-sequence $T_1 \overset{\smor_{k_1}}{\to} X_1 \overset{\smor_{k_2}}{\to} \cdots \overset{\smor_{k_l}}{\to} X_l$, where $k_1 \le j < i$ and $X_l$ is $\smor$-minimal.
This contradicts the fact that $T_1 \overset{\smor_i}{\to} T_{min}$ is a minimal directed $\smor$-sequence.

\end{proof}

\begin{proposition}\label{T_1}

Candidates of $T_1$ are the following:

{\small \rm
\begin{table}[htb]
\caption{Candidates of $T_1$}\label{T_1_table}
\renewcommand{\arraystretch}{1.4}
  \begin{tabular}{|c|c|c|c|c|c|c|c|}  \hline
    \multicolumn{1}{|l|}{No.} & \multicolumn{1}{c|}{$T_{min}$} & directed seq.
      & \multicolumn{1}{c|}{($n_3, n_4$)} & \multicolumn{1}{c|}{$(-K_{T_1})^2$} & $\rho (T_1)$ &  in Table \ref{main_table}   \\ \hline \hline
    1 & $\pr(1,1,4)$ & $\smor_8$ & (2,2) & $\frac{14}{3}$ & 4 & No.15   \\  
    2 & $\pr(1,1,3)$ & $\smor_7$ & (2,2) & $\frac{14}{3}$  & 4 & No.16  \\ \hline
    3 & $\pr(1,1,3)$ & $\smor_4$ & (3,1) & 5  & 3 & No.18   \\  \hline
    4 & $\pr(1,1,4)$ & $\smor_6$ & (1,2) & $\frac{16}{3}$ & 4 & No.22  \\  \hline
    5 & $\pr(1,1,4)$ & $\smor_3$ & (2,1) & $\frac{17}{3}$ & 3 & No.23  \\
    6 & $\pr(1,1,3)$ & $\smor_5$ & (2,1) & $\frac{14}{3}$  & 4 & No.24  \\ 
    7 & $\PP$ & $\smor_4$ & (2,1) & $\frac{14}{3}$ & 4 & No.25  \\  \hline
    8 & $\pr(1,1,3)$ & $\smor_3$ & (3,0) & 5  & 3 & No.27   \\  
  \end{tabular}
\end{table}
}

{\small \rm
\begin{table}[htb]
\renewcommand{\arraystretch}{1.4}
  \begin{tabular}{|c|c|c|c|c|c|c|c|} 
    \multicolumn{1}{|l|}{No.} & \multicolumn{1}{c|}{$T_{min}$} & directed seq.
      & \multicolumn{1}{c|}{($n_3, n_4$)} & \multicolumn{1}{c|}{$(-K_{T_1})^2$} & $\rho (T_1)$ &  in Table \ref{main_table}   \\ \hline \hline
    9 & $\pr(1,1,4)$ & $\smor_1$ & (1,1) & $\frac{19}{3}$ & 3  & No.29 \\
    10 & $\pr(1,1,3)$ & $\smor_2$ & (1,1) & $\frac{16}{3}$  & 4 & No.30  \\ 
    11 & $\pr(1,1,3)$ & $\smor_1$ & (2,0) & $\frac{17}{3}$  & 3 &  No.32  \\  \hline
    12 & $\PP$ & $\smor_2$ & (0,1) & 5 & 5 & No.35  \\   \hline
    13 & $\PP$ & $\smor_1$ & (1,0) & $\frac{16}{3}$ & 4 &  No.37 \\  \hline
  \end{tabular}
\end{table}
}

\end{proposition}

\begin{proof}

By Proposition \ref{T_min}, the candidates of $T_{min}$ are six cases.
Observing their anti-canonical volumes, we see that $T_{min}$ is one of $\pr(1,1,4)$, $\pr(1,1,3)$, $\pr^2$ and $\PP$.

\noindent{\bf Case 1 :  $T_{min} = \pr(1,1,4)$} 

By Lemmas \ref{quasi-0} and \ref{P(1,1,4)_P1}, we see that $\varphi$ is of type $\smor_1$, $\smor_3$, $\smor_6$ or $\smor_8$.
Thus cadidates are the following:
{\small \rm
\begin{table}[htb]
\renewcommand{\arraystretch}{1.4}
  \begin{tabular}{|c|c|c|c|c|c|}  \hline
      directed seq.  & \multicolumn{1}{c|}{($n_3, n_4$)} & \multicolumn{1}{c|}{$(-K_{T_1})^2$} & $\rho (T_1)$  &  in Table \ref{T_1_table}  \\ \hline \hline
     $\smor_1$ & (1,1) & $\frac{19}{3}$ & 3  &  No.9 \\
     $\smor_3$ & (2,1) & $\frac{17}{3}$ & 3 &  No.5 \\
     $\smor_6$ & (1,2) & $\frac{16}{3}$ & 4 & No.4 \\
     $\smor_8$ & (2,2) & $\frac{14}{3}$ & 4 &  No.1 \\    \hline
  \end{tabular}
\end{table}
} \\
We cannot eliminate all cases (we can confirm the existence in the next section).

\noindent{\bf Case 2 :  $T_{min} = \pr(1,1,3)$} 

Candidates are the following:
{\small \rm
\begin{table}[htb]
\renewcommand{\arraystretch}{1.4}
  \begin{tabular}{|c|c|c|c|c|c|}  \hline
 directed seq.   & \multicolumn{1}{c|}{($n_3, n_4$)} & \multicolumn{1}{c|}{$(-K_{T_1})^2$} & $\rho (T_1)$ & in Table \ref{T_1_table}   \\ \hline \hline
     $\smor_1$ & (2,0) & $\frac{17}{3}$  & 3 & No.11  \\ 
     $\smor_2$ & (1,1) & $\frac{16}{3}$  & 4 & No.10  \\ 
     $\smor_3$ & (3,0) & 5  & 3 & No.8  \\ 
     $\smor_4$ & (3,1) & 5  & 3 & No.3  \\ 
     $\smor_5$ & (2,1) & $\frac{14}{3}$  & 4 & No.6  \\ 
\rowcolor[gray]{0.9}     $\smor_6$ & (2,1) & $\frac{14}{3}$  & 4 & -  \\ 
     $\smor_7$ & (2,2) & $\frac{14}{3}$  & 4 & No.2  \\ 
\rowcolor[gray]{0.9}     $\smor_8$ & (3,1) & 4  & 4 & -  \\    \hline
  \end{tabular}
\end{table}
} 

If $\varphi$ is of type $\smor_6$ (resp. $\smor_8$), then we can find a $\smor_4$-line pair (resp. $\smor_4$-line pair) on $T_1$.
This contradicts Lemma \ref{T_1_dir}.
We cannot eliminate the other cases.

\noindent{\bf Case 3 :  $T_{min} = \pr^2$} 

By Lemma \ref{P^2}, we see that $\varphi$ is of type $\smor_3$, $\smor_4$, $\smor_6$, $\smor_7$ or $\smor_8$.
Thus cadidates are the following:
{\small \rm
\begin{table}[htb]
\renewcommand{\arraystretch}{1.4}
  \begin{tabular}{|c|c|c|c|c|c|c|}  \hline
  directed seq.  & \multicolumn{1}{c|}{($n_3, n_4$)} & \multicolumn{1}{c|}{$(-K_{T_1})^2$} & $\rho (T_1)$ & in Table \ref{T_1_table}   \\ \hline \hline
\rowcolor[gray]{0.9}    $\smor_3$ & (2,0) & $\frac{17}{3}$ & 3 & -  \\
\rowcolor[gray]{0.9}    $\smor_4$ & (2,1) & $\frac{17}{3}$ & 3 &- \\
\rowcolor[gray]{0.9}    $\smor_6$ & (1,1) & $\frac{16}{3}$ & 4  & - \\
\rowcolor[gray]{0.9}    $\smor_7$ & (1,2) & $\frac{16}{3}$ & 4  & - \\
\rowcolor[gray]{0.9}    $\smor_8$ & (2,1) & $\frac{14}{3}$ & 4  & - \\  \hline
  \end{tabular}
\end{table}
} \\
In this case, we can eliminate all the possibilities.
If $\varphi$ is of type $\smor_3$ (resp. $\smor_4$, $\smor_6$, $\smor_7$, $\smor_8$), then we can find a $\smor_1$-line pair (resp. $\smor_3$, $\smor_2$, $\smor_6$, $\smor_5$-line pair) on $T_1$.
This contradicts Lemma \ref{T_1_dir}.

\noindent{\bf Case 4 :  $T_{min} = \PP$} 

Candidates are the following:
{\small \rm
\begin{table}[htb]
\renewcommand{\arraystretch}{1.4}
  \begin{tabular}{|c|c|c|c|c|c|c|c|}  \hline
directed seq.   & \multicolumn{1}{c|}{($n_3, n_4$)} & \multicolumn{1}{c|}{$(-K_{T_1})^2$} & $\rho (T_1)$ & in Table \ref{T_1_table}  \\ \hline \hline
      $\smor_1$ & (1,0) & $\frac{16}{3}$ & 4 & No.13 \\
      $\smor_2$ & (0,1) & 5 & 5 & No.12 \\ 
\rowcolor[gray]{0.9}  $\smor_3$ & (2,0) & $\frac{14}{3}$ & 4 & - \\ 
      $\smor_4$ & (2,1) & $\frac{14}{3}$ & 4 & No.7 \\ 
\rowcolor[gray]{0.9}      $\smor_5$ & (1,1) & $\frac{13}{3}$ & 5 & -  \\ 
\rowcolor[gray]{0.9} $\smor_6$ & (1,1) & $\frac{13}{3}$ & 5 & - \\ 
\rowcolor[gray]{0.9} $\smor_7$ & (1,2) & $\frac{13}{3}$ & 5 & -  \\ 
\rowcolor[gray]{0.9}  $\smor_8$ & (2,1) & $\frac{11}{3}$ & 5 & -  \\ \hline
  \end{tabular}
\end{table}
} \\
We cannot eliminate the cases that $\varphi$ is of type $\smor_1$, $\smor_2$ or $\smor_4$.
If $\varphi$ is of type $\smor_3$ (resp. $\smor_5$, $\smor_6$, $\smor_7$, $\smor_8$), then we can find a $\smor_1$-line pair (resp. $\smor_2$, $\smor_1$, $\smor_6$, $\smor_1$-line pair) on $T_1$.
This contradicts Lemma \ref{T_1_dir}.

\end{proof}

\subsubsection{Candidates of $T_2$}

We will determine candidates of $T_2$.
Then $T_2$ has a minimal directed $\smor$-sequence $T_2 \overset{\varphi}{\to} T_1 \to T_{min}$.
Note that the center of $\varphi$ is not on any quasi-lines on $T_1$ by Lemma \ref{sm_on_qline}.

\begin{lemma}\label{T_1_P}

Let $T_1$ be one in Table \ref{T_1_table}.
Set $Q(T_1) := \{ x \in T_1 \ | $ there exists a quasi-line $L$ such that $x \in L \} $.
If $T_1$ is one of No.1, No,4, No.5 and No.9 in Table \ref{T_1_table}, then any point on $T_1 \backslash (\Sing T_1 \cup Q(T_1) )$ satisfies the condition {\rm ({\bf P})}. 

\end{lemma}

\begin{proof}

The assertion follows from Lemma \ref{P(1,1,4)_P1}.

\end{proof}

\begin{lemma}\label{T_2_dir}

Assume that $T_2 \overset{\smor_{a_1}}{\to} T_1 \overset{\smor_{a_2}}{\to} T_{min}$ is a minimal directed $\smor$-sequence of $T_2$.
If there exists a second morphism $T_2 \overset{\smor_i}{\to} T$ where $i < a_1$, then $T$ is $\smor$-minimal.

\end{lemma}

\begin{proof}

Assume that $T$ is not $\smor$-minimal.
Then we can find a $\smor$-sequence $T_2 \overset{\smor_{k_1}}{\to} X_1 \overset{\smor_{k_2}}{\to} \cdots \overset{\smor_{k_l}}{\to} X_l$, where $k_1 \le i < a_1$ and $X_l$ is $\smor$-minimal.
This contradicts the fact that $T_2 \overset{\smor_{a_1}}{\to} T_1 \overset{\smor_{a_2}}{\to} T_{min}$ is a minimal directed $\smor$-sequence.

\end{proof}

\begin{proposition}\label{T_2}

Candidates of $T_2$ are the following:

{\small \rm
\begin{table}[htb]
\caption{Candidates of $T_2$}\label{T_2_table}
\renewcommand{\arraystretch}{1.4}
  \begin{tabular}{|c|c|c|c|c|c|c|c|}  \hline
    \multicolumn{1}{|l|}{No.} & \multicolumn{1}{c|}{$X_{min}$} & directed seq.
      & \multicolumn{1}{c|}{($n_3, n_4$)} & \multicolumn{1}{c|}{$(-K_{T_2})^2$} & $\rho (T_2)$ & in Table \ref{main_table} \\ \hline \hline
     1 & $\pr(1,1,4)$ & $\smor_8 \circ \smor_8$ & (4,3) & $\frac{1}{3}$ & 7 & No.2  \\  
     2 & $\pr(1,1,3)$ & $\smor_7 \circ \smor_4$ & (4,3) & $\frac{4}{3}$  & 6 & No.3   \\  \hline
     3 & $\pr(1,1,3)$ & $\smor_4 \circ \smor_4$ & (5,2) & $\frac{5}{3}$  & 5 & No.4   \\  \hline
     4 & $\pr(1,1,3)$ & $\smor_7 \circ \smor_5$ & (3,3) & 1  & 7 & No.5   \\  \hline
     5 & $\pr(1,1,3)$ & $\smor_7 \circ \smor_3$ & (4,2) & $\frac{4}{3}$  & 6 & No.6   \\
     6 & $\PP$ & $\smor_4 \circ \smor_4$ & (4,2) & $\frac{4}{3}$ & 6 & No.7  \\  \hline
     7 & $\pr(1,1,3)$ & $\smor_4 \circ \smor_3$ & (5,1) & $\frac{5}{3}$  & 5 & No.8   \\  \hline
     8 & $\pr(1,1,3)$ & $\smor_7 \circ \smor_2$ & (2,3) & $\frac{5}{3}$  & 7 & No.10   \\ \hline
     9 & $\pr(1,1,4)$ & $\smor_8 \circ \smor_1$ & (3,2) & 2 & 6  & No.11  \\  
     10 & $\pr(1,1,3)$ & $\smor_7 \circ \smor_1$ & (3,2) & 2  & 6  & No.12  \\     \hline
     11 & $\pr(1,1,3)$ & $\smor_4 \circ \smor_1$ & (4,1) & $\frac{7}{3}$ & 5 & No.13   \\    \hline
     12 & $\pr(1,1,3)$ & $\smor_3 \circ \smor_3$ & (5,0) & $\frac{5}{3}$  & 5 & No.14    \\   \hline
     13 & $\PP$ & $\smor_4 \circ \smor_2$ & (2,2) & $\frac{5}{3}$ & 7 & No.17  \\   \hline
     14 & $\PP$ & $\smor_4 \circ \smor_1$ & (3,1) & 2 & 6 & No.19   \\ 
  \end{tabular}
\end{table}
}

{\small \rm
\begin{table}[htb]
\renewcommand{\arraystretch}{1.4}
  \begin{tabular}{|c|c|c|c|c|c|c|c|} 
    \multicolumn{1}{|l|}{No.} & \multicolumn{1}{c|}{$X_{min}$} & directed seq.
      & \multicolumn{1}{c|}{($n_3, n_4$)} & \multicolumn{1}{c|}{$(-K_{T_2})^2$} & $\rho (T_2)$ & in Table \ref{main_table} \\ \hline \hline
     15 & $\pr(1,1,3)$ & $\smor_3 \circ \smor_1$ & (4,0) & $\frac{7}{3}$  & 5 & No.21   \\     \hline
     16 & $\pr(1,1,4)$ & $\smor_1 \circ \smor_1$ & (2,1) & $\frac{11}{3}$ & 5 & No.26    \\  \hline
     17 & $\PP$ & $\smor_2 \circ \smor_2$ & (0,2) & 2 & 8 & No.28   \\   \hline
     18 & $\PP$ & $\smor_2 \circ \smor_1$ & (1,1) & $\frac{7}{3}$ & 7 & No.31    \\    \hline
     19 & $\PP$ & $\smor_1 \circ \smor_1$ & (2,0) & $\frac{8}{3}$ & 6 & No.33  \\  \hline
  \end{tabular}
\end{table}
}

\end{proposition}

\begin{proof}

By Proposition \ref{T_1}, candidates of $T_1$ are 13 cases.
We consider candidates of $T_2$ for each candidate of $T_1$.

\noindent{\bf Case 1 : $T_1$ is of No.1 in Table \ref{T_1_table} i.e. $T_2 \overset{\varphi}{\to} T_1 \overset{\smor_8}{\to} \pr(1,1,4)$}

By Lemmas \ref{quasi-0} and \ref{T_1_P}, we see that the type of $\varphi$ is one of $\smor_1, \smor_3, \smor_6$ and $\smor_8$.
Candidates are the following:
{\small \rm
\begin{table}[htb]
\renewcommand{\arraystretch}{1.4}
  \begin{tabular}{|c|c|c|c|c|c|c|c|}  \hline
 \multicolumn{1}{|c|}{$X_{min}$} & directed seq.
      & \multicolumn{1}{c|}{($n_3, n_4$)} & \multicolumn{1}{c|}{$(-K_{T_2})^2$} & $\rho (T_2)$ & in Table \ref{T_2_table}  \\ \hline \hline
     $\pr(1,1,4)$ & $\smor_8 \circ \smor_1$ & (3,2) & 2 & 6 & No.9  \\  
\rowcolor[gray]{0.9}      $\pr(1,1,4)$ & $\smor_8 \circ \smor_3$ & (4,2) & $\frac{4}{3}$ & 6 & -   \\ 
\rowcolor[gray]{0.9}      $\pr(1,1,4)$ & $\smor_8 \circ \smor_6$ & (3,3) & 1 & 7 & -   \\ 
     $\pr(1,1,4)$ & $\smor_8 \circ \smor_8$ & (4,3) & $\frac{1}{3}$ & 7 & No.1  \\  \hline
  \end{tabular}
\end{table}
} \\
If $\varphi$ is of type $\smor_3$, then we can find a sequence $T_2 \overset{\smor_3}{\to} T_1 \overset{\smor_7}{\to} \pr(1,1,3)$.
This contradicts the minimal directedness.
If $\varphi$ is of type $\smor_6$, then we can find a $\smor_5$-line pair on $T_2$.
Consider a second morphism $T_2 \overset{\smor_5}{\to} T$ of type $\smor_5$.
By Lemma \ref{T_2_dir}, $T$ is $\smor$-minimal.
Here $T$ has two singular points of type $\frac{1}{3}(1,1)$ and two singular points of type $\frac{1}{4}(1,1)$.
This contradicts the classification of $\smor$-minimal surfaces in Proposition \ref{T_min}.
Thus $\varphi$ is not of type $\smor_6$.
We cannot eliminate the other cases.

\noindent{\bf Case 2 : $T_1$ is of No.2 in Table \ref{T_1_table} i.e. $T_2 \overset{\varphi}{\to} T_1 \overset{\smor_7}{\to} \pr(1,1,3)$}

By the minimal directedness, we see that $\varphi$ is of type $\smor_i$ where $i \le 7$.
Candidates are the following:
{\small \rm
\begin{table}[htb]
\renewcommand{\arraystretch}{1.4}
  \begin{tabular}{|c|c|c|c|c|c|c|c|}  \hline
 \multicolumn{1}{|c|}{$X_{min}$} & directed seq.
      & \multicolumn{1}{c|}{($n_3, n_4$)} & \multicolumn{1}{c|}{$(-K_{T_2})^2$} & $\rho (T_2)$ & in Table \ref{T_2_table}  \\ \hline \hline
 $\pr(1,1,3)$ & $\smor_7 \circ \smor_1$ & (3,2) & 2  & 6 & No.10 \\  
 $\pr(1,1,3)$ & $\smor_7 \circ \smor_2$ & (2,3) & $\frac{5}{3}$  & 7 & No.8  \\
 $\pr(1,1,3)$ & $\smor_7 \circ \smor_3$ & (4,2) & $\frac{4}{3}$  & 6 & No.5  \\
 $\pr(1,1,3)$ & $\smor_7 \circ \smor_4$ & (4,3) & $\frac{4}{3}$  & 6 & No.2  \\
  \end{tabular}
\end{table}
} 

{\small \rm
\begin{table}[htb]
\renewcommand{\arraystretch}{1.4}
  \begin{tabular}{|c|c|c|c|c|c|c|c|}  
 \multicolumn{1}{|c|}{$X_{min}$} & directed seq.
      & \multicolumn{1}{c|}{($n_3, n_4$)} & \multicolumn{1}{c|}{$(-K_{T_2})^2$} & $\rho (T_2)$ & in Table \ref{T_2_table}  \\ \hline \hline
 $\pr(1,1,3)$ & $\smor_7 \circ \smor_5$ & (3,3) & 1  & 7 & No.4  \\  
\rowcolor[gray]{0.9} $\pr(1,1,3)$ & $\smor_7 \circ \smor_6$ & (3,3) & 1  & 7 & -  \\ 
\rowcolor[gray]{0.9} $\pr(1,1,3)$ & $\smor_7 \circ \smor_7$ & (3,4) & 1  & 7 & -  \\  \hline
  \end{tabular}
\end{table}
} 
If $\varphi$ is of type $\smor_6$ or $\smor_7$, then we can find a $T$-line on $T_2$.
This is a contradiction.
We cannot eliminate the other cases.

\noindent{\bf Case 3 : $T_1$ is of No.3 in Table \ref{T_1_table} i.e. $T_2 \overset{\varphi}{\to} T_1 \overset{\smor_4}{\to} \pr(1,1,3)$}

By the minimal directedness, we see that $\varphi$ is of type $\smor_i$ where $i \le 4$.
Candidates are the following:

{\small \rm
\begin{table}[htb]
\renewcommand{\arraystretch}{1.4}
  \begin{tabular}{|c|c|c|c|c|c|c|c|}  \hline
 \multicolumn{1}{|c|}{$X_{min}$} & directed seq.
      & \multicolumn{1}{c|}{($n_3, n_4$)} & \multicolumn{1}{c|}{$(-K_{T_2})^2$} & $\rho (T_2)$ & in Table \ref{T_2_table}  \\ \hline \hline
 $\pr(1,1,3)$ & $\smor_4 \circ \smor_1$ & (4,1) & $\frac{7}{3}$ & 5 & No.11  \\ 
\rowcolor[gray]{0.9} $\pr(1,1,3)$ & $\smor_4 \circ \smor_2$ & (3,2) & 2  & 6 & -  \\ 
 $\pr(1,1,3)$ & $\smor_4 \circ \smor_3$ & (5,1) & $\frac{5}{3}$  & 5 & No.7  \\  
 $\pr(1,1,3)$ & $\smor_4 \circ \smor_4$ & (5,2) & $\frac{5}{3}$  & 5 & No.3  \\  \hline
  \end{tabular}
\end{table}
} 
If $\varphi$ is of type $\smor_2$, then we can find a $\smor_1$-line pair.
Then we can obtain a contradiction in the same way as in Case 1.
We cannot eliminate the other cases.

\noindent{\bf Case 4 : $T_1$ is of No.4 in Table \ref{T_1_table} i.e. $T_2 \overset{\varphi}{\to} T_1 \overset{\smor_6}{\to} \pr(1,1,4)$}

By Lemma \ref{quasi-0}, Lemma \ref{T_1_P} and the minimal directedness, we see that the type of $\varphi$ is one of $\smor_1, \smor_3$ and $\smor_6$.
Candidates are the following:
{\small \rm
\begin{table}[htb]
\renewcommand{\arraystretch}{1.4}
  \begin{tabular}{|c|c|c|c|c|c|c|c|}  \hline
 \multicolumn{1}{|c|}{$X_{min}$} & directed seq.
      & \multicolumn{1}{c|}{($n_3, n_4$)} & \multicolumn{1}{c|}{$(-K_{T_2})^2$} & $\rho (T_2)$ & in Table \ref{T_2_table}  \\ \hline \hline
\rowcolor[gray]{0.9} $\pr(1,1,4)$ & $\smor_6 \circ \smor_1$ & (2,2) & $\frac{8}{3}$ & 6 & - \\
\rowcolor[gray]{0.9} $\pr(1,1,4)$ & $\smor_6 \circ \smor_3$ & (3,2) & 2 & 6 & - \\
\rowcolor[gray]{0.9} $\pr(1,1,4)$ & $\smor_6 \circ \smor_6$ & (2,3) & $\frac{5}{3}$ & 7 & - \\  \hline
  \end{tabular}
\end{table}
} \\
In this case, we can eliminate all the possibilities.
If $\varphi$ is of type $\smor_1$, then we can find some floating $(-1)$-curves on $T_2$.
This is a contradiction to assumption.
If $\varphi$ is of type $\smor_3$ (resp. $\smor_6$), then we can find a $\smor_1$-line pair (resp. $\smor_2$-line pair).
Then we can obtain a contradiction in the same way as in Case 1 respectively.
Hence this case is impossible.

\noindent{\bf Case 5 : $T_1$ is of No.5 in Table \ref{T_1_table} i.e. $T_2 \overset{\varphi}{\to} T_1 \overset{\smor_3}{\to} \pr(1,1,4)$}

By Lemma \ref{quasi-0}, Lemma \ref{T_1_P} and the minimal directedness, we see that the type of $\varphi$ is one of $\smor_1$ and $\smor_3$.
Candidates are the following:
{\small \rm
\begin{table}[htb]
\renewcommand{\arraystretch}{1.4}
  \begin{tabular}{|c|c|c|c|c|c|c|c|}  \hline
 \multicolumn{1}{|c|}{$X_{min}$} & directed seq.
      & \multicolumn{1}{c|}{($n_3, n_4$)} & \multicolumn{1}{c|}{$(-K_{T_2})^2$} & $\rho (T_2)$ & in Table \ref{T_2_table}  \\ \hline \hline
\rowcolor[gray]{0.9} $\pr(1,1,4)$ & $\smor_3 \circ \smor_1$ & (3,1) & 3 & 5 & - \\
\rowcolor[gray]{0.9} $\pr(1,1,4)$ & $\smor_3 \circ \smor_3$ & (4,1) & $\frac{7}{3}$ & 5 & - \\  \hline
  \end{tabular}
\end{table}
} \\
In this case, we also can eliminate all the possibilities.
If $\varphi$ is of type $\smor_1$, then we can find a foating $(-1)$-curve on $T_2$.
If $\varphi$ is of type $\smor_3$, then we can find a $\smor_1$-line pair.
Then we can obtain a contradiction in the same way as in Case 1.
Hence this case is impossible.

\noindent{\bf Case 6 : $T_1$ is of No.6 in Table \ref{T_1_table} i.e. $T_2 \overset{\varphi}{\to} T_1 \overset{\smor_5}{\to} \pr(1,1,3)$}

By the minimal directedness, we see that $\varphi$ is of type $\smor_i$ where $i \le 5$.
Candidates are the following:

{\small \rm
\begin{table}[htb]
\renewcommand{\arraystretch}{1.4}
  \begin{tabular}{|c|c|c|c|c|c|c|c|}  \hline
 \multicolumn{1}{|c|}{$X_{min}$} & directed seq.
      & \multicolumn{1}{c|}{($n_3, n_4$)} & \multicolumn{1}{c|}{$(-K_{T_2})^2$} & $\rho (T_2)$ & in Table \ref{T_2_table}  \\ \hline \hline
\rowcolor[gray]{0.9} $\pr(1,1,3)$ & $\smor_5 \circ \smor_1$ & (3,1) & 2  & 6 & -  \\ 
\rowcolor[gray]{0.9} $\pr(1,1,3)$ & $\smor_5 \circ \smor_2$ & (2,2) & $\frac{5}{3}$  & 7 & -  \\ 
\rowcolor[gray]{0.9} $\pr(1,1,3)$ & $\smor_5 \circ \smor_3$ & (4,1) & $\frac{4}{3}$  & 6 & - \\ 
\rowcolor[gray]{0.9} $\pr(1,1,3)$ & $\smor_5 \circ \smor_4$ & (4,2) & $\frac{4}{3}$  & 6 & - \\ 
\rowcolor[gray]{0.9} $\pr(1,1,3)$ & $\smor_5 \circ \smor_5$ & (3,2) & 1  & 7 & -  \\   \hline
  \end{tabular}
\end{table}
} 
In this case, we also can eliminate all the possibilities.
If $\varphi$ is of type $\smor_1$ or $\smor_2$, then we can find a foating $(-1)$-curve on $T_2$.
If $\varphi$ is of type $\smor_3$ (resp. $\smor_4, \smor_5$), then we can find a $\smor_1$-line pair (resp. $\smor_3, \smor_1$-line pair).
Then we can obtain a contradiction in the same way as in Case 1 respectively.
Hence this case is impossible.

\noindent{\bf Case 7 : $T_1$ is of No.7 in Table \ref{T_1_table} i.e. $T_2 \overset{\varphi}{\to} T_1 \overset{\smor_4}{\to} \PP$}

By the minimal directedness, we see that $\varphi$ is of type $\smor_i$ where $i \le 4$.
Candidates are the following:
{\small \rm
\begin{table}[htb]
\renewcommand{\arraystretch}{1.4}
  \begin{tabular}{|c|c|c|c|c|c|c|c|}  \hline
 \multicolumn{1}{|c|}{$X_{min}$} & directed seq.
      & \multicolumn{1}{c|}{($n_3, n_4$)} & \multicolumn{1}{c|}{$(-K_{T_2})^2$} & $\rho (T_2)$ & in Table \ref{T_2_table}  \\ \hline \hline
 $\PP$ & $\smor_4 \circ \smor_1$  & (3,1) & 2 & 6 & No.14 \\
 $\PP$ & $\smor_4 \circ \smor_2$ & (2,2) & $\frac{5}{3}$ & 7 & No.13 \\
\rowcolor[gray]{0.9} $\PP$ & $\smor_4 \circ \smor_3$ & (4,1) & $\frac{4}{3}$ & 6 & - \\
 $\PP$ & $\smor_4 \circ \smor_4$ & (4,2) & $\frac{4}{3}$ & 6 & No.6 \\  \hline
  \end{tabular}
\end{table}
} 

If $\varphi$ is of type $\smor_3$, then we can find a $\smor_1$-line pair on $T_2$.
Then we can obtain a contradiction in the same way as in Case 1.
We cannot eliminate the other cases.

\noindent{\bf Case 8 : $T_1$ is of No.8 in Table \ref{T_1_table} i.e. $T_2 \overset{\varphi}{\to} T_1 \overset{\smor_3}{\to} \pr(1,1,3)$}

By the minimal directedness, we see that $\varphi$ is of type $\smor_i$ where $i \le 3$.
Candidates are the following:
{\small \rm
\begin{table}[htb]
\renewcommand{\arraystretch}{1.4}
  \begin{tabular}{|c|c|c|c|c|c|c|c|}  \hline
 \multicolumn{1}{|c|}{$X_{min}$} & directed seq.
      & \multicolumn{1}{c|}{($n_3, n_4$)} & \multicolumn{1}{c|}{$(-K_{T_2})^2$} & $\rho (T_2)$ & in Table \ref{T_2_table}  \\ \hline \hline
 $\pr(1,1,3)$ & $\smor_3 \circ \smor_1$ & (4,0) & $\frac{7}{3}$  & 5 & No.15  \\  
\rowcolor[gray]{0.9} $\pr(1,1,3)$ & $\smor_3 \circ \smor_2$  & (3,1) & 2  & 6 & -   \\ 
 $\pr(1,1,3)$ & $\smor_3 \circ \smor_3$ & (5,0) & $\frac{5}{3}$  & 5 & No.12   \\   \hline
  \end{tabular}
\end{table}
} \\
If $\varphi$ is of type $\smor_2$, then we can find a $\smor_1$-line pair on $T_2$.
Then we can obtain a contradiction in the same way as in Case 1.
We cannot eliminate the other cases.

\noindent{\bf Case 9 : $T_1$ is of No.9 in Table \ref{T_1_table} i.e. $T_2 \overset{\varphi}{\to} T_1 \overset{\smor_1}{\to} \pr(1,1,4)$}

By the minimal directedness, we see that $\varphi$ is of type $\smor_1$.
The candidate is the following:
{\small \rm
\begin{table}[htb]
\renewcommand{\arraystretch}{1.4}
  \begin{tabular}{|c|c|c|c|c|c|c|c|}  \hline
 \multicolumn{1}{|c|}{$X_{min}$} & directed seq.
      & \multicolumn{1}{c|}{($n_3, n_4$)} & \multicolumn{1}{c|}{$(-K_{T_2})^2$} & $\rho (T_2)$ & in Table \ref{T_2_table}  \\ \hline \hline
 $\pr(1,1,4)$ & $\smor_1 \circ \smor_1$ & (2,1) & $\frac{11}{3}$ & 5  & No.16  \\  \hline
  \end{tabular}
\end{table}
} \\
This case cannot be eliminated.

\noindent{\bf Case 10 : $T_1$ is of No.10 in Table \ref{T_1_table} i.e. $T_2 \overset{\varphi}{\to} T_1 \overset{\smor_2}{\to} \pr(1,1,3)$}

By the minimal directedness, we see that $\varphi$ is of type $\smor_i$ where $i \le 2$.
Candidates are the following:
{\small \rm
\begin{table}[htb]
\renewcommand{\arraystretch}{1.4}
  \begin{tabular}{|c|c|c|c|c|c|c|c|}  \hline
 \multicolumn{1}{|c|}{$X_{min}$} & directed seq.
      & \multicolumn{1}{c|}{($n_3, n_4$)} & \multicolumn{1}{c|}{$(-K_{T_2})^2$} & $\rho (T_2)$ & in Table \ref{T_2_table}  \\ \hline \hline
\rowcolor[gray]{0.9} $\pr(1,1,3)$ & $\smor_2  \circ \smor_1$ & (2,1) & $\frac{8}{3}$  & 6 & -  \\ 
\rowcolor[gray]{0.9} $\pr(1,1,3)$ & $\smor_2 \circ \smor_2$ & (1,2) & $\frac{7}{3}$  & 7 & -  \\   \hline
  \end{tabular}
\end{table}
} \\
In this case, we also can eliminate all the possibilities.
If $\varphi$ is of type $\smor_1$ or $\smor_2$, then we can find a foating $(-1)$-curve on $T_2$.
This is a contradiction.

\noindent{\bf Case 11 : $T_1$ is of No.11 in Table \ref{T_1_table} i.e. $T_2 \overset{\varphi}{\to} T_1 \overset{\smor_1}{\to} \pr(1,1,3)$}

By the minimal directedness, we see that $\varphi$ is of type $\smor_1$.
The candidate is the following:
{\small \rm
\begin{table}[htb]
\renewcommand{\arraystretch}{1.4}
  \begin{tabular}{|c|c|c|c|c|c|c|c|}  \hline
 \multicolumn{1}{|c|}{$X_{min}$} & directed seq.
      & \multicolumn{1}{c|}{($n_3, n_4$)} & \multicolumn{1}{c|}{$(-K_{T_2})^2$} & $\rho (T_2)$ & in Table \ref{T_2_table}  \\ \hline \hline
\rowcolor[gray]{0.9} $\pr(1,1,3)$ & $\smor_1 \circ \smor_1$ & (3,0) & 3  & 5 & -  \\  \hline
  \end{tabular}
\end{table}
} \\
We can find a floating $(-1)$-curve on $T_2$.
This is a contradiction.

\noindent{\bf Case 12 : $T_1$ is of No.12 in Table \ref{T_1_table} i.e. $T_2 \overset{\varphi}{\to} T_1 \overset{\smor_2}{\to} \PP$}

By the minimal directedness, we see that $\varphi$ is of type $\smor_i$ where $i \le 2$.
Candidates are the following:
{\small \rm
\begin{table}[htb]
\renewcommand{\arraystretch}{1.4}
  \begin{tabular}{|c|c|c|c|c|c|c|c|}  \hline
 \multicolumn{1}{|c|}{$X_{min}$} & directed seq.
      & \multicolumn{1}{c|}{($n_3, n_4$)} & \multicolumn{1}{c|}{$(-K_{T_2})^2$} & $\rho (T_2)$ & in Table \ref{T_2_table}  \\ \hline \hline
 $\PP$ & $\smor_2 \circ \smor_1$ & (1,1) & $\frac{7}{3}$ & 7 & No.18 \\ 
 $\PP$ & $\smor_2 \circ \smor_2$ & (0,2) & 2 & 8 & No.17 \\   \hline
  \end{tabular}
\end{table}
} \\
Both cases cannot be eliminated.

\noindent{\bf Case 13 : $T_1$ is of No.13 in Table \ref{T_1_table} i.e. $T_2 \overset{\varphi}{\to} T_1 \overset{\smor_1}{\to} \PP$}

By the minimal directedness, we see that $\varphi$ is of type $\smor_1$.
Candidates are the following:

{\small \rm
\begin{table}[htb]
\renewcommand{\arraystretch}{1.4}
  \begin{tabular}{|c|c|c|c|c|c|c|c|}  \hline
 \multicolumn{1}{|c|}{$X_{min}$} & directed seq.
      & \multicolumn{1}{c|}{($n_3, n_4$)} & \multicolumn{1}{c|}{$(-K_{T_2})^2$} & $\rho (T_2)$ & in Table \ref{T_2_table}  \\ \hline \hline
 $\PP$ & $\smor_1 \circ \smor_1$ & (2,0) & $\frac{8}{3}$ & 6 & No.19 \\  \hline
  \end{tabular}
\end{table}
} 
This case also cannot be eliminated.

\end{proof}

\subsubsection{Candidates of $T_3$}

We determine the candidates of $T_3$ and show that there is no example of $T_m$ where $m \ge 4$.

\begin{proposition}\label{T_3}
The candidate of $T_3$ is the following:

{\small \rm
\begin{table}[htb]
\caption{Candidates of $T_3$}\label{T_3_table}
\renewcommand{\arraystretch}{1.4}
  \begin{tabular}{|c|c|c|c|c|c|c|c|}  \hline
    \multicolumn{1}{|l|}{No.} & \multicolumn{1}{c|}{$X_{min}$} & directed seq.
      & \multicolumn{1}{c|}{($n_3, n_4$)} & \multicolumn{1}{c|}{$(-K_{T_3})^2$} & $\rho (T_3)$ & in Table \ref{main_table}  \\ \hline \hline
     1 & $\pr(1,1,4)$ & $\smor_1 \circ \smor_1 \circ \smor_1$ & (3,1) & 1 & 7 & No.20  \\   \hline
  \end{tabular}
\end{table}
}

\end{proposition}

\begin{proof}

Observing anti-canonical volumes, we see that the candidate of $T_3$ is only the surface of No.16 in Table \ref{T_2_table}.
By the minimal directedness, there is only one possibility of $\varphi$.

\end{proof}

Observing anti-canonical volumes, we also see the following corollary.

\begin{corollary}\label{T_>4}

For $m \ge 4$, there is no a $\smor$-sequence $T_m \to T_{m-1} \to \cdots \to T_1 \to T_{min}$. 

\end{corollary}

\section{Constructions of del Pezzo surfaces}\label{construct_sec}

In this section, we construct each candidate on Table \ref{main_table} and check the ampleness of anti-canonical divisors.

\subsection{Reduction to some cases}
Let $f : X \to X_1$ be a composition of birational extremal contractions.
By Lemma \ref{keep_ldP}, we see that if $X$ is a del Pezzo surface, then $X_1$ is also a del Pezzo surface.
Moreover, we also see that if $X$ has no floating $(-1)$-curves, then $X_1$ also has no floating $(-1)$-curves by Corollary \ref{float}.

\begin{notation}

We prepare notation for Table \ref{reduction_table}.
Recall that a second morphism is a composition of several birational extremal contractions in Table \ref{bir_ext_cont}.
Denote a contraction $\varphi$ of type $\sB_5$ by $\smor_4^7$ (resp. $\smor_1^2$, $\smor_3^5$), where $\smor_4 \circ \varphi = \smor_7$ (resp. $\smor_1 \circ \varphi = \smor_2$, $\smor_3 \circ \varphi = \smor_5$).

\end{notation}

\begin{proposition}\label{reduction}

If del Pezzo surfaces with no floating $(-1)$-curves of No.1, 2, 3, 5, 7, 10, 11, 17, 20, 22 and 28 in Table \ref{main_table} exist, then all surfaces in Table \ref{main_table} really exist.

\end{proposition}

\begin{proof}
The existence of smooth cases is well-known.
The existence of del Pezzo surfaces with at most $\frac{1}{3}(1,1)$-singularity is also shown in \cite{CH17}.
By Lemma \ref{unique_1/n}, we see that a surface of No.34 is $\pr(1,1,4)$.
The relations of the other cases are the following:
{\footnotesize \rm
\begin{table}[htb]
\caption{The relations of reductions}\label{reduction_table}
\renewcommand{\arraystretch}{1.4}
  \begin{tabular}{|c|c||c|c||c|c|}  \hline
    \multicolumn{1}{|l|}{No.} &   \multicolumn{1}{c||}{how to obtain} & No.
     & \multicolumn{1}{c|}{how to obtain} & No.     & \multicolumn{1}{c|}{how to obtain}   \\ \hline \hline
     1 &  - &                    	         	14 &  known in \cite{CH17}  &	 	27 &  known in \cite{CH17}   	\\ \rowcolor[gray]{0.9}  
     2 &  - &                          		15 &  $\smor_8$ from No.2 &  		28 &  - 				\\   
     3 &  - &                              		16 &  $\smor_4$ from No.3 &		29 &  $\smor_8$ from No.11	\\ \rowcolor[gray]{0.9} 
     4 &  $\smor_4^7$ from No.3  &  		17 & - & 					30 &  $\smor_7$ from No.10	\\ 
     5 &  - &                                        	18 & $\smor_7$ from No.3 &		31 & $\smor_1^2$ from No.28  	\\  \rowcolor[gray]{0.9}
     6 & $\smor_3^5$ from No.5 &		19 &  $\smor_1^2$ from No.17  &	32 &  known in \cite{CH17}	\\ 
     7 &  - &						20 & - &					33 &  known in \cite{CH17} 	\\ \rowcolor[gray]{0.9}
     8 &  $\smor_4^7$ from No.6  &		21 &  known in \cite{CH17}	&	34 &  well known 			\\ 
     9 &  known in \cite{CH17}  &		22 &  - &					35 & $\smor_4$ from No.17	\\  \rowcolor[gray]{0.9}
     10 &  - & 					23 &  $\smor_8$ from No.6	&	36 &  known in \cite{CH17}	\\ 
     11 &  - & 					24 &  $\smor_7$ from No.5	&	37 &  known in \cite{CH17}  	\\ \rowcolor[gray]{0.9}
     12 &  $\smor_1^2$ from No.10  &		25 &  $\smor_4$ from No.7	&	38 &  well known 			\\   
     13 &  $\smor_4^7$ from No.12 &		26 &  $\smor_1$ from No.20	&	39 &  well known  	   		\\ \hline	
  \end{tabular}
\end{table}
}

By this table, we obtain the assertion.

\end{proof}

\begin{lemma}\label{red_float}

Let $X$ be a del Pezzo surface of type $\tA$.
Assume that $X$ is one of No.1, 3, 5, 7, 10, 11, 22 or 28.
Then $X$ has no floating $(-1)$-curves.

\end{lemma}

\begin{proof}

Assume that $X$ has a floating $(-1)$-curve $C$ by contradiction.
Then we have a sequence of first morphisms $X \to  X_1 \to \cdots \to S$ such that $S$ has no floating $(-1)$-curves.
By Lemma \ref{keep_ldP}, we see that $S$ is also a del Pezzo surface of type $\tA$.
We also see that the numbers of singular points on $X$ and $S$ are equal and $(-K_S)^2 > (-K_X)^2$.
We, however, know that $S$ must be in Table \ref{main_table} by Theorem \ref{candi_S}.
This is a contradiction if $X$ is one of No.1, 3, 5, 7, 10, 11, 22 or 28.

\end{proof}

By this lemma, it is enough to show that surfaces of No.2, 17 and 20 do not have any floating $(-1)$-curves.

\subsection{Construction}

In this subsection, we confirm the existences of the eleven cases in Proposition \ref{reduction}.

\begin{notation}
We first prepare notation.
In $\mathbb{F}_0 := \PP$, fix two distinct fibers of $\pi_1$ as $l_1, l_2$ and two distinct fibers  of $\pi_2$ as $l_3, l_4$. 
Then we know that $-K_{\mathbb{F}_0} \sim l_1 + l_2 + l_3 + l_4$.
Denote a del Pezzo surface of degree six by $S_6$.
We also denote the $(-1)$-curves on $S_6$ by $l_1, \ldots , l_6$. 
We see that $(l_1 \cup l_3 \cup l_5) \cap (l_2 \cup l_4 \cup l_6)$ is a set of six points and call them {\it the six points on $S_6$}.
We know that $-K_{S_6} \sim l_1 + l_2 + l_3 + l_4 + l_5 + l_6$. \\

\end{notation}

\noindent{\bf No.1}

Let $X$ be a del Pezzo surface of No.1.
Let $Y \to X$ be the minimal resolution.
By observing the configuration of negative curves on $Y$, we can find the following blow-downs $\alpha_1$, $\alpha_2$:
\begin{figure}[htbp]
\centering\includegraphics[width=13.5cm, bb=0 0 1440 578]{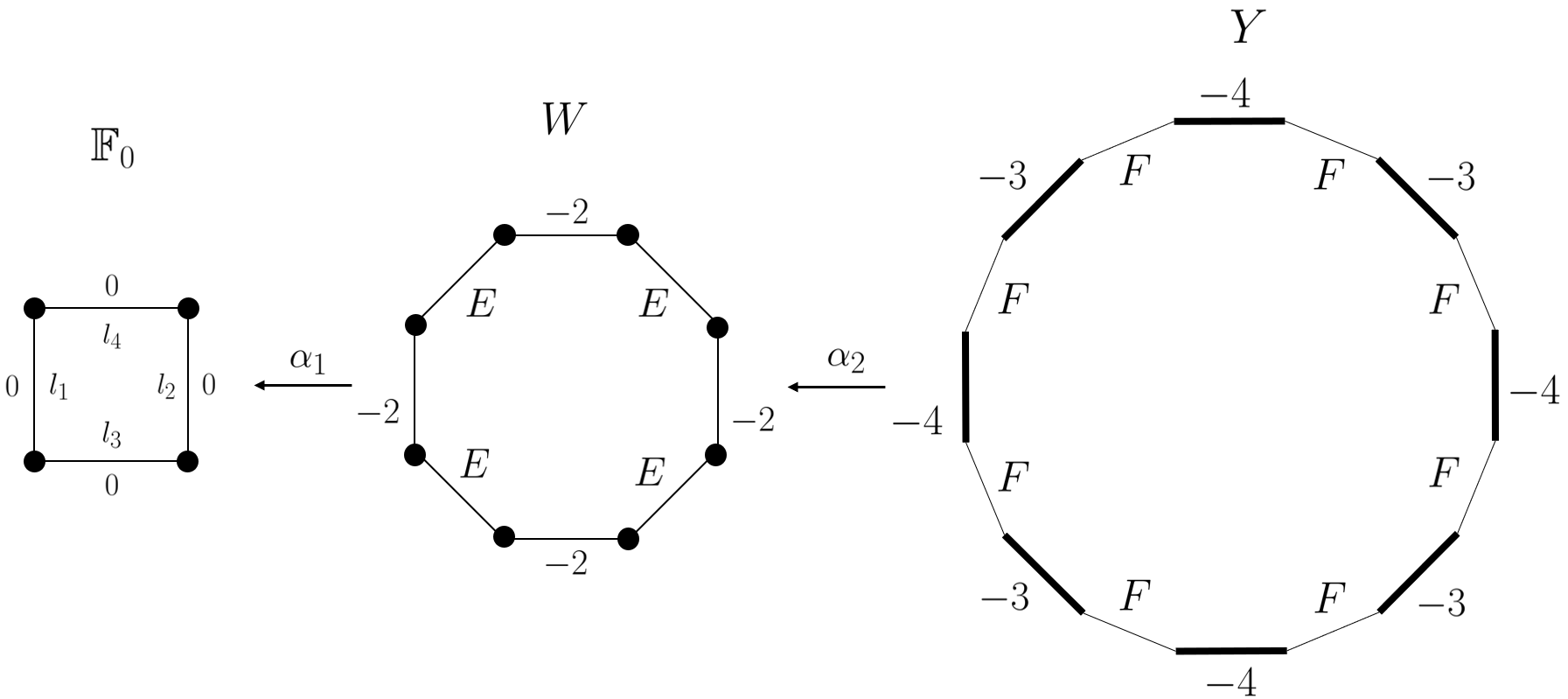}
\end{figure}

Let us construct an example of No.1.
Let $-K_{\mathbb{F}_0} \sim l_1 + l_2 + l_3 + l_4 =: L$.
Let $\alpha_1 : W \to \mathbb{F}_0$ be the blow-up at $(l_1 \cup l_2) \cap (l_3 \cup l_4)$.
Denote the exceptional divisor of $\alpha_1$ by $E$.
Then we have
\[
 -K_W = \alpha_1^* (-K_{\mathbb{F}_0}) - E  = \alpha_1^* L - E  = L_W + E . 
\]

Let $\alpha_2 : Y \to W$ be the blow-up at $E \cap L_W$.
Denote the exceptional divisor of $\alpha_2$ by $F$. 
Then we have
\begin{eqnarray*}
 -K_Y &=& \alpha_2^* (-K_{W}) - F  \\
  &=& \alpha_2^* (L_W + E) - F  \\
  &=& L_Y + E_Y + F . 
\end{eqnarray*} 
Let $f : Y \to X$ be the contraction of $L_Y$ and $E_Y$.
Then we have
\[
K_Y = f^*K_X - \frac{1}{2}L_Y - \frac{1}{3}E_Y .
\]
by the above construction.
Hence we obtain the following relation;
\[
f^*(-K_X) = \frac{1}{2}L_Y + \frac{2}{3}E_Y + F  .
\]

\begin{claim}\label{K^2_nef}

$(-K_X)^2 = \frac{4}{3}$ and $-K_X$ is nef.

\end{claim}

\begin{proof}
By the previous equation, we have
\[
-K_Y \cdot f^* (-K_X) = -K_Y \cdot \frac{1}{2}L_Y + (-K_Y) \cdot \frac{2}{3}E_Y + (-K_Y) \cdot F.
\]
Since we see that $L_Y$ is a sum of four $(-4)$-curves, $E_Y$ is a sum of four $(-3)$-curves and $F$ is a sum of eight $(-1)$-curves, we have
\[
(-K_X)^2 = \frac{1}{2} \cdot (-2 \cdot 4) + \frac{2}{3} \cdot (-1 \cdot 4) + 1 \cdot 8 = \frac{4}{3}. 
\]
Let $C$ be an irreducible curve on $X$.
Since $-K_X \cdot C = f^* (-K_X) \cdot C_Y$, it is enough to show that $f^* (-K_X)$ is nef.
Let $D \subset Y$ be an irreducible curve.
We see that $f^* (-K_X) \cdot D = 0$ if $D \subset L_Y \cup E_Y$.
If $D \subset F$, then we have  $f^* (-K_X) \cdot D = \frac{1}{2} \cdot 1 + \frac{2}{3} \cdot 1 + 1 \cdot (-1) = \frac{1}{6} \ge 0$.
In the other cases, since $f^* (-K_X)$ is effective, we see that $f^* (-K_X) \cdot D \ge 0$.
Hence $f^* (-K_X)$ is nef.

\end{proof}

\begin{claim}\label{exi_1}

$X$ is a del Pezzo surface, that is, $-K_X$ is ample.

\end{claim}

\begin{proof}

Let $C$ be an irreducible curve on $X$.
By the Nakai-Moishezon criterion, it suffices to show that $-K_X \cdot C > 0$ since we see that $(-K_X)^2 = \frac{4}{3} > 0$.

Assume that $-K_X \cdot C = 0$ by contradiction.
We see that $C_Y \not\subset L_Y \cup E_Y$ by definition.
If $C_Y \subset F$, then we see that $-K_X \cdot C = f^* (-K_X) \cdot C_Y = \frac{1}{6}$.
This is a contradiction.
Hence we see that $C_Y \not\subset L_Y \cup E_Y \cup F$.
Thus by assumption, we see $L_Y \cdot C_Y = E_Y \cdot C_Y = F \cdot C_Y = 0$.
Hence ${\alpha}_* C_Y $ is an irreducible curve on ${\mathbb{F}_0}$ and ${\alpha}^* {\alpha}_* C_Y = C_Y $, where $\alpha := \alpha_1 \circ \alpha_2$.
Since ${\alpha}_* f^* (-K_X) = -\frac{1}{2} K_{\mathbb{F}_0} $, we see that $-K_X \cdot C = f^*(-K_X) \cdot C_Y = -\frac{1}{2} K_{\mathbb{F}_0} \cdot {\alpha}_* C_Y > 0$.
This is a contradiction.
Thus we see that $X$ is a del Pezzo surface. 

\end{proof}

\noindent From this construction, it follows that $X$ is a del Pezzo surface such that $(-K_X)^2 = \frac{4}{3}$ and $(n_3, n_4) = (4,3)$.
If $X$ has some floating $(-1)$-curves, we obtain a contradiction as in the proof of Lemma \ref{red_float}.
Thus we see that $X$ has no floating $(-1)$-curves.
Hence $X$ is of No.1. \\

\noindent{\bf No.2}

Let $X$ be a del Pezzo surface of No.2.
Let $Y \to X$ be the minimal resolution.
By observing the configuration of negative curves on $Y$, we can find the following blow-downs $\alpha_1$, $\alpha_2$:
\begin{figure}[htbp]
\centering\includegraphics[width=13.5cm, bb=0 0 1450 568]{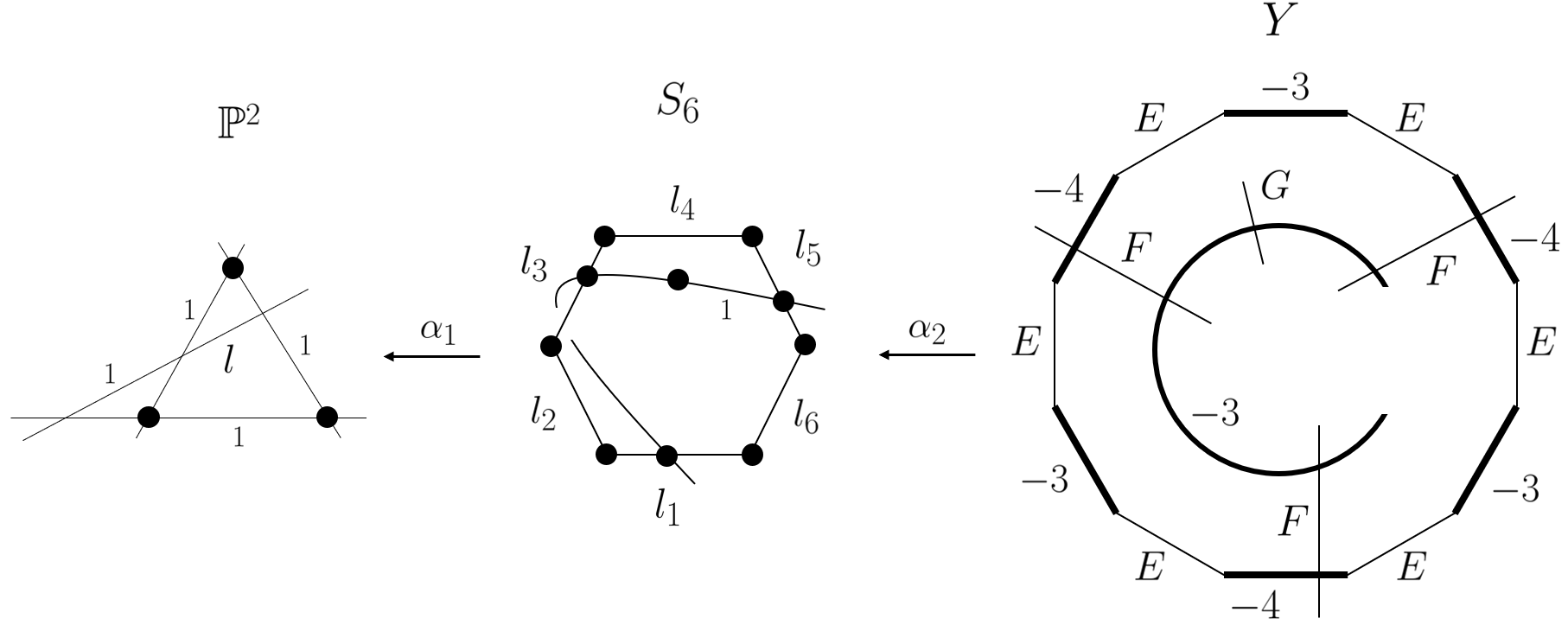}
\end{figure}

Let us construct an example of No.2.
In $\pr^2$, take non-collinear distinct three points $P_1, P_2, P_3$ and a line $l$ which does not pass through them.
Let $\alpha_1 : S_6 \to \pr^2$ be the blow-up at the three points.

Set $L := l_1 + l_3 + l_5$ and $M := l_2 + l_4 + l_6$.
We may assume that $L \cdot l_{S_6} = 3$ and $M \cdot l_{S_6} = 0$.
Take a general point $P$ on $l_{S_6}$.
Let $\alpha_2 : Y \to S_6$ be the blow-up at ten points, $((l_1 \cup l_3 \cup l_5) \cap (l_2 \cup l_4 \cup l_6 \cup l_{S_6})) \cup \{ P \}$.
Let $E$ denote the exceptional divisor over $(l_1 \cup l_3 \cup l_5) \cap (l_2 \cup l_4 \cup l_6)$, $F$ the exceptinal divisor over $(l_1 \cup l_3 \cup l_5) \cap l_{S_6}$ and $G$ the exceptional divisor over $P$.
We also have $l_{S_6} \sim \frac{1}{3} L + \frac{2}{3} M$.
Then we have
\begin{eqnarray*}
 -K_Y &=& \alpha_2^* (-K_{S_6}) - E - F - G  \\
  &=& \alpha_2^* (\frac{2}{3} L + \frac{1}{3} M + l_{S_6}) - E - F - G  \\
  &=& \frac{2}{3} (L_Y + E + F) + \frac{1}{3} (M_Y + E) + (l_Y + F + G)  \\
  & & - E - F - G \\
  &=& \frac{2}{3} L_Y + \frac{1}{3} M_Y  + l_Y + \frac{2}{3} F   .
\end{eqnarray*} 
Let $f : Y \to X$ be the contraction of $L_Y$, $M_Y$ and $l_Y$.
Then we also have
\[
K_Y = f^* K_X - \frac{1}{2} L_Y - \frac{1}{3} M_Y - \frac{1}{3} l_Y  .
\]
Hence we obtain the following relation;
\[
f^* (-K_X) = \frac{1}{6} L_Y + \frac{2}{3} l_Y + \frac{2}{3} F  .
\]

\begin{claim}

$-K_X$ is ample.

\end{claim}

\begin{proof}

We see that $(-K_X)^2 = \frac{1}{3} > 0$ and $-K_X$ is nef as in Claim \ref{K^2_nef}.
Let $C$ be an irreducible curve on $X$.
Assume that $-K_X \cdot C = 0$ by contradiction. 
Set $\alpha_* C_Y \sim dl$, where $\alpha := \alpha_1 \circ \alpha_2$.
By assumption, we see that $C_Y \cdot F = 0$.
Thus we have
\[
\alpha^* \alpha_* C_Y = C_Y + \sum^6_{i=1} a_iE_{i} + bG + \sum^3_{j=1}c_jM_{Y,j}  ,
\]
where $E_{i}$ and $M_{Y,j}$ are irreducible components of $E$ and $M_Y$ respectively.
Here we also see that $l_Y \cdot C_Y = 0$ and $l_Y \cdot E = 0$ by assumption.
Thus we have
\begin{equation}\label{no2}
l_Y \cdot \alpha^* \alpha_* C_Y = \alpha_* l_Y \cdot \alpha_* C_Y = l \cdot dl = d ,
\end{equation}
and
\[
l_Y \cdot M_Y = (\alpha_2^* l_{S_6} - E - F ) \cdot M_Y = l_{S_6} \cdot M + 0 + 0 = 0 .
\]
Moreover, since we see that $l_Y \cdot G = 1$, we obtain $d = b$ by calculating $l_Y \cdot (\ref{no2})$.

$\Gamma := \alpha_* C_Y$ is an irreducible curve of degree $d$ in $\pr^2$ and it passes through the center of $G$ $d$ times.
By these facts, we can conclude that $d=1$ and $\Gamma$ is a line.
Therefore, $C_Y$ is the strict transform of a line on $\pr^2$ passing through the center of $G$.
We see that $L_Y \cdot C_Y = 0$, $L_Y \cdot G = 0$ and $L_Y \cdot M_Y = 0$.
We also have
\[
L_Y \cdot \alpha^* \alpha_* C_Y = \alpha_* L_Y \cdot \alpha_* C_Y = 3l \cdot l = 3
\]
and
\[
L_Y \cdot \sum^6_{i=1} a_i E_{i} = \sum^6_{i=1} a_i . 
\]
Thus we have
\[
3 =  \sum^6_{i=1} a_i
\]
by calculating $L_Y \cdot (\ref{no2} )$.
This means that $\Gamma$ passes through $P_1, P_2, P_3$ three times.
This, however, contradicts the fact that $\Gamma$ is a line in $\pr^2$.

\end{proof}

\begin{claim}

There is no floating $(-1)$-curves on $X$.

\end{claim}

\begin{proof}

Assume there exists a floating $(-1)$-curve $C$ on $X$.
Since $C$ does not pass through any singular points on $X$, we have $f^* C = C_Y$.
Thus we have
$1 = -K_X \cdot C = f^*(-K_X) \cdot C_Y = \frac{1}{6} L_Y \cdot C_Y + \frac{2}{3} l_Y \cdot C_Y + \frac{2}{3} E_2 \cdot C_Y =  \frac{2}{3} E_2 \cdot C_Y$.
Hence we have $E_Y \cdot C_Y = \frac{3}{2}$.
Since $Y$ is a smooth surface, $E_Y \cdot C_Y$ must be an integer number.
This is a contradiction.  

\end{proof}

\noindent From this construction, it follows that $X$ is a del Pezzo surface having no floating $(-1)$-curves such that $(-K_X)^2 = \frac{1}{3}$ and $(n_3, n_4) = (4,3)$.
Hence $X$ is of No.2. \\

\noindent{\bf No.3}

Let $X$ be a del Pezzo surface of No.3.
Let $Y \to X$ be the minimal resolution.
By observing the configuration of negative curves on $Y$, we can find the following blow-downs $\alpha_1$, $\alpha_2$:
\begin{figure}[htbp]
\centering\includegraphics[width=14cm, bb=0 0 1435 608]{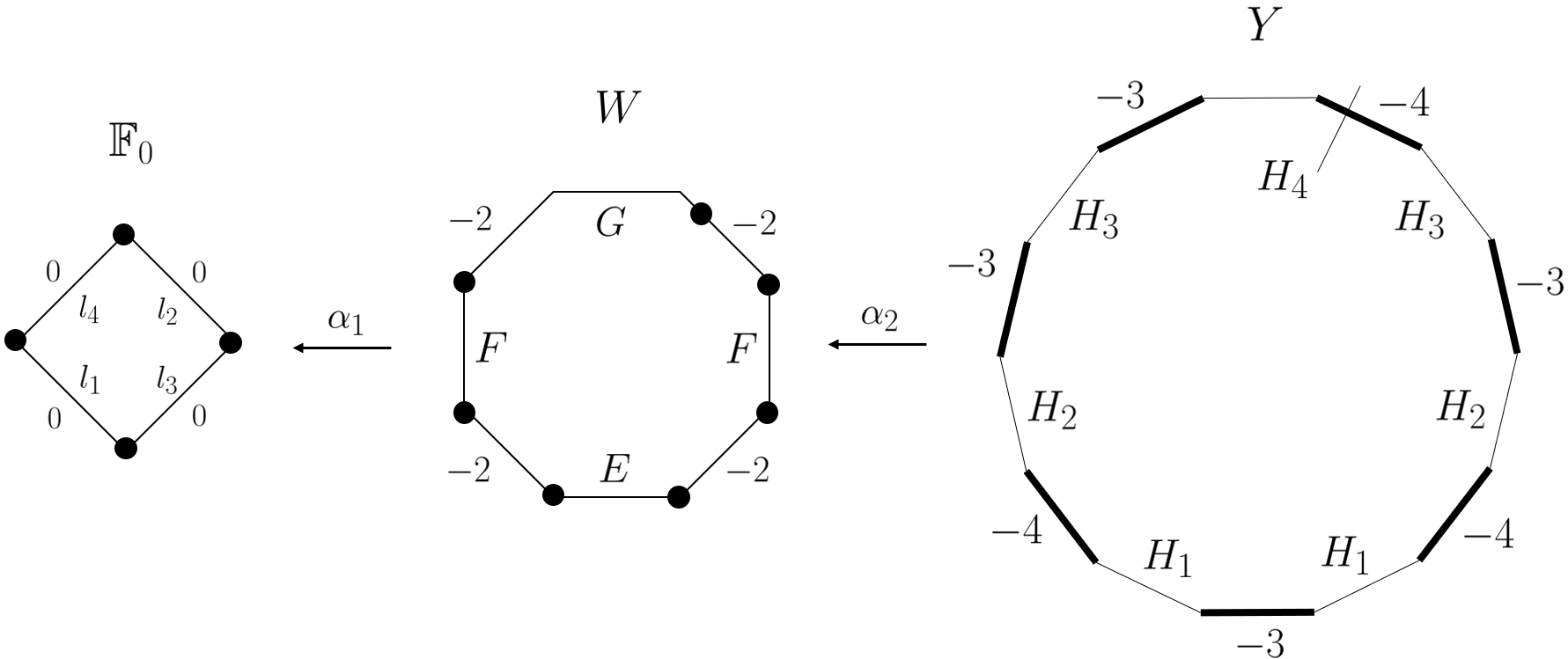}
\end{figure}

Let us construct an example of No.3.
Let $\alpha_1 : W \to \mathbb{F}_0$ be the blow-up at $(l_1 + l_2) \cap (l_3 + l_4)$.
Denote the exceptional divisor over $l_1 \cap l_3$ by $E$, over $(l_1 \cap l_4) \cup (l_2 \cap l_3)$ by $F$ and $l_2 \cap l_4$ by $G$. 
Denote the strict transform of $l_1, \ldots , l_4$ by $l'_1, \ldots , l'_4$. 
Then we have
\begin{eqnarray*}
 -K_W &=& \alpha_1^* (-K_{\mathbb{F}_0}) - E - F - G  \\
  &=& \alpha_1^* (l_1 + l_2 + l_3 + l_4) - E - F - G  \\
  &=& (l'_1 + l'_3 + 2E + F) + (l'_2 + l'_4 + F + 2G) -E - F - G \\
  &=& l'_1 + l'_2 + l'_3 + l'_4 + E + F + G.
\end{eqnarray*}

Take a general point $P$ on $l'_2$.
Let $\alpha_2 : Y \to W$ be a blow up at seven points, $((l'_1 \cup l'_2 \cup l'_3 \cup l'_4) \cap (E \cup F)) \cup \{ P \}$.
Denote the exceptional divisors over $(l'_1 \cup l'_3) \cap E$ by $H_1$, over $(l'_1 \cup l'_3) \cap F$, over $(l'_2 \cup l'_4) \cap F$ by $H_3$ and over $P$ by $H_4$.
Denote the strict transform of $l'_1, \ldots , l'_4$ by $l''_1, \ldots , l''_4$.
Then we have
\begin{eqnarray*}
 -K_Y &=& \alpha_2^* (-K_W) - H_1 - H_2 - H_3 - H_4  \\
  &=& \alpha_2^* (l'_1 + l'_2 + l'_3 + l'_4 + E + F + G) - H_1 - H_2 - H_3 - H_4  \\
  &=&   (l''_1 + l''_3 + H_1 + H_2) + (l''_2 + l''_4 + H_3 + H_4) + (E_Y + H_1)  \\
   &  & + (F_Y + H_2 + H_3) + G_Y - H_1 - H_2 - H_3 - H_4 \\
  &=& l''_1 + l''_2 + l''_3 + l''_4 + E_Y + F_Y + G_Y + H_1 + H_2 + H_3 .
\end{eqnarray*} 
Let $f : Y \to X$ be the contraction of $l''_1, \ldots , l''_4$, $E_Y$ and $F_Y$.
Then we also have
\[
K_Y = f^*K_X - \frac{1}{2}l''_1 - \frac{1}{2}l''_2 - \frac{1}{2}l''_3 - \frac{1}{3}l''_4 - \frac{1}{3}E_Y - \frac{1}{3}F_Y .
\]        
Hence we obtain the following relation;
\[
f^*(-K_X) = \frac{1}{2}l''_1 + \frac{1}{2}l''_2 + \frac{1}{2}l''_3 + \frac{2}{3}l''_4 + \frac{2}{3}E_Y + \frac{2}{3}F_Y + G_Y + H_1 + H_2 + H_3.
\]

\begin{claim}

$-K_X$ is ample.

\end{claim}

\begin{proof}

We see that $(-K_X)^2 = \frac{4}{3} > 0$ and $-K_X$ is nef as in Claim \ref{K^2_nef}.
Let $C$ be an irreducible curve on $X$.
Assume that $-K_X \cdot C = 0$ by contradiction. 
By definition, we see that $C_Y \cdot E_Y = C_Y \cdot F_Y = 0$ and $C_Y \cdot l''_i = 0$ for $i \in \{ 1, \ldots , 4 \}$.
If $C_Y \subset G_Y \cup H_1 \cup \cdots \cup H_4$, then we see that $f^* (-K_X) \cdot C_Y \not= 0$ by calculation.
Thus we may assume that $C_Y \not\subset G_Y \cup H_1 \cup \cdots \cup H_4$.
Then we see that $C_Y \cdot G_Y = C_Y \cdot H_1 = C_Y \cdot H_2 = C_Y \cdot H_3 = 0$ since $f^* (-K_X) \cdot C_Y = 0$.
Thus $\alpha_* C_Y$ is an irreducible curve and let $\alpha_* C_Y = al_1 +bl_3$, where $\alpha := \alpha_1 \circ \alpha_2$.
Then we have
\[
\alpha^* \alpha_* C_Y = C_Y + dH_4 ,
\]
where $d = C_Y \cdot H_4 \ge 0$.
Thus we have
\[
l''_1 \cdot \alpha^* \alpha_* C_Y = l''_1 \cdot C_Y + l''_1 \cdot dH_4 .
\]
We see that $l''_1 \cdot \alpha^* \alpha_* C_Y = l_1 \cdot (al_1 +bl_3) = b$.
We also see that $l''_1 \cdot C_Y = l''_1 \cdot H_4 = 0$.
Hence we obtain $b=0$.
We also obtain $a=0$ similarly.
Therefore, we have $\alpha_* C_Y = 0$.
This contradicts the fact that $\alpha_* C_Y$ is an irreducible curve.
Thus we see that $X$ is a del Pezzo surface. 

\end{proof}

\noindent From this construction, it follows that $X$ is a del Pezzo surface such that $(-K_X)^2 = \frac{4}{3}$ and $(n_3, n_4) = (4,3)$.
If $X$ has some floating $(-1)$-curves, we obtain a contradiction as in the proof of Lemma \ref{red_float}.
Thus we see that $X$ has no floating $(-1)$-curves.
Hence $X$ is of No.3. \\

\noindent{\bf No.5} 

Let $X$ be a del Pezzo surface of No.5.
Let $Y \to X$ be the minimal resolution.
By observing the configuration of negative curves on $Y$, we can find the following blow-downs $\alpha_1$, $\alpha_2$:
\begin{figure}[htbp]
\centering\includegraphics[width=14cm, bb=0 0 1450 560]{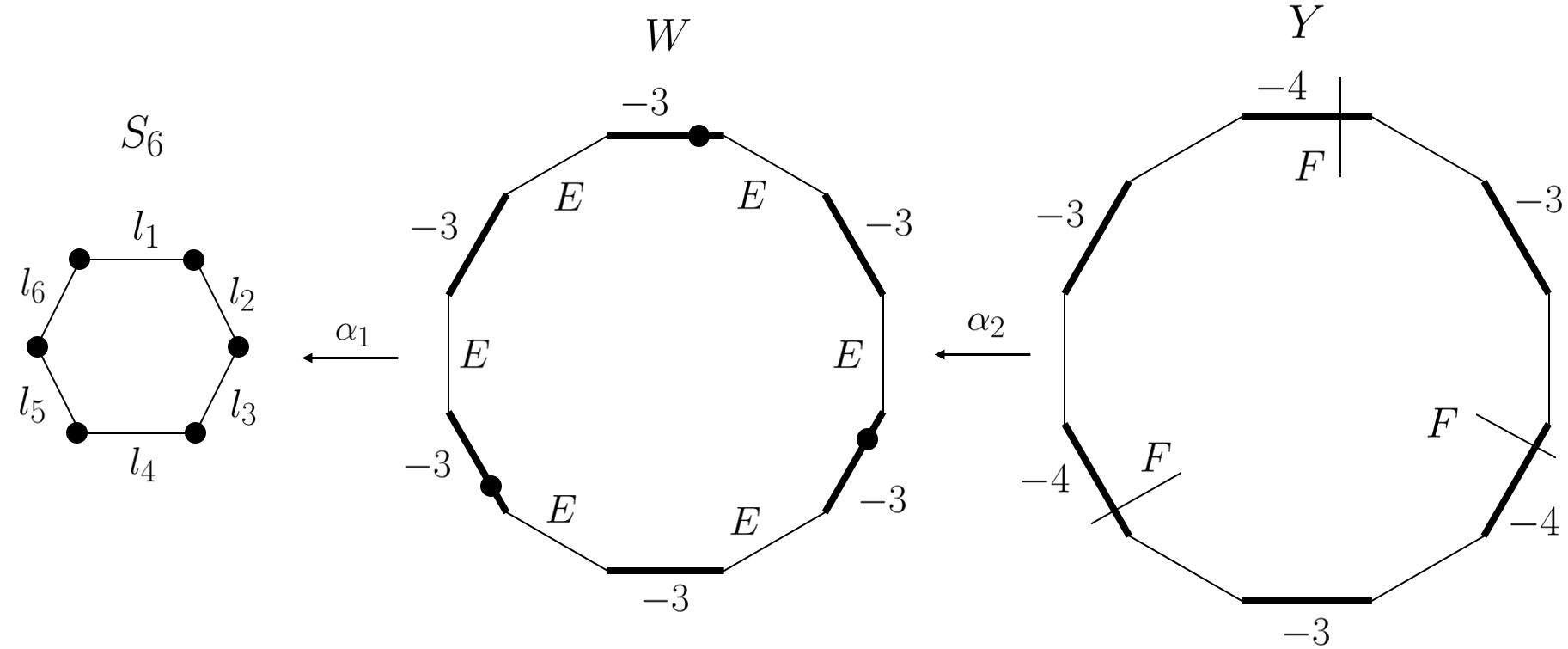}
\end{figure}

Let us construct an example of No.5.
Denote $L := l_1 + l_3 + l_5$ and $M := l_2 + l_4 + l_6$.
Let $\alpha_1 : W \to S_6$ be the blow-up at the six points on $S_6$.
Denote the exceptional divisor by $E$.
Denote the strict transforms of $l_1, l_3, l_5$ by $l'_1, l'_3, l'_5$ respectively.
Then we have
\begin{eqnarray*}
 -K_W &=& \alpha_1^* (-K_{S_6}) - E  \\
  &=& \alpha_1^* (L + M) - E  \\
  &=&   L_W + M_W + E .
\end{eqnarray*} 

Let $\alpha_2 : Y \to W$ be the blow-up at general points on $l'_1, l'_3, l'_5$.
Take a point $P_1$ on $l'_1$, $P_2$ on $l'_3$ and $P_3$ on $l'_5$.
Denote the exceptional divisor over $P_i$ by $F_i$ for $i \in \{1,2,3 \}$.
Set $F := F_1 + F_2 + F_3$.
Then we have 
\begin{eqnarray*}
 -K_Y &=& \alpha_1^* (-K_W) - F  \\
  &=& \alpha_1^* (L_W + M_W + E) - F  \\
  &=&   (L_Y + F) + M_Y + E_Y - F   \\
  &=&  L_Y + M_Y + E_Y .
\end{eqnarray*}
Let $f : Y \to X$ be the contraction of $L_Y$ and $M_Y$.
Then we also have
\[
K_Y  = f^*K_X - \frac{1}{2}L_Y - \frac{1}{3}M_Y .
\]
Hence we obtain the following relation;
\[
f^*(-K_X) =  \frac{1}{2}L_Y + \frac{2}{3}M_Y + E_Y .
\]

\begin{claim}

$-K_X$ is ample.

\end{claim}

\begin{proof}

We see that $(-K_X)^2 = 1 > 0$ and $-K_X$ is nef as in Claim \ref{K^2_nef}.
Let $C$ be an irreducible curve on $X$.
Assume that $-K_X \cdot C = 0$ by contradiction.
By assumption, we see that $L_Y \cdot C_Y = M_Y \cdot C_Y = E_Y \cdot C_Y = 0$.
Let $p : S_6 \to \pr^2$ be a contraction of $l_2, l_4, l_6$.
Set $\alpha_* C_Y = dl$, where $\alpha := p \circ \alpha_1 \circ \alpha_2 : Y \to \pr^2$.
Denote the strict transforms of $l_2, l_4, l_6$ by $l''_2, l''_4, l''_6$.
Then we have
\begin{equation}\label{No.5}
\alpha^* \alpha_* C_Y = C_Y + xl''_2 + yl''_4 + zl''_6 + \sum_{i=1}^6 c_i E_{Y,i} + \sum_{i=1}^3 d_i F_i ,
\end{equation}
where $x \, y, z, c_i, d_i \in \Z$.
If we multiply both sides of (\ref{No.5}) by $L_Y$, then we have
\[
3d = \sum_{i=1}^6 c_i + \sum_{i=1}^3 d_i . 
\]
By multiplying both sides of (\ref{No.5}) by $M_Y$ and $E_Y$, we have
\[
0 = -3(x+y+z) + \sum_{i=1}^6 c_i \ \ {\rm and} \ \ 0 = 2(x+y+z) - \sum_{i=1}^6 c_i . 
\]
By these relations, we obtain $x+y+z = \sum_{i=1}^6 c_i = 0$.
Thus we have
\[
3d - \sum_{i=1}^3 d_i = 0.
\] 
Then we see that $\alpha(F)$ is a set of three points in $\pr^2$.
Let $\varphi : S'_6 \to \pr^2$ be the blow-up at the three points.
Since we take $P_1$, $P_2$ and $P_3$ generally, $S'_6$ is a del Pezzo surface.
Then there exists a birational morphism $\psi : Y \to S'_6$ such that $\alpha = \varphi \circ \psi : Y \to \pr^2$.
We may denote $\psi_* F$ by $F$ again.
Then we have
\[
\varphi^* \alpha_* C_Y = C_{S'_6} +  \sum_{i=1}^3 d_i F_i.
\]
By this relation, we have
\[
-K_{S'_6} \cdot \varphi^* \alpha_* C_Y = -K_{S'_6} \cdot C_{S'_6} + (-K_{S'_6}) \cdot \sum_{i=1}^3 d_i F_i.
\]
We see that $-K_{S'_6} \cdot \varphi^* \alpha_* C_Y = -K_{\pr^2} \cdot dl = 3d$ and $-K_{S'_6} \cdot \sum_{i=1}^3 d_i F_i = \sum_{i=1}^3 d_i$.
Hence we have
\[
-K_{S'_6} \cdot C_{S'_6} = 3d - \sum_{i=1}^3 d_i  = 0.
\]
This contradicts the fact that $S'_6$ is a del Pezzo surface.

\end{proof}

\noindent From this construction, it follows that $X$ is a del Pezzo surface such that $(-K_X)^2 = 1$ and $(n_3, n_4) = (3,3)$.
If $X$ has some floating $(-1)$-curves, we obtain a contradiction as in the proof of Lemma \ref{red_float}.
Thus we see that $X$ has no floating $(-1)$-curves.
Hence $X$ is of No.5. \\

\noindent{\bf No.7}

Let $X$ be a del Pezzo surface of No.7.
Let $Y \to X$ be the minimal resolution.
By observing the configuration of negative curves on $Y$, we can find the following blow-downs $\alpha_1$, $\alpha_2$:
\begin{figure}[htbp]
\centering\includegraphics[width=14cm, bb=0 0 1440 500]{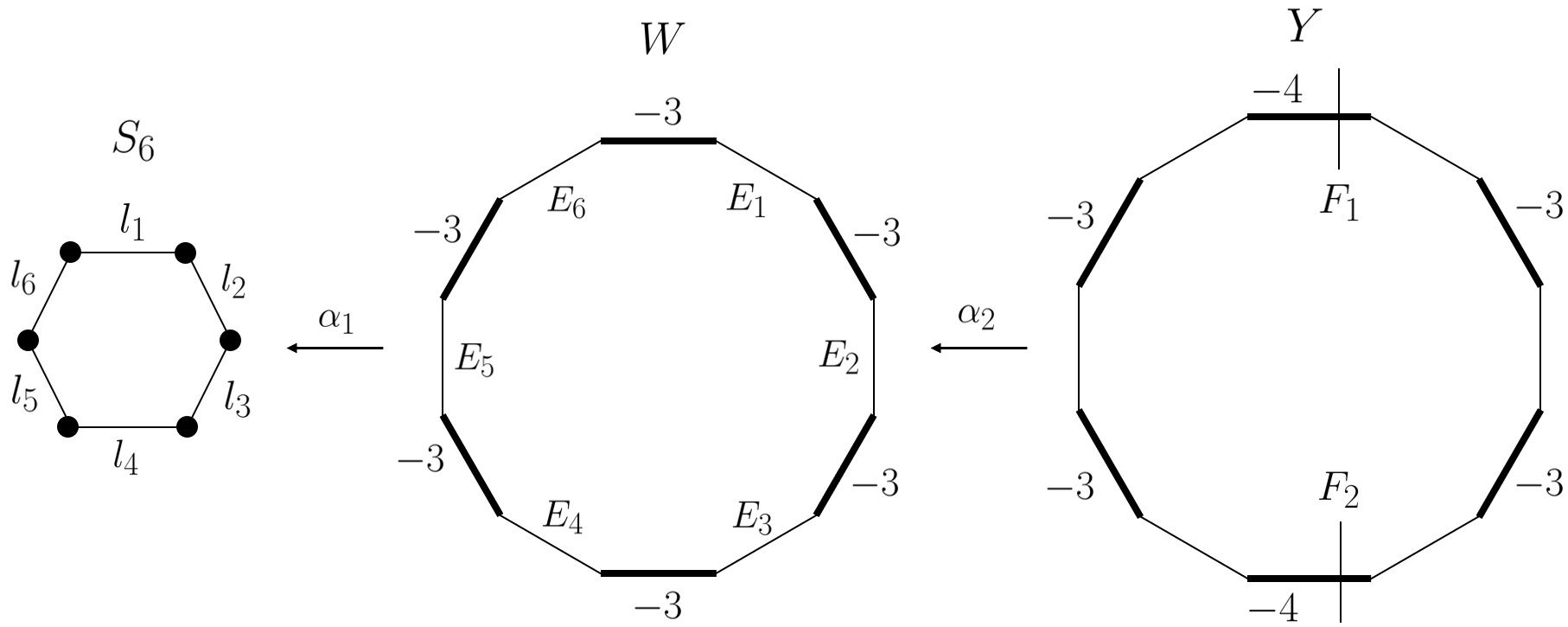}
\end{figure}

Let us construct an example of No.7.
In $S_6$, set $L := l_1 + l_4$ and $M := l_2 + l_3 + l_5 + l_6$.
Let $\alpha_1 : W \to S_6$ be the blow-up at the six points on $S_6$.
Denote the exceptional divisor by $E$.
Denote each irreducible component of $E$ by $E_i$ as in the picture.
Denote the strict transform of $l_1, l_4$ by $l'_1, l'_4$.
Then we have
\begin{eqnarray*}
 -K_W &=& \alpha_1^* (-K_{S_6}) - E  \\
  &=& \alpha_1^* (L + M) - E  \\
  &=&   L_W + M_W + E .
\end{eqnarray*} 
Take a general point $P_1$ on $l'_1$ and $P_2$ on $l'_4$.
Let $\alpha_2 : Y \to W$ be the blow-up at $P_1$ and $P_2$.
Denote the exceptional divisor over $P_i$ by $F_i$ and set $F := F_1 + F_1$.
Then we have
\begin{eqnarray*}
 -K_Y &=& \alpha_1^* (-K_W) - F  \\
  &=& \alpha_1^* (L_W + M_W + E ) - F  \\
  &=&   (L_Y + F) + M_Y + E_Y - F   \\
  &=&  L_Y + M_Y + E_Y .
\end{eqnarray*} 
Let $f : Y \to X$ be the contraction of $L_Y$ and $M_Y$.
Then we also have
\[
K_Y = f^*K_X - \frac{1}{2}L_Y - \frac{1}{3}M_Y  ,
\] 
Hence we obtain the following relation;
\[
f^*(-K_X) = \frac{1}{2}L_Y + \frac{2}{3}M_Y + E_Y  .
\]

\begin{claim}

$-K_X$ is ample.

\end{claim}

\begin{proof}

We see that $(-K_X)^2 = \frac{4}{3} > 0$ and $-K_X$ is nef as in Claim \ref{K^2_nef}.
Let $C$ be an irreducible curve on $X$.
Assume that $-K_X \cdot C = 0$ by contradiction.
Let $p : S_6 \to \F_0$ be a contraction of $l_3, l_6$.
Denote $p_* l_1$ by $G$ and $p_* l_2$ by $H$. 
Let $\alpha_* C_Y = a_1 G + a_2 H$, where $\alpha := p \circ \alpha_1 \circ \alpha_2 : Y \to \F_0$.
Then we have
\begin{equation}\label{No.7}
\alpha^* \alpha_* C_Y = C_Y + xl''_3 + yl''_6 + \sum_{i=1}^6 b_i E_{Y,i} + \sum_{i=1}^2 c_i F_i ,
\end{equation}
where $x, y, b_i, c_i \in \Z$.
If we multiply both sides of (\ref{No.7}) by $(l''_1 + \cdots + l''_6)$, then we have
\[
2(a_1 + a_2) = -3(x+y) + 2\sum_{i=1}^6 b_i + c_1 + c_2 . 
\]
By multiplying both sides of (\ref{No.7}) by $(l''_3 + l''_6 - (E_1)_Y - (E_4)_Y)$ and $E_Y$, then we have
\[
0 = -3(x+y) + \sum_{i=1}^6 b_i \ \ {\rm and} \ \ 0 = 2(x+y) - \sum_{i=1}^6 b_i . 
\]
By these relations, weobtain $x+y = \sum_{i=1}^6 b_i = 0$.
Thus we have
\[
2(a_1 + a_2) - ( c_1 + c_2 ) = 0 . 
\]
Then we see that $\alpha(F)$ is a set of two points in $\F_0$.
Let $\varphi : S'_6 \to \F_0$ be the blow-up at the two points.
Since we take $P_1$ and $P_2$ generally, $S'_6$ is a del Pezzo surface.
Then there exists a birational morphism $\psi : Y \to S'_6$ such that $\alpha = \varphi \circ \psi : Y \to \F_0$.
We may denote $\psi_* F$ be $F$ again. 
Then we have
\[
\varphi^* \alpha_* C_Y = C_{S'_6} + \sum_{i=1}^2 c_i F_i.
\]
By this relation, we have
\[
-K_{S'_6} \cdot \varphi^* \alpha_* C_Y = -K_{S'_6} \cdot C_{S'_6} + (-K_{S'_6}) \cdot  \sum_{i=1}^2 c_i F_i.
\]
We see that $-K_{S'_6} \cdot \varphi^* \alpha_* C_Y = -K_{\F_0} \cdot (a_1 G + a_2 H) = 2(a_1 + a_2)$ and $-K_{S'_6} \cdot  \sum_{i=1}^2 c_i F_i = c_1 + c_2$.
Hence we have
\[
-K_{S'_6} \cdot C_{S'_6} = 2(a_1 + a_2) - ( c_1 + c_2 ) = 0 . 
\]
This contradicts the fact that $S'_6$ is a del Pezzo surface.

\end{proof}

\noindent From this construction, it follows that $X$ is a del Pezzo surface such that $(-K_X)^2 = \frac{4}{3}$ and $(n_3, n_4) = (4,2)$.
If $X$ has some floating $(-1)$-curves, we obtain a contradiction as in the proof of Lemma \ref{red_float}.
Thus we see that $X$ has no floating $(-1)$-curves and $X$ is No.6 or No.7.
Moreover, we can find a $\smor$-sequence $X \overset{\smor_4}{\to} X_1  \overset{\smor_4}{\to} \PP$.
Hence $X$ is of No.7.
We also show that surfaces of No.6 and No.7 are distinct in the next section. \\

\noindent{\bf No.10}

Let $X$ be a del Pezzo surface of No.10.
Let $Y \to X$ be the minimal resolution.
By observing the configuration of negative curves on $Y$, we can find the following blow-downs $\alpha_1$, $\alpha_2$:
\begin{figure}[htbp]
\centering\includegraphics[width=14cm, bb=0 0 1400 550]{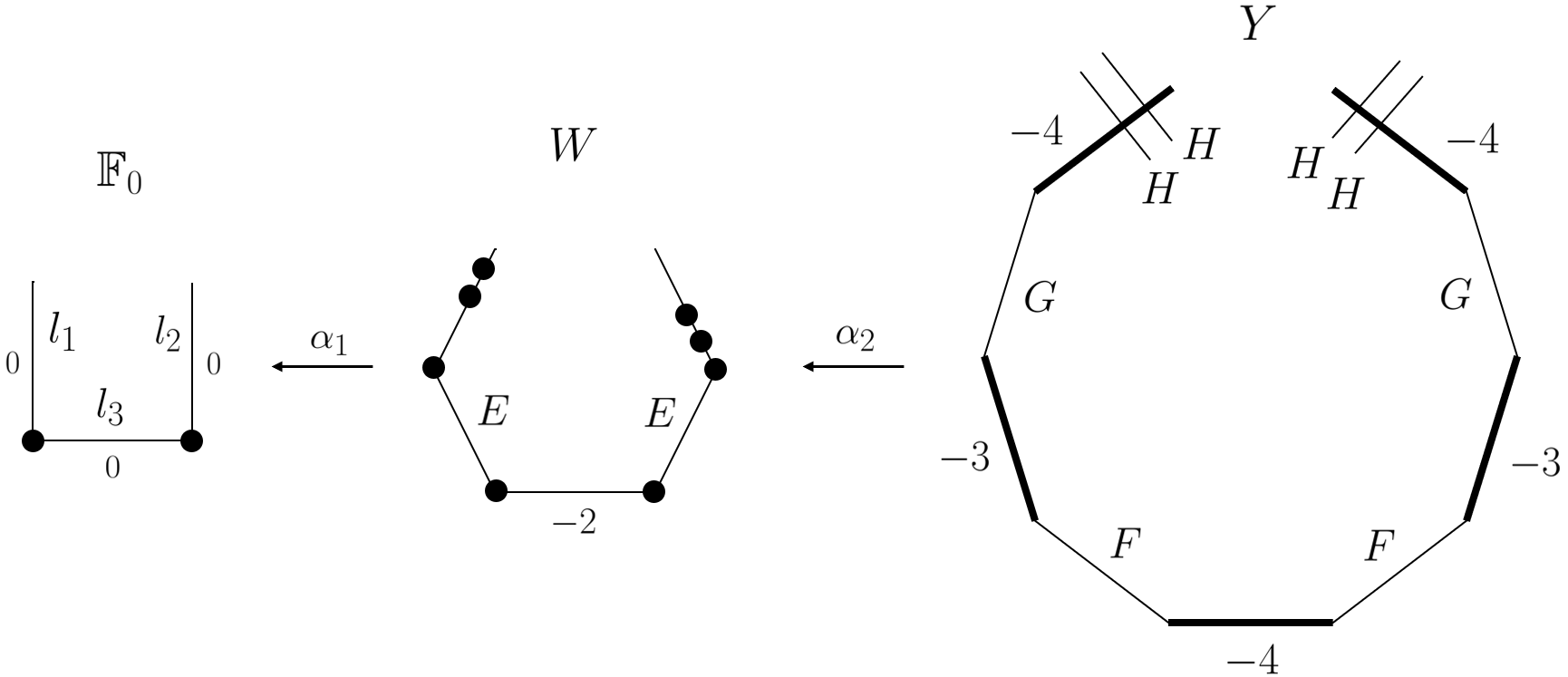}
\end{figure}

Let us construct an example of No.10.
Let $L := l_1 + l_2$ in $\F_0$.
Let $\alpha_1 : S \to \mathbb{F}_0$ be the blow-up at $L \cap l_3$.
Denote the exceptional divisor by $E$.
Denote the strict transform of $l_3$ by $l'_3$.
We have
\begin{eqnarray*}
 -K_W &=& \alpha_1^* (-K_{\mathbb{F}_0}) - E  \\
  &=& \alpha_1^* (L + 2l_3) - E  \\
  &=&   (L_W + E) + 2(l'_3 + E) - E \\
  &=&  L_W + 2l'_3 + 2E .
\end{eqnarray*} 
Take two general points $ P_1, P_2$ on $l'_1$ and two general points $ P_3, P_4 $ on $l'_2$.
Let $\alpha_2 : Y \to W$ be the blow-up at $((L_W + l'_3) \cap E) \cup \{ P_1, P_2, P_3, P_4 \}$.
Denote the exceptional divisor over $l'_1 \cap E$ by $F$, over $E \cap L_W$ by $G$ and over $\{ P_1, P_2, P_3, P_4 \}$ by $H$.
Denote the strict transform of $l'_3$ by $l''_3$.
We have
\begin{eqnarray*}
 -K_Y &=& \alpha_1^* (-K_W) - F - G - H  \\
  &=& \alpha_1^* (L_W + 2l'_3 + 2E) - F - G - H  \\
  &=&   (L_Y + G + H) + 2(l''_3 + F) + 2(E_Y + F + G) - F - G - H   \\
  &=&  L_Y + 2l''_3 + 2E_Y + 3F + 2G.
\end{eqnarray*} 
Let $f : Y \to X$ be the contraction of $L_Y$, $l''_3$ and $E_Y$.
Then we also have
\[
K_Y = f^*K_X - \frac{1}{2}L_Y - \frac{1}{2}l''_3 - \frac{1}{3}E_Y  .
\] 
Hence we obtain the following relation;
\[
f^*(-K_X) = \frac{1}{2}L_Y + \frac{3}{2}l''_3 + \frac{5}{3}E_Y + 3F + 2G  .
\]

\begin{claim}

$-K_X$ is ample.

\end{claim}

\begin{proof}

We see that $(-K_X)^2 = \frac{5}{3} > 0$ and $-K_X$ is nef as in Claim \ref{K^2_nef}.
Let $C$ be an irreducible curve on $X$.
Assume that $-K_X \cdot C = 0 $ by contradiction.
Set $\alpha_* C_Y = xl_1 + yl_3$, where $\alpha := \alpha_1 \circ \alpha_2$.
Then we have
\begin{equation}\label{No.10}
\alpha^* \alpha_* C_Y = C_Y + \sum_{i=1}^2 a_i (E_i)_Y + \sum_{i=1}^2 b_i F_i + \sum_{i=1}^2 c_i G_i + \sum_{i=1}^4 d_i H_i.
\end{equation}
By multiplying both sides of (\ref{No.10}) by $l''_3$, $E_Y$, $F$ and $G$, we have
\begin{equation*}
\begin{cases}
\ \displaystyle x =  \sum_{i=1}^2 b_i \\
\ \displaystyle 0 = -3\sum_{i=1}^2 a_i + \sum_{i=1}^2 b_i + \sum_{i=1}^2 c_i \\
\ \displaystyle 0 = \sum_{i=1}^2 a_i - \sum_{i=1}^2 b_i \\
\ \displaystyle 0 = \sum_{i=1}^2 a_i - \sum_{i=1}^2 c_i .
\end{cases}
\end{equation*}
Thus we obtain
\[
x =  \sum_{i=1}^2 a_i = \sum_{i=1}^2 b_i = \sum_{i=1}^2 c_i  = 0 .
\]
Hence we see that $\alpha_* C_Y = yl_3$.
We also see that $\alpha_* C_Y$ is irreducible.
Thus we see that $y=1$ and $\alpha_* C_Y$ is a fiber of $\pi_2$.
Then by multiplying both sides of (\ref{No.10}) by $L_Y$, we also obtain
\[
2  =  \sum_{i=1}^4 d_i.
\]
From these facts, we see that $C_Y$ is the strict transform of a fiber of $\pi_2$ and two of $P_1, \ldots , P_4$ are on the fiber.
Since we take $P_1, \ldots , P_4$ generally, this is a contradiction.

\end{proof}

\noindent From this construction, it follows that $X$ is a del Pezzo surface such that $(-K_X)^2 = \frac{5}{3}$ and $(n_3, n_4) = (2,3)$.
If $X$ has some floating $(-1)$-curves, we obtain a contradiction as in the proof of Lemma \ref{red_float}.
Thus we see that $X$ has no floating $(-1)$-curves.
Hence $X$ is of No.10. \\

\noindent{\bf No.11}

Let $X$ be a del Pezzo surface of No.11.
Let $Y \to X$ be the minimal resolution.
By observing the configuration of negative curves on $Y$, we can find the following blow-downs $\alpha_1$, $\alpha_2$:
\begin{figure}[htbp]
\centering\includegraphics[width=14cm, bb=0 0 1450 550]{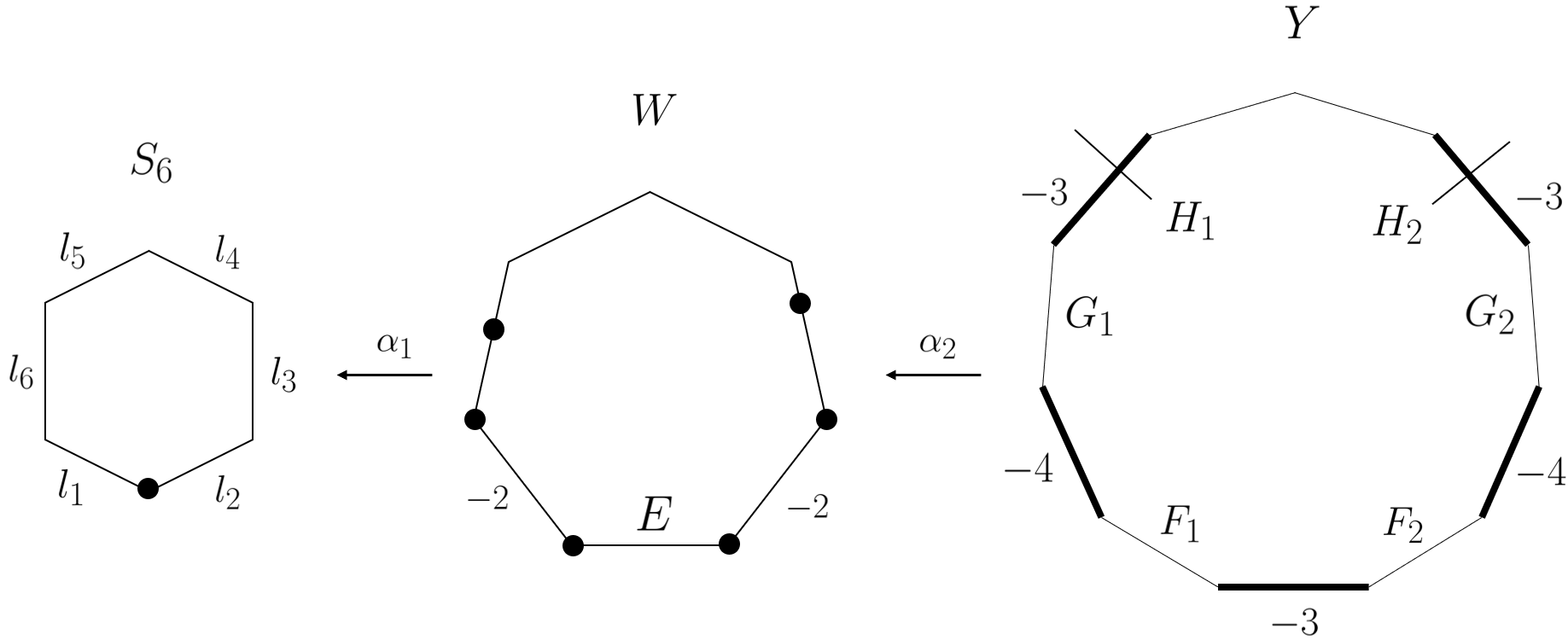}
\end{figure}

Let us construct an example of No.11.
Let $L : = l_1 + l_2$, $M := l_3 + l_6$ and $N := l_4 + l_5$.
Let $\alpha_1 : W \to S_6$ be the blow-up at $l_1 \cap l_2$.
Denote the exceptional divisor by $E$.
Then we have
\begin{eqnarray*}
 -K_S &=& \alpha_1^* (-K_{S_6}) - E  \\
  &=& \alpha_1^* (2L + M) - E  \\
  &=&   2(L_W + 2E) + M_W  - E \\
  &=&  2L_W + M_W + 3E .
\end{eqnarray*} 
Take a general point $P_1$ on $l'_6$ and a general point $P_2$ on $l'_3$.
Let $\alpha_2 : Y \to W$ be the blow-up at $(L_W \cap (M_W + E) ) \cup \{ P_1, P_2 \}$.
Denote the exceptional divisor over $L_W \cap E$ by $F$, over $L_W \cap M_W$ by $G$ and over $\{ P_1, P_2 \}$ by $H$.
Then we have
\begin{eqnarray*}
 -K_Y &=& \alpha_2^* (-K_W) - F - G - H  \\
  &=& \alpha_2^* (2L_W + M_W + 3E) - F - G - H  \\
  &=&   2(L_Y + F + G) + (M_Y + G + H) + 3(E_Y + F) - F - G - H   \\
  &=&  2L_Y + M_Y  + 3E_Y + 4F + 2G.
\end{eqnarray*} 
Let $f : Y \to X$ be the contraction of $L_Y$, $M_Y$ and $E_Y$.
Then we also have
\[
K_Y = f^*K_X - \frac{1}{2}L_Y - \frac{1}{3}M_Y - \frac{1}{3}E_Y  ,
\] 
Hence we obtain the following relation;
\[
f^*(-K_X) = \frac{3}{2}L_Y + \frac{2}{3}M_Y + \frac{8}{3}E_Y + 4F + 2G  .
\]

\begin{claim}

$-K_X$ is ample.

\end{claim}

\begin{proof}

We see that $(-K_X)^2 = 2 > 0$ and $-K_X$ is nef as in Claim \ref{K^2_nef}.
Let $C$ be an irreducible curve on $X$.
Assume that $-K_X \cdot C = 0$ by contradiction.
Let $p : S_6 \to \F_0$ be a contraction of $l_1$ and $l_4$.
Denote $p_* l_2$ by $I$ and $p_* l_3$ by $J$. 
Let $\alpha_* C_Y = a_1 I + a_2J$, where $\alpha := p \circ \alpha_1 \circ \alpha_2 : Y \to F_0$. 
We may assume that $I$ is a fiber of $\pi_1$ and $J$ is a fiber of $\pi_2$.
Then we have
\begin{equation}\label{No.11}
\alpha^* \alpha_* C_Y = C_Y + xl''_1 + yl''_4 + bE_Y + \sum_{i=1}^2 c_i F_i + \sum_{i=1}^2 d_i G_i + \sum_{i=1}^2 e_i H_i .
\end{equation}
By multiplying both sides of (\ref{No.11}) by $l''_1$, $l''_2$, $E_Y$, $F_1$, $F_2$, $G_1$ and $G_2$, we have
\begin{equation*}
\begin{cases}
\ 0 = -4x + c_1 + d_1 \\
\ a_2 = c_2 + d_2 \\
\ 0 = -3b + c_1 + c_2  \\
\ 0 = x + b - c_1   \\
\ 0 = b - c_2  \\
\ 0 = x - d_1  \\
\ 0 = -d_2 .
\end{cases}
\end{equation*}
Thus we obtain $x = a_2 = b = c_1 = c_2 = d_1 = d_2 = 0$.
Since $\alpha_* C_Y$ is irreducible, we see that $a_1 = 1$.
Hence $\alpha_* C_Y = I$ is a fiber.
If we multiply both sides of (\ref{No.11}) by $l''_4$, we have
\[
0 =  C_Y \cdot l''_4 - y .
\]
Thus we see that $y = C_Y \cdot l''_4 \ge 0$.
If we multiply both sides of (\ref{No.11}) by $l''_5$, we have
\[
0 =  C_Y \cdot l''_5 + y .
\]
Thus we see that $y = C_Y \cdot l''_5 \le 0$.
Hence we obtain $y = 0$.
If we multiply both sides of (\ref{No.11}) by $M_Y$, we have
\[
2 -  \sum_{i=1}^2 e_i = 0 .
\]
From these facts, we see that $C_Y$ is the strict transform of a fiber of $\pi_1$ and $P_1$ and $P_2$ are on the fiber.
This contradicts how to take $P_1$ and $P_2$.

\end{proof}

\noindent From this construction, it follows that $X$ is a del Pezzo surface such that $(-K_X)^2 = 2$ and $(n_3, n_4) = (3,2)$.
If $X$ has some floating $(-1)$-curves, we obtain a contradiction as in the proof of Lemma \ref{red_float}.
Thus we see that $X$ has no floating $(-1)$-curves and $X$ is No.11 or No.12.
Moreover, we can find a $\smor$-sequence $X \overset{\smor_1}{\to} X_1  \overset{\smor_8}{\to} \pr(1,1,4)$.
Hence $X$ is of No.11.
We also show that surfaces of No.11 and No.12 are distinct in the next section. \\

\noindent{\bf No.17} 

Let $X$ be a del Pezzo surface of No.17.
In order to prove that $X$ does not have any floating $(-1)$-curves, we use more explicit notation than the other cases.
Let $Y \to X$ be the minimal resolution.
By observing the configuration of negative curves on $Y$, we can find the following blow-downs $\alpha_1$, $\alpha_2$:
\begin{figure}[htbp]
\centering  \includegraphics[width=14cm, bb=0 0 1420 520]{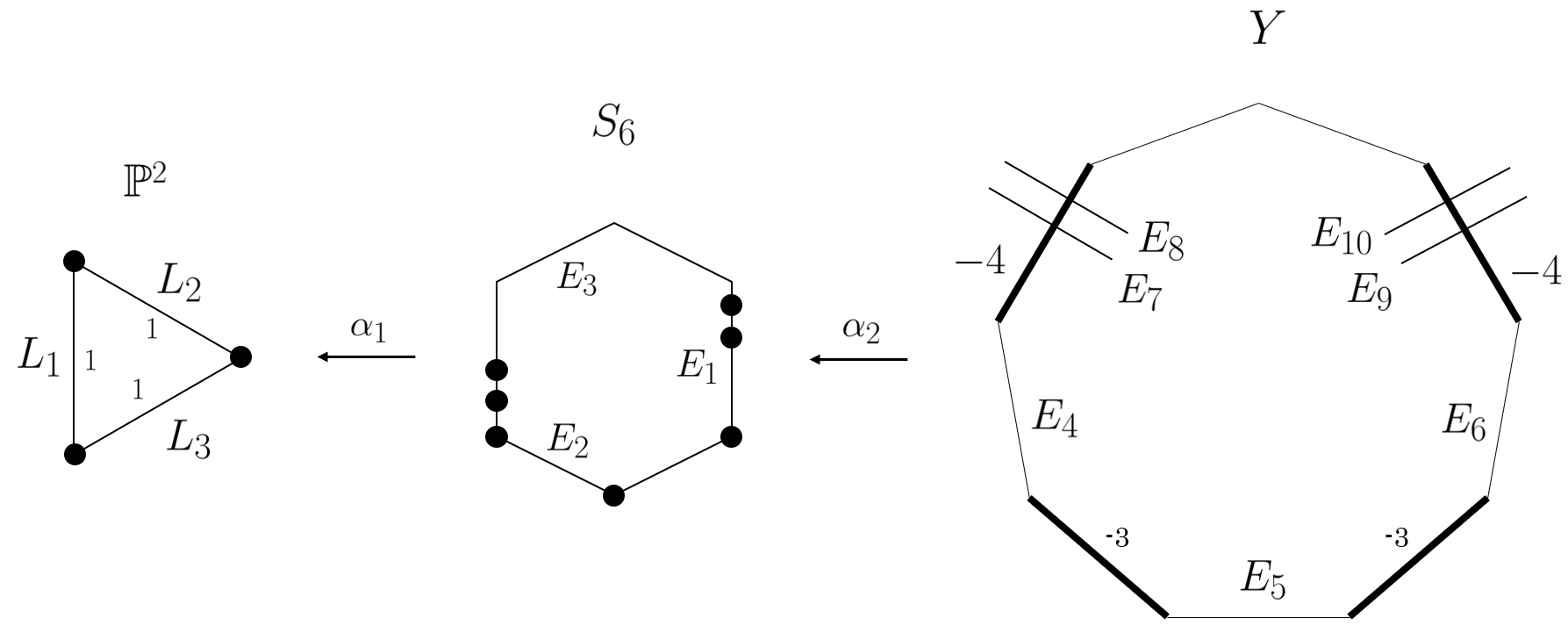}
\end{figure}

Let us construct an example of No.17.
Let $\alpha_1 : S_6 \to \pr^2$ be the blow-up at distinct three points $P_1, P_2, P_3$ on $\pr^2$ which are not on a line. 
Denote by $L_i$ a line which passes through $P_j$ and $P_k$ where $(i, j, k) = (1,2,3), (2,3,1), (3,1,2)$.
Denote the exceptional curve over $P_i$ by $E_i$ for each $i \in \{ 1,2,3 \} $.
We see that $(L_1)_{S_6} \cdot E_2 = 1$, $E_2 \cdot (L_3)_{S_6} = 1$ and $(L_3)_{S_6} \cdot E_1 = 1$.
Hence we set $P_4 := (L_1)_{S_6} \cap E_2$, $P_5 := E_2 \cap (L_3)_{S_6}$ and $P_6 := (L_3)_{S_6} \cap E_1$.
Take two general points $P_7, P_8$ on $(L_1)_{S_6}$ and two general points $P_9, P_{10}$ on $E_1$.

Let $\alpha_2 : Y \to S_6$ be the blow-up at $P_4, \ldots , P_{10}$.
Denote the exceptional curves over $P_i$ by $E_i$.

Set $L : = E_2 + (L_3)_{S_6}$, $M := E_1 + (L_1)_{S_6}$ and $N := E_3 + (L_2)_{S_6}$.
Set $E := E_5$, $F := E_4 + E_6$ and $G := E_7 + \cdots + E_{10}$.
Then we have
\begin{eqnarray*}
 -K_Y &=& \alpha_2^* (-K_{S_6}) - E - F - G  \\
  &=& \alpha_2^* (L + M + N) - E - F - G  \\
  &=&   (L_Y + 2E + F) + (M_Y + F + G) + N_Y - E - F - G  \\
  &=&  L_Y + M_Y + N_Y + E + F.
\end{eqnarray*} 
Let $f : Y \to X$ be the contraction of $L_Y$ and $M_Y$.
Then we also have
\[
K_Y = f^*K_X - \frac{1}{3}L_Y - \frac{1}{2}M_Y   .
\] 
Hence we obtain the following relation;
\[
f^*(-K_X) = \frac{2}{3}L_Y + \frac{1}{2}M_Y + N_Y + E + F.
\]

\begin{claim}

$-K_X$ is ample.

\end{claim}

\begin{proof}

We see that $(-K_X)^2 = \frac{5}{3} > 0$ and $-K_X$ is nef as in Claim \ref{exi_1}. 
Let $C$ be an irreducible curve on $X$.
Assume that $-K_X \cdot C = 0$ by contradiction.
Let $p : S_6 \to \F_0$ be a contraction of $E_3$ and $(L_3)_{S_6}$.
Denote $p_* E_1$ by $H$ and $p_* E_2$ by $I$.
Set $\alpha_* C_Y = a_1 H + a_2 I$, where $\alpha := p \circ \alpha_2$.
We may assume that $H$ is a fiber of $\pi_1$ and $I$ is a fiber of $\pi_2$.
Then we have
\begin{equation}\label{No.17}
\alpha^* \alpha_* C_Y = C_Y + x(E_3)_Y + y(L_3)_Y + cE + \sum_{i=1}^2 d_i F_i + \sum_{i=1}^4 e_i G_i
\end{equation}
By multiplying both sides of (\ref{No.17}) by $(E_2)_Y$, $(L_3)_Y$, $(E_3)_Y$, $(L_2)_Y$, $E$, $F_1$ and $F_2$, we have 
\begin{equation*}
\begin{cases}
\ a_1 = c + d_1 \\
\ 0 = -3y + c + d_2 \\
\ 0 = -3x \\
\ a_1 = x \\
\ 0 = -c + d_1 + d_2  \\
\ 0 = -d_1  \\ 
\ 0 = y - d_2 . \\ 
\end{cases}
\end{equation*}
Thus we obtain $x = d_1 = d_2 = c = y = a_1 = 0$.
Since $a_1 = 0$, we see that $a_2 = 1$ and $\alpha_* C_Y$ is a fiber.
If we multiply both sides of (\ref{No.17}) by $M_Y$, then we have
\[
2 - \sum_{i=1}^4 e_i = 0.
\]
From these facts, we see that $C_Y$ is the strict transform of a fiber of $\pi_1$ and two of $P_7, \ldots , P_{10}$ are on the fiber.
Since we take $P_7, \ldots , P_{10}$ generally, this is a contradiction.

\end{proof}

\begin{claim}

There is no floating $(-1)$-curves on $X$.

\end{claim}

\begin{proof}

Assume that there is a floating $(-1)$-curve $C$ on $X$.
Set $C_Y \sim xl + \sum_{i=1}^{10} a_i e_i$ and it is also a $(-1)$-curve.
$Y$ has exactly two $(-4)$-cruves $(L_1)_Y, (E_1)_Y$ and exactly two $(-3)$-curves $(E_2)_Y, (L_3)_Y$. 
Then we have
\[
(L_1)_Y \sim l - e_2 - e_3 - e_4 - e_7 - e_8,
\]
\[
(E_1)_Y \sim e_1 - e_6 - e_9 - e_{10},
\]
\[
(E_2)_Y \sim e_2 - e_4 - e_5
\]
and
\[
(L_3)_Y \sim l - e_1 - e_2 - e_5 - e_6  .
\]
Since $C$ does not pass through any singular points, $C_Y$ does not cross with $(L_1)_Y$, $(E_1)_Y$, $(E_2)_Y$ and $(L_3)_Y$.
Hence we have
\begin{numcases}
{}
\ 0 = x + a_2 + a_3 + a_4 + a_7 + a_8 \label{a}  \\ 
\ 0 = -a_1 + a_6 + a_9 + a_{10} \label{b} \\ 
\ 0 = -a_2 + a_4 + a_5 \label{c} \\
\ 0 = x + a_1 + a_2 + a_5 + a_6 \label{d} \ .
\end{numcases}
By calculating (\ref{a})$+$(\ref{b})$+ 2 \times $(\ref{c})$+ 2 \times $(\ref{d}), we have
\[
0 = 3x + 2a_4 + 3a_5 + 2a_6 + \sum_{i=1}^{10} a_i.
\]
Since $-K_Y \cdot C_Y = 1$, we have
\[
1 = 3x + \sum_{i=1}^{10} a_i  .
\]
Therefore, we obtain
\[
1 = -2a_4 - 3a_5 - 2a_6.
\]
For $i = 4, 5$ and $6$, we see that $a_i \le 0$ since $E_i \cdot C_Y \ge 0$.
Thus this is a contradiction.

\end{proof}

\noindent From this construction, it follows that $X$ is a del Pezzo surface having no floating $(-1)$-curves such that $(-K_X)^2 = \frac{5}{3}$ and $(n_3, n_4) = (2,2)$.
Hence $X$ is of No.17. \\

\noindent{\bf No.20} 

Let $X$ be a del Pezzo surface of No.20.
Let $Y \to X$ be the minimal resolution.
By observing the configuration of negative curves on $Y$, we can find the following blow-down $\alpha_1$:
\begin{figure}[htbp]
\centering\includegraphics[width=14cm, bb=0 0 1060 350]{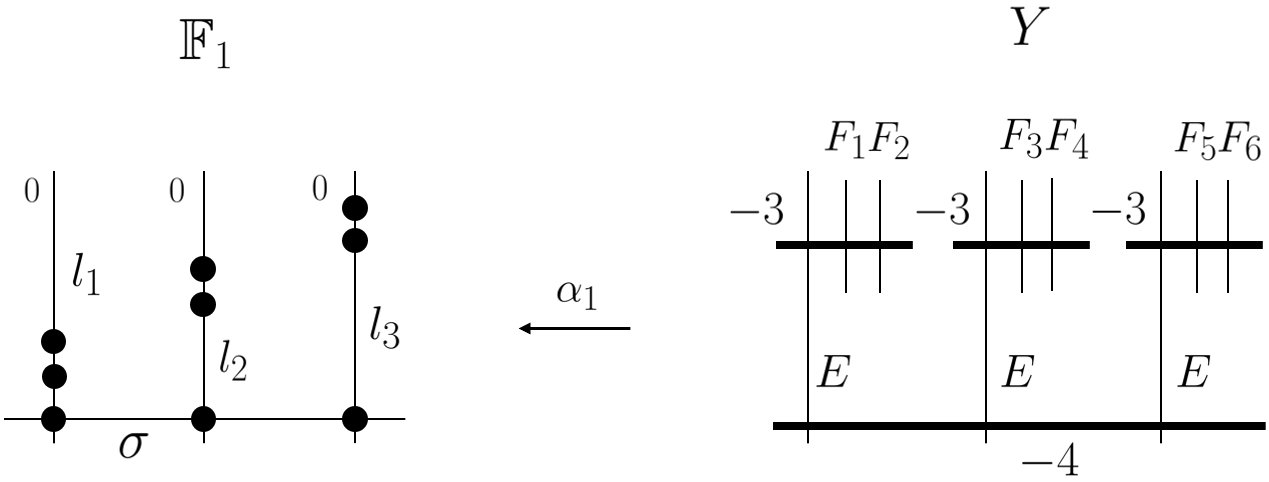}
\end{figure}

Let us construct an example of No.20.
In $\mathbb{F}_1$, take three distinct fibers $l_1, l_2, l_3$. 
Set $L := l_1 + l_2 + l_3$.
For $i \in \{ 1,2,3 \}$, take two general points $\{ P_i, Q_i \}$ on $l_i$ respectively.
Let $\alpha_1 : Y \to \mathbb{F}_1$ be the blow-up at $(L \cap \sigma) \cup \{P_1, P_2, P_3, Q_1, Q_2, Q_3 \}$. 
Denote the exceptional divisor over $(L \cap \sigma)$ by $E$ and over $P_1$, $Q_1$, $P_2$, $Q_2$, $P_3$, $Q_3$ by $F_1$, $F_2$, $F_3$, $F_4$, $F_5$, $F_6$ respectively.
Set $F := F_1 + \cdots + F_6$.
Denote the strict transform of $l_i$ by $l'_i$.
Then we have
\begin{eqnarray*}
 -K_Y &=& \alpha_1^* (-K_{\mathbb{F}_1}) - E - F  \\
  &=& \alpha_1^* (2 \sigma + L) - E - F  \\
  &=& 2(\sigma_Y + E) + (L_Y + E + F) -E - F \\
  &=& 2 \sigma_Y + L_Y + 2E_Y.
\end{eqnarray*}
Let $f : Y \to X$ be the contraction of $L_Y$ and $\sigma_Y$.
Then we also have
\[
K_Y = f^*K_X - \frac{1}{2}\sigma_Y - \frac{1}{3}L_Y.
\]        
Hence we obtain the following relation;
\[
 f^*( -K_X) =  \frac{3}{2}\sigma_Y + \frac{2}{3}L_Y + 2E_Y.
\]     

\begin{claim}

$-K_X$ is ample.

\end{claim}

\begin{proof}

We see that $(-K_X)^2 = 1 > 0$ and $-K_X$ is nef as in Claim \ref{K^2_nef}.
Let $C$ be an irreducible curve on $X$.
Assume that $-K_X \cdot C = 0$ by contradiction.
Let $p : \mathbb{F}_1 \to \pr^2$ be a contraction of $\sigma$.
Take a line $l$ on $\pr^2$.
Let $\alpha_* C_Y = dl$, where $\alpha := p \circ \alpha_1$.
Since $-K_X \cdot C = 0$, we see that $\alpha_{*} C_Y$ is an irreducible curve.
Then we have
\[
\alpha^* \alpha_* C_Y = C_Y + \sum_{i=1}^6 d_i F_i ,
\]
where $F_1, \ldots , F_6$ are irreducible components of $F$ and $d_i = C_Y \cdot F_i \ge 0$ for $i \in \{1, \ldots , 6 \}$.
We have
\[
l'_1 \cdot \alpha^* \alpha_* C_Y = l'_1 \cdot C_Y + l'_1 \cdot \sum_{i=1}^6 d_i F_i .
\]
Since we see that $l'_1 \cdot \alpha^* \alpha_* C_Y = l \cdot dl = d$, $l'_1 \cdot C_Y = 0$ and $l'_1 \cdot \sum_{i=1}^6 d_i F_i = d_1 + d_2$, we have
\[
d = d_1 + d_2.
\]
Thus we obtain $d = d_3 + d_4 = d_5 + d_6$ similarly. 

Then we see that $\alpha(F)$ on $\pr^2$ is a set of six points.
Let $\varphi : S_3 \to \pr^2$ be the blow-up at the six points.
Since we take $P_1$, $Q_1$, $P_2$, $Q_2$, $P_3$, $Q_3$ generally,  we see that $S_3$ is a cubic del Pezzo surface with an Eckerd point.
Then there exists a birational morphism $\psi : Y \to S_3$ such that $\alpha = \varphi \circ \psi : Y \to \pr$.
We may denote $\psi_* F_i$ by $F_i$ again.
Then we have
\[
\varphi^* \alpha_* C_Y = C_{S_3} + \sum_{i=1}^6 d_i F_i .
\]
By this relation, we have
\[
-K_{S_3} \cdot \varphi^* \alpha_* C_Y = -K_{S_3} \cdot C_{S_3}  - K_{S_3} \cdot \sum_{i=1}^6 d_i F_i .
\]
Then we see that $-K_{S_3} \cdot \varphi^* \alpha_* C_Y = -K_{\pr^2} \cdot dl = 3d$ and $- K_{S_3} \cdot \sum_{i=1}^6 d_i F_i = \sum_{i=1}^6 d_i$.
Hence we have
\[
-K_{S_3} \cdot C_{S_3} = 3d - \sum_{i=1}^6 d_i = 0.
\]
This contradicts the fact that $S_3$ is a del Pezzo surface.

\end{proof}

\begin{claim}

There is no floating $(-1)$-curves on $X$.

\end{claim}

\begin{proof}

Assume there exists a floating $(-1)$-curve $C$ on $X$.
Since $C$ does not pass through any singular points on $X$, we have $f^* C = C_Y$.
Thus we have $1 = -K_X \cdot C = f^*(-K_X) \cdot C_Y = \frac{3}{2}\sigma_Y \cdot C_Y + \frac{2}{3}L_Y \cdot C_Y + 2E_Y \cdot C_Y =  2E_Y \cdot C_Y$.
Hence we have $E_Y \cdot C_Y = \frac{1}{2}$.
Since $Y$ is a smooth surface, $E_Y \cdot C_Y$ must be an integer number.
This is a contradiction.  

\end{proof}

\noindent From this construction, it follows that $X$ is a del Pezzo surface having no floating $(-1)$-curves such that $(-K_X)^2 = 1$ and $(n_3, n_4) = (3,1)$.
Hence $X$ is of No.20. \\

\noindent{\bf No.22}
 
Let $X$ be a del Pezzo surface of No.22.
Let $Y \to X$ be the minimal resolution.
By observing the configuration of negative curves on $Y$, we can find the following blow-downs $\alpha_1$, $\alpha_2$:

\begin{figure}[htbp]
\centering\includegraphics[width=14cm, bb=0 0 1257 350]{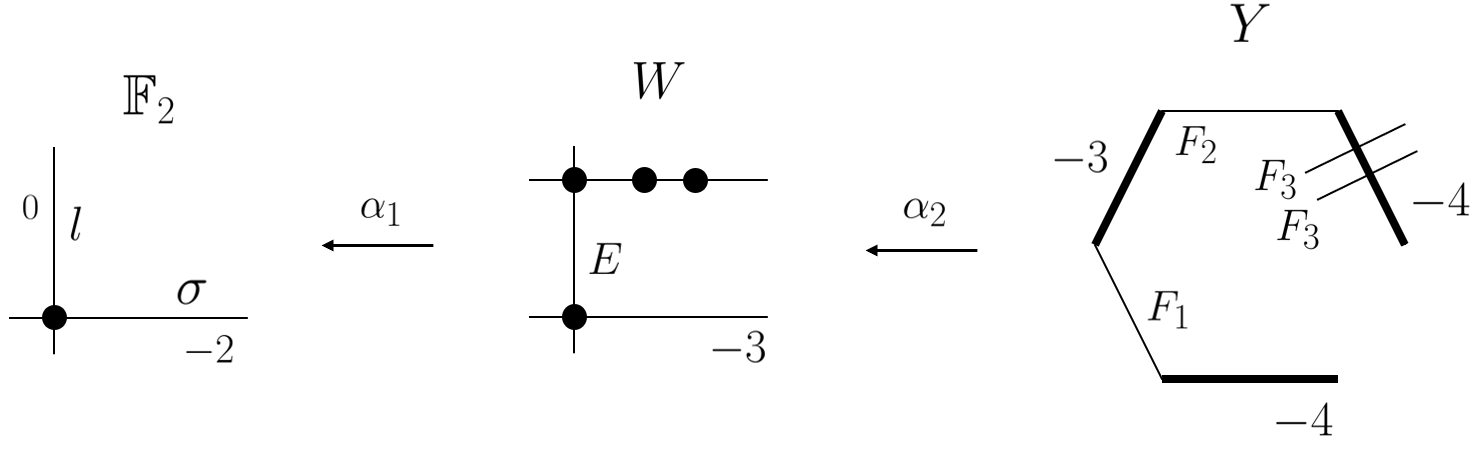}
\end{figure}

Let us construct an example of No.22.
Take a point $P$ on the minimal section $\sigma$.
Let $l$ be a fiber passing through $P$.
Let $\alpha_1 : W \to \mathbb{F}_2$ be the blow-up at $P$.
Denote the exceptional divisor by $E$.
Then we have
\begin{eqnarray*}
 -K_W &=& \alpha_1^* (-K_{\mathbb{F}_2}) - E  \\
  &=& \alpha_1^* (2 \sigma + 4l) - E  \\
  &=& 2(\sigma_W + E) + 4(l_W + E) -E \\
  &=& 2 \sigma_W + 4l_W + 5E.
\end{eqnarray*}

Take two general points $P_1, P_2$ on $l_W$.
Let $\alpha_2 : Y \to W$ be the blow-up at $((\sigma_W \cup l_W) \cap E) \cup \{ P_1, P_2 \}$.
Denote the exceptional divisor over $\sigma_W \cap E$ by $F_1$, over $l_W \cap E$ by $F_2$ and over $P_1, P_2$ by $F_3$.
Then we have
\begin{eqnarray*}
 -K_Y &=& \alpha_2^* (-K_W) - F_1 - F_2 - F_3 \\
  &=& \alpha_2^* (2 \sigma_W + 4l_W + 5E) - F_1 - F_2 - F_3  \\
  &=& 2(\sigma_Y + F_1) + 4(l_Y + F_2 + F_3) + 5(E_Y + F_1 + F_2)  \\
  & & - F_1 - F_2 - F_3 \\
  &=& 2\sigma_Y + 4l_Y+ 5E_Y + 6F_1 + 8F_2 + 3F_3.
\end{eqnarray*}
Let $f : Y \to X$ be the contraction of $\sigma_Y$, $l_Y$ and $E_Y$.
Then we also have
\[
K_Y = f^*K_X - \frac{1}{2}\sigma_Y - \frac{1}{2}l_Y -\frac{1}{3}E_Y.
\]        
Hence we obtain the following relation;
\[
 f^*( -K_X) =  \frac{3}{2}\sigma_Y + \frac{7}{2}l_Y + \frac{14}{3}E_Y + 6F_1 + 8F_2 + 3F_3.
\]

\begin{claim}

$-K_X$ is ample.

\end{claim}

\begin{proof}

We see that $(-K_X)^2 = \frac{16}{3} > 0$ and $-K_X$ is nef as in Claim \ref{K^2_nef}.
Let $C$ be an irreducible curve on $X$.
Assume that $-K_X \cdot C = 0$.
Denote $\alpha_* C_Y = a \sigma + bl$, where  $\alpha := \alpha_1 \circ \alpha_2$. 
Since $-K_X \cdot C = 0$, we see that $\alpha_* C_Y$ is an irreducible curve.
We also see that $\alpha^* \alpha_* C_Y = C_Y$.
Thus we have $0 = \sigma_Y \cdot C_Y = \sigma \cdot (a \sigma + bl) = -2a + b$ and $0 = l_Y \cdot C_Y = l \cdot (a \sigma + bl) = a$.
Thus we see that $\alpha_* C_Y = 0$.
This is a contradiction.

\end{proof}

\noindent From this construction, it follows that $X$ is a del Pezzo surface such that $(-K_X)^2 = \frac{16}{3}$ and $(n_3, n_4) = (1,2)$.
If $X$ has some floating $(-1)$-curves, we obtain a contradiction as in the proof of Lemma \ref{red_float}.
Thus we see that $X$ has no floating $(-1)$-curves.
Hence $X$ is of No.22. \\

\noindent{\bf No.28} 

Let $X$ be a del Pezzo surface of No.28.
Let $Y \to X$ be the minimal resolution.
By observing the configuration of negative curves on $Y$, we can find a blow-down $\alpha$:
\begin{figure}[htbp]
\centering\includegraphics[width=14cm, bb=0 0 1100 290]{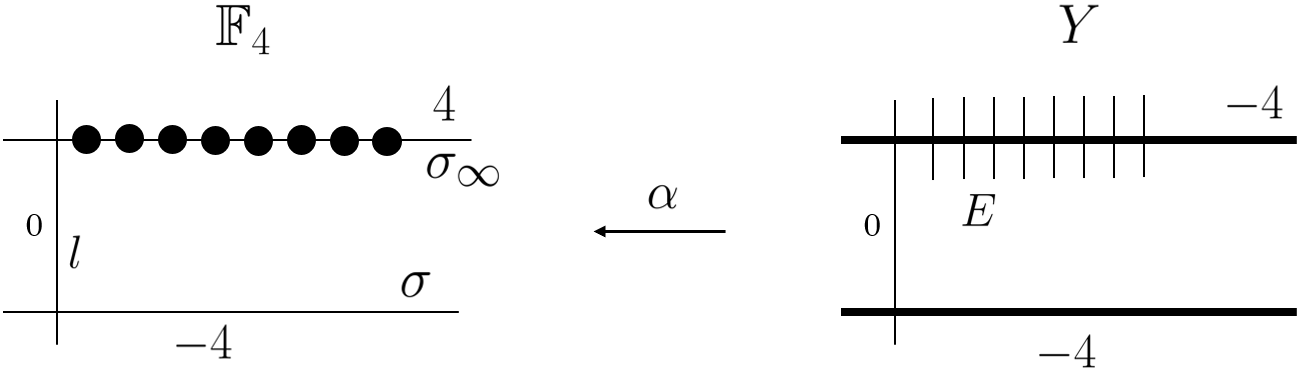}
\end{figure}

Let us construct an example of No.28.
Take a section at infinity $\sigma_{\infty}$.
Take distinct eight points $P_1, \ldots , P_8$ on $\sigma_{\infty}$.
Let $l$ be a fiber which doesn't pass through the points.
Let $\alpha : Y \to \mathbb{F}_4$ be the blow-up at the eight points.
Denote the exceptional divisor over $P_i$ by $E_i$ for each $i \in \{ 1, \ldots , 8 \}$.
Set $E := E_1 + \cdots + E_8$.
We have
\begin{eqnarray*}
 -K_Y &=& \alpha^* (-K_{\mathbb{F}_4}) - E  \\
  &=& \alpha^* (\sigma + \sigma_{\infty} + 2l) - E  \\
  &=& \sigma_Y + ((\sigma_{\infty})_Y + E) + 2l_Y - E \\
  &=& \sigma_Y + (\sigma_{\infty})_Y  + 2l_Y.
\end{eqnarray*}
Let $f : Y \to X$ be the contraction of $\sigma_Y$ and $(\sigma_{\infty})_Y$.
Then we also have
\[
K_Y = f^*K_X - \frac{1}{2}\sigma_Y - \frac{1}{2}(\sigma_{\infty})_Y.
\]  
We obtain the following relation;
\[
 f^*( -K_X) =  \frac{1}{2}\sigma_Y + \frac{1}{2}(\sigma_{\infty})_Y + 2l_Y.
\]  

\begin{claim}

$-K_X$ is ample.

\end{claim}

\begin{proof}

We see that $(-K_X)^2 = 2 > 0$ and $-K_X$ is nef as in Claim \ref{K^2_nef}.
Let $C$ be an irreducible curve on $X$.
Assume that $-K_X \cdot C = 0$ by contradiction.
If $C_Y \subset E$, then $f^* (-K_X) \cdot C_Y = \frac{1}{2} > 0$.
We may assume that $C_Y \not\subset E$.
Hence $\alpha_{*} C_Y$ is an irreducible curve on $\F_4$ and set $\alpha_{*} C_Y = a \sigma + bl$.
Then we have
\[
\alpha^* \alpha_{*} C_Y = C_Y + \sum_{i=1}^8 d_i E_i ,
\]
where $d_i = C_Y \cdot E_i \ge 0$ for $i \in \{1, \ldots , 8 \}$.
We have 
\[
\sigma_Y \cdot \alpha^* \alpha_{*} C_Y = \sigma_Y \cdot C_Y + \sigma_Y \cdot \sum_{i=1}^8 d_i E_i 
\]
and 
\[
l_Y \cdot \alpha^* \alpha_{*} C_Y = l_Y \cdot C_Y + l_Y \cdot \sum_{i=1}^8 d_i E_i .
\]
By calculating these, we see that $-4a + b = 0$ and $a=0$.
Hence we see that $\alpha_{*} C_Y =0$.
This contradicts the fact that $\alpha_{*} C_Y$ is an irreducible curve.

\end{proof}

\noindent From this construction, it follows that $X$ is a del Pezzo surface such that $(-K_X)^2 = 2$ and $(n_3, n_4) = (0,2)$.
If $X$ has some floating $(-1)$-curves, we obtain a contradiction as in the proof of Lemma \ref{red_float}.
Thus we see that $X$ has no floating $(-1)$-curves.
Hence $X$ is of No.28. \\

\section{Distinction of del Pezzo surfaces}\label{distinction_sec}

In the last section, we confirm the existence of each candidate in Table \ref{main_table}.
There are, however, several surfaces which have the same invariants, $n_3, n_4$, anti-canonical volume and Picard number. 
The pairs of surfaces which we must confirm that are different types are the following four pairs.

{\small \rm
\begin{table}[htb]
\caption{Del Pezzo surfaces which we must confirm that are different surfaces}\label{dist_table} 
\renewcommand{\arraystretch}{1.4}
  \begin{tabular}{c|c|c|c|c|c|c} 
    \multicolumn{1}{l|}{No.} &  \multicolumn{1}{c|}{$X_{min}$} & directed seq.
      & \multicolumn{1}{c|}{($n_3, n_4$)} & \multicolumn{1}{c|}{$(-K_{X})^2$} & $\rho (X)$ & $h^0(-K_X)$  \\ \hline \hline 
     6 &  $\pr(1,1,3)$ & $\smor_7 \circ \smor_3$& (4,2) & $\frac{4}{3}$  & 6  & 1 \\
     7 &  $\PP$ & $\smor_4 \circ \smor_4$ & (4,2) & $\frac{4}{3}$ & 6 & 1 \\  \rowcolor[gray]{0.9} \hline
     11 & $\pr(1,1,4)$ & $\smor_8 \circ \smor_1$ & (3,2) & 2 & 6  & 2 \\  \rowcolor[gray]{0.9}
     12 & $\pr(1,1,3)$ & $\smor_7 \circ \smor_1$ & (3,2) & 2  & 6 & 2 \\  
     15 & $\pr(1,1,4)$ & $\smor_8$ & (2,2) & $\frac{14}{3}$ & 4 & 5 \\ 
     16 & $\pr(1,1,3)$ & $\smor_7$ & (2,2) & $\frac{14}{3}$  & 4 & 5  \\ \rowcolor[gray]{0.9} \hline
     24 & $\pr(1,1,3)$ & $\smor_5$ & (2,1) & $\frac{14}{3}$  & 4 & 5   \\  \rowcolor[gray]{0.9}
     25 & $\PP$ & $\smor_4$ & (2,1) & $\frac{14}{3}$ & 4 & 5 \\    
  \end{tabular}
\end{table}
}

In this section, we confirm such surfaces are different by showing that they cannot have the same directed sequences.

\noindent{\bf No.24  and No.25}

Let $X_{24}$ be a del Pezzo surface of No.24 and $X_{25}$ one of No.25.
Let $\pi_i : Y_i \to X_i$ be the minimal resolution for $i = 24, 25$.
Since there is a sequence $X_{25} \overset{\smor_4}{\to} \PP$, we see that $Y_{25}$ has the following negative curves:
\begin{center}
 \ \xygraph{
	\square ([]!{+(.0,-.3)} {-3}) 
	- [r]	\bullet ([]!{+(.0,-.3)} {-1}) 
        - [r]	\square ([]!{+(.0,-.3)} {-3}) }
\end{center}
Moreover, we see that there are exactly two $(-3)$-curves on $Y_{24}$ by Lemma \ref{no_neg}.
Denote them by $C_1$ and $C_2$.
Then for the distinction of $X_{24}$ and $X_{25}$, it is enough to show the following claim.

\begin{claim}\label{24_25}

Let $C$ be a $(-1)$-curve on $Y_{24}$.
Then we have 
\[
C \cdot (C_1 + C_2) \le 1.
\]

\end{claim}

\begin{proof}

Observing the configuration of negative curves on $Y_{24}$, we see that there are the following blow-downs $\alpha_1, \alpha_2$:

\begin{figure}[htbp]
\centering  \includegraphics[width=14cm, bb=0 0 1294 422]{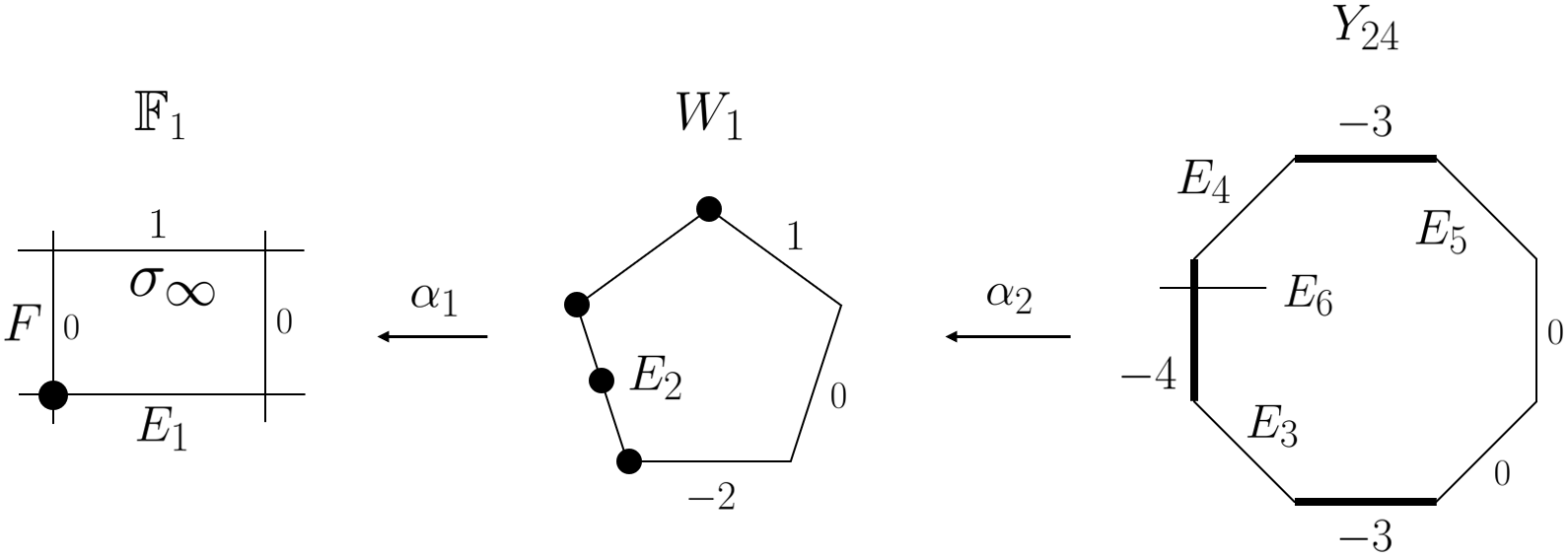}
\end{figure}

$\alpha_1 : W_1 \to \F_1$ is the blow-up at a point on the minimal section (Here we denote the minimal section by not $\sigma$ but $E_1$).
Denote the fiber which passes through the blow-up point by $F$ and the exceptional curve of $\alpha_1$ by $E_2$.
Take a section at infinity $\sigma_{\infty}$.
Then we see that $(E_1)_{W_1} \cdot E_2 = 1$, $E_2 \cdot F_{W_1} = 1$ and $F_{W_1} \cdot \sigma_{\infty} = 1$.
Hence we denote $(E_1)_{W_1} \cap E_2$, $E_2 \cap F_{W_1}$ and $F_{W_1} \cap \sigma_{\infty}$ by $P_3$, $P_4$ and $P_5$ respectively.
Take a general point $P_6$ on $E_2$.
$\alpha_2 : Y_{24} \to W_1$ is the blow-up at the four points $P_3, \ldots , P_6$.
Denote the exceptional curve over $P_i$ by $E_i$ for $i \in \{3, \ldots ,6 \}$.
Set $l := \alpha_2^* \alpha_1^* \sigma_{\infty}$, $e_1  := \alpha_2^* \alpha_1^* E_1$, $e_2 := \alpha_2^*E_2$ and $e_i := E_i$ for each $i \in \{ 3, \ldots , 6 \}$.
Then we see that $\Pic Y_{24}$ is spanned by $l, e_1, \ldots , e_6$ and they are disjoint. 
Since one of the two $(-3)$-curves is the strict transform of $E_1$, we may assume that $C_1 \sim e_1 - e_2 - e_3$. 
The other $(-3)$-curve is the strict transform of $F$.
Thus we may assume that $C_2 \sim l - e_1 - e_2 - e_4 - e_5$.
Thus we see that $C_1 + C_2 \sim l -2e_2 - e_3 - e_4 - e_5$. 
Let $C \sim xl + \sum_{i=1}^6 a_i e_i$ be a $(-1)$-curve.
Since $-K_{Y_{24}} \cdot C = 1$ and $C^2 = -1$, we have
\[
1 = 3x + \sum_{i=1}^6 a_i \ \ \ {\rm and} \ \ \ -1 = x^2 - \sum_{i=1}^6 a_i^2
\]
The same caluculation is discussed in \cite{Ha77}.
We know that the solutions are 27 cases.
For each case, we see that $C \cdot (C_1 + C_2) = x + 2a_2 + a_3 + a_4 + a_5 \le 1$.

\end{proof}

\noindent{\bf No.15 and 16} 

Let $X_{15}$ be a del Pezzo surface of No.15 and $X_{16}$ one of No.16.
Let $\pi_i : Y_i \to X_i$ be the minimal resolution for each $i = 15, 16$.
Since there is a sequence $X_{16} \overset{\smor_7}{\to} \pr(1,1,3)$, we see that $Y_{16}$ has the following negative curves:
\begin{center}
 \ \xygraph{
	\bigcirc ([]!{+(.0,-.3)} {-4}) 
	- [r]	\bullet ([]!{+(.0,-.3)} {-1}) 
        - [r]	\square ([]!{+(.0,-.3)} {-3}) 
        - [r]	\bullet ([]!{+(.0,-.3)} {-1}) 
        - [r]	\bigcirc ([]!{+(.0,-.3)} {-4}) 
        - [r]	\bullet ([]!{+(.0,-.3)} {C_{Y_{16}}}) }
\end{center}
$C_{Y_{16}}$ is a $(-1)$-curve.
We see that $C_{Y_{16}}$ does not cross any $(-3)$-curves.
Hence it is enough to show the following claim.

\begin{claim}\label{15_16}

Let $C$ be a $(-1)$-curve on $Y_{15}$.
There exists a $(-3)$-curve $D$ such that $C \cdot D \ge 1$.

\end{claim}

\begin{proof}

Observing the configuration of negative curves on $Y_{15}$, we see that there are the following blow-downs $\alpha_3, \alpha_4$:

\begin{figure}[htbp]
\centering\includegraphics[width=13cm, bb=0 0 1310 472]{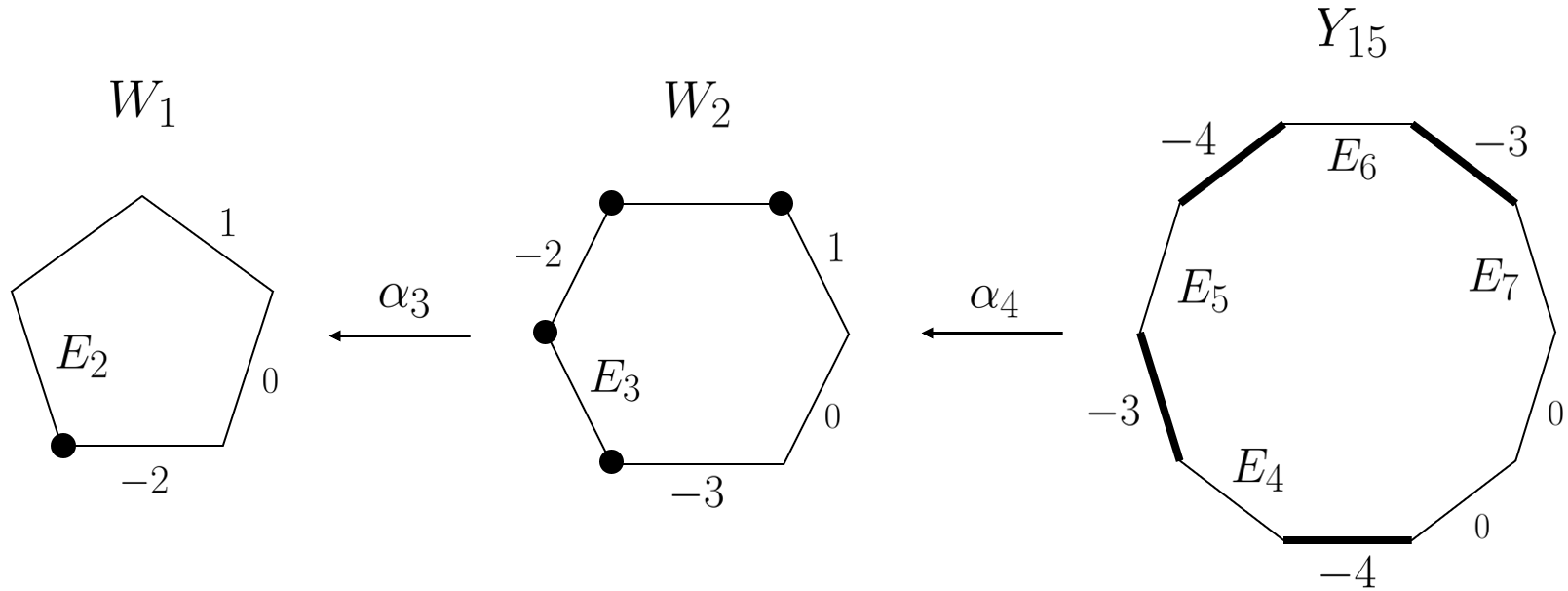}
\end{figure}

$W_1$ is the same surface as the one in Claim \ref{24_25}.
$\alpha_3 : {W_2} \to W_1$ is the blow-up at a point $P_3 = (E_1)_{W_1} \cap E_2$.
Denote the exceptional curve of $\alpha_3$ by $E_3$.
Then we see that $(E_1)_{W_2} \cdot E_3 = 1$, $E_3 \cdot (E_2)_{W_2} = 1$, $(E_2)_{W_2} \cdot F_{W_2} = 1$ and $F_{W_2} \cdot L_{W_2} = 1$.
Denote  $(E_1)_{W_2} \cap E_3$, $E_3 \cap (E_2)_{W_2}$, $(E_2)_{W_2} \cap F_{W_2}$ and $F_{W_2} \cap L_{W_2}$ by $P_4, \ldots , P_7$ respectively.
$\alpha_4 : Y_{15} \to W_2$ is the blow-up at $P_4, \ldots , P_7$.
Denote the exceptional curve over $P_i$ by $E_i$ for $i \in \{ 4, \ldots , 7 \}$.
Set $l := \alpha_4^* \alpha_3^* \alpha_1^* \sigma_{\infty}$, $e_1 := \alpha_4^* \alpha_3^* \alpha_1^* E_1$, $e_2 := \alpha_4^* \alpha_3^* E_2$, $e_3 := \alpha_4^* E_3$ and $e_i := E_i$ for each $i \in \{ 4, \ldots , 7 \}$.
Then we see that $\Pic Y_{15}$ is spanned by $l, e_1, \ldots , e_7$ and they are disjoint.
We see that the $(-3)$-curves on $Y_{15}$ are the strict transforms of $E_3$ and $F$.
Here we have
\[
(E_3)_{Y_{15}} \sim e_3 - e_4 -e_5 
\]
and 
\[
F_{Y_{15}} \sim l - e_1 - e_2 - e_6 - e_7 . 
\]
Let $C \sim xl + \sum_{i=1}^7 a_i e_i$ be a $(-1)$-curve.
What we should prove is the inequality $C \cdot ( (E_3)_{Y_{15}} + F_{Y_{15}}) \ge 1$.
Since $C$, $(E_3)_{Y_{15}}$ and $F_{Y_{15}}$ are distinct irreducible curves, we see that $ C \cdot (E_3)_{Y_{15}} \ge 0$ and $C \cdot F_{Y_{15}} \ge 0$.
Since $-K_{Y_{15}} \cdot C = 1$, we have
\[
1 = 3x + \sum_{i=1}^7 a_i.
\]
By this relations, we have
\begin{eqnarray*}
1 & = & 2x + 2a_3 + (-a_3 + a_4 + a_5) + (x + a_1 + a_2 +a_6 + a_7) \\
  & = & 2x + 2a_3 + C \cdot (E_3)_{Y_{15}} + C \cdot F_{Y_{15}} .
\end{eqnarray*}
Thus we see that $C \cdot (E_3)_{Y_{15}} + C \cdot F_{Y_{15}} \ge 1$.

\end{proof}

\noindent{\bf No.11 and 12} 

Let $X_{11}$ be a del Pezzo surface of No.11 and $X_{12}$ one of No.12.
Let $\pi_i : Y_i \to X_i$ be the minimal resolution for each $i = 11, 12$.
Since there is a sequence $X_{12} \overset{\smor_7}{\to} U_1$, we see that $Y_{12}$ has the following negative curves:
\begin{center}
 \ \xygraph{
	\bigcirc ([]!{+(.0,-.3)} {-4}) 
	- [r]	\bullet ([]!{+(.0,-.3)} {-1}) 
        - [r]	\square ([]!{+(.0,-.3)} {-3}) 
        - [r]	\bullet ([]!{+(.0,-.3)} {-1}) 
        - [r]	\bigcirc ([]!{+(.0,-.3)} {-4}) 
        - [r]	\bullet ([]!{+(.0,-.3)} {C_{Y_{12}}}) }
\end{center}
$C_{Y_{12}}$ is a $(-1)$-curve.
We see that $C_{Y_{12}}$ does not cross any $(-3)$-curves.
Hence it is enough to show the following claim.

\begin{claim}

Let $C$ be a $(-1)$-curve on $Y_{11}$.
There exists a $(-3)$-curve $D$ such that $C \cdot D \ge 1$.

\end{claim}

\begin{proof}

Observing the configuration of negative curves on $Y_{11}$, we see that there are the following blow-downs $\alpha_5, \alpha_6$:

\begin{figure}[htbp]
\centering\includegraphics[width=14cm, bb=0 0 1521 530]{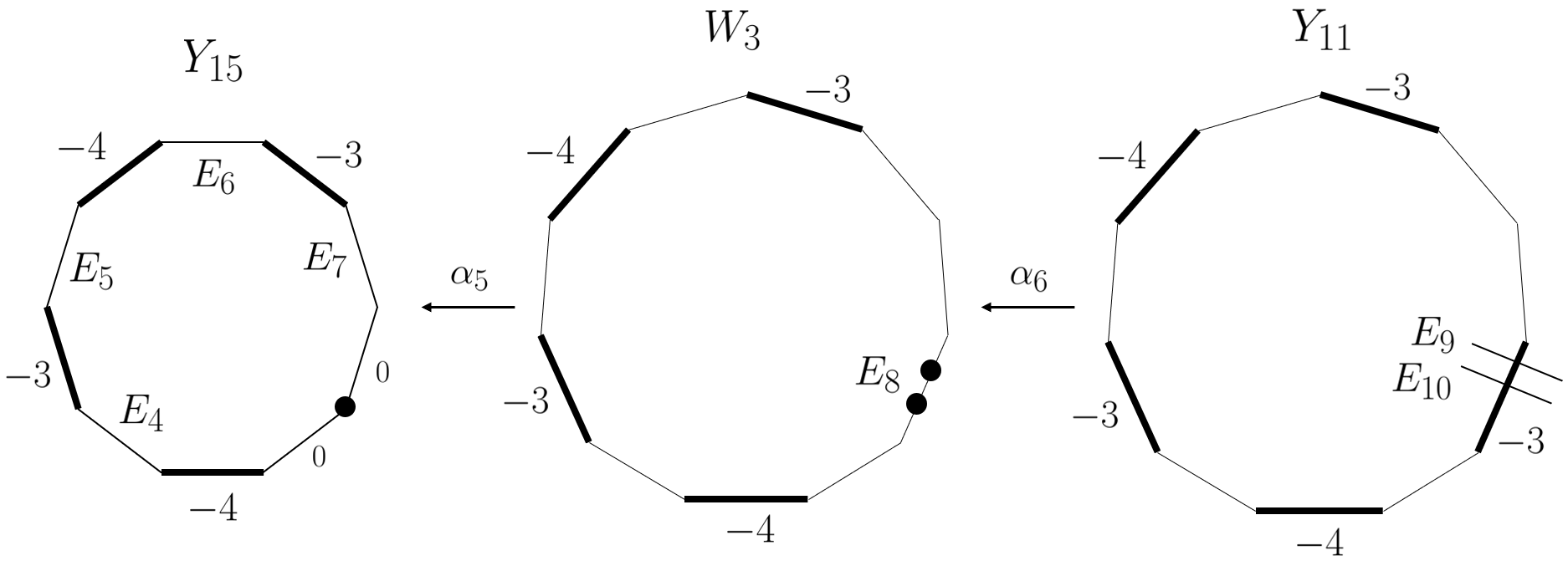}
\end{figure}

Here $Y_{15}$ is the same surface as the one in Claim \ref{15_16}.
Take a point $P_8$ which any negative curves does not pass through.
$\alpha_5 : W_3 \to Y_{15}$ is the blow-up at $P_8$.
Denote the exceptional curve by $E_8$.
Take distinct two points $P_9$ and $P_{10}$ on $E_8$.
$\alpha_6 : Y_{11} \to W_3$ is the blow-up at $P_9$ and $P_{10}$. 
Denote the exceptional curves over $P_9, P_{10}$ by $E_9, E_{10}$ respectively.
Set $l := \alpha_6^* \alpha_5^* l$, $e_i := \alpha_6^* \alpha_5^* e_i$ for each $i \in \{ 1, \ldots , 7 \}$ again.
Set $e_8 := \alpha_6^* E_8$, $e_9 := E_9$ and $e_{10} := E_{10}$.
Then we see that $\Pic Y_{11}$ is spanned by $l, e_1, \ldots , e_{10}$ and they are disjoint.
We see that the $(-3)$-curves on $Y_{11}$ are the strict transforms of $E_3$, $F$ and $E_8$.
We have
\[
C_1 := (E_3)_{Y_{11}} \sim e_3 - e_4 - e_5 ,
\]
\[
C_2 := F_{Y_{11}} \sim l - e_1 - e_2 - e_6 - e_7
\]
and
\[
C_3 := (E_8)_{Y_{11}} \sim e_8 - e_9 - e_{10} .
\]
Let $C \sim xl + \sum_{i=1}^{10}$ be a $(-1)$-curve.
What we should prove is the inequality $C \cdot (C_1 + C_2 + C_3) \ge 1$.
Since $C_1, C_2$ and $C_3$ are distince irreducible curves, we have $C \cdot C_i \ge 0$ for $1 \le i \le 3$.
Since $-K_X \cdot C = 1$, we have
\[
1 = 3x + \sum_{i=1}^{10} a_i.
\]
By this relations, we have
\begin{eqnarray*}
1 & = & 2x + 2a_3 + 2a_8 + (-a_3 + a_4 + a_5)  \\
  &  & + (x + a_1 + a_2 + a_6 + a_7) + (-a_8 + a_9 + a_{10} ) \\
  & = & 2x + 2a_3 + 2a_8 + C \cdot C_1 + C \cdot C_2 + C \cdot C_3 .
\end{eqnarray*}
Thus we see that $C \cdot C_1 + C \cdot C_2 + C \cdot C_3 \ge 1$.

\end{proof}

\noindent{\bf No.6 and No.7}

Let $X_6$ be a del Pezzo surface of No.6 and $X_7$ one of No.7.
Let $\pi_i : Y_i \to X_i$ be the minimal resolution for each $i = 6, 7$.
We see that $Y_7$ has exactly two $(-4)$-curves $C, D$.

\begin{claim}

There are distinct four $(-1)$-curves $C_1, C_2, D_1$ and $D_2$ on $Y_7$ such that $C \cdot C_i =1$, $D \cdot C_i = 0$, $C \cdot D_i = 0$ and $D \cdot D_i =1$ for $i = 1,2$.

\end{claim}

\begin{proof}

We have a sequence $X_7 \overset{\varphi_1}{\to} U_1 \overset{\varphi_2}{\to} \PP$, where both $\varphi_1$ and $\varphi_2$ are of type $\smor_4$.
Since $\varphi_1$ and $\varphi_2$ are disjoint, we may denote the center of $\varphi_i$ on $\PP$ by $P_i$ for $i = 1,2$.
Since $P_1$ and $P_2$ are not on the same fiber, there are two fibers for $P_1$ and $P_2$ respectively.
The strict transforms of the fibers on $Y_7$ by $\varphi_2 \circ \varphi_1 \circ \pi_7$ are what we need.

\end{proof}

We will prove that there is no such a pair of $(-1)$-curves on $Y_6$.
By observing the configuration of negative curves on $Y_6$, we see that there is a sequence of blow-downs:
\begin{figure}[htbp]
\centering \includegraphics[width=14cm, bb=0 0 1400 500]{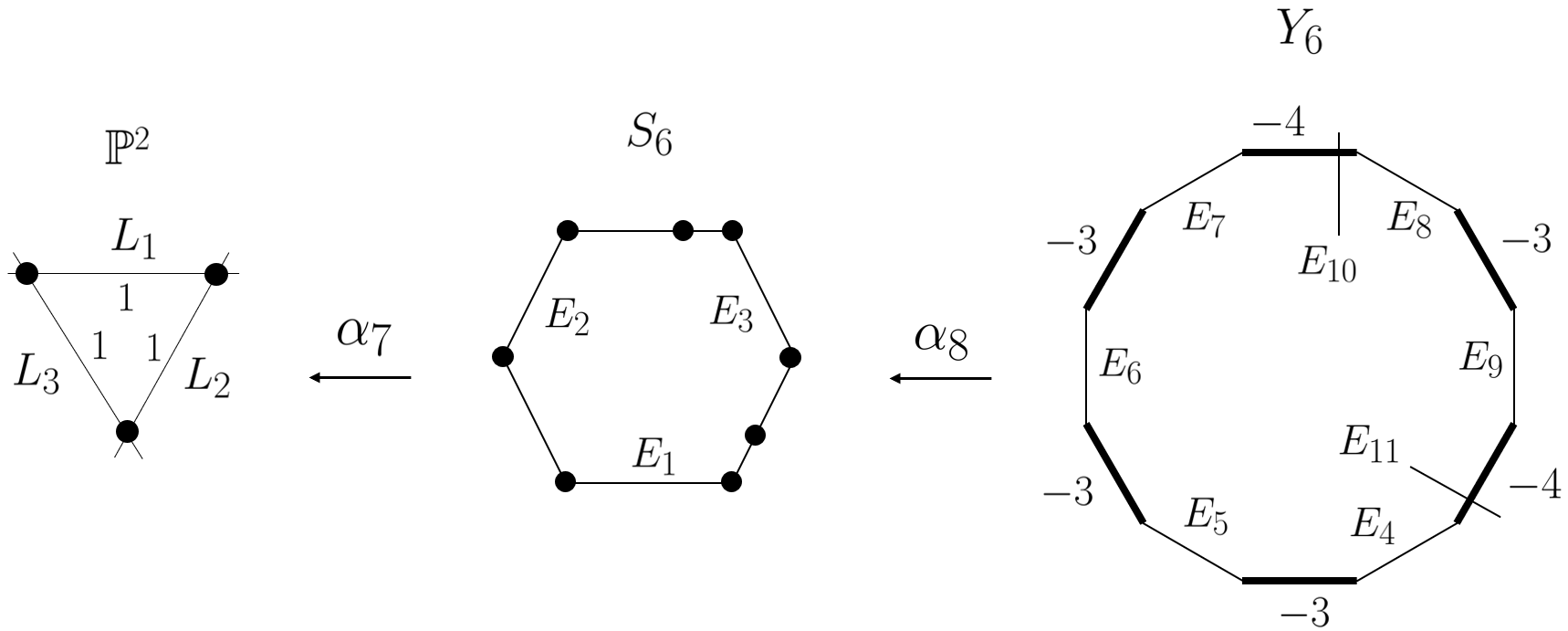}
\end{figure}

$\alpha_7 : S_6 \to \pr^2$ is the blow-up at distinct three points $P_1, P_2, P_3$ on $\pr^2$ which are not on a line. 
Denote by $L_i$ a line which does passes through $P_j$ and $P_k$ where $(i, j, k) = (1,2,3), (2,3,1), (3,1,2)$.
Denote the exceptional curve over $P_i$ by $E_i$ for each $i \in \{ 1,2,3 \} $.
We see that $E_i \cdot (L_j)_{S_6} = 1$ for $i,j \in \{ 1,2,3 \}$ where $i \neq j$.
Hence we set $P_4 := E_1 \cap (L_2)_{S_6}$, $P_5 := E_1 \cap (L_3)_{S_6}$, $P_6 := E_2 \cap (L_3)_{S_6}$, $P_7 := E_2 \cap (L_1)_{S_6}$, $P_8 := E_3 \cap (L_1)_{S_6}$ and $P_9 := E_3 \cap (L_2)_{S_6}$.
Take a general point $P_{10}$ on $(L_1)_{S_6}$ and $P_{11}$ on $(L_2)_{S_6}$.
$\alpha_8 : Y_6 \to S_6$ is the blow-up at $P_4, \ldots , P_{11}$.
Denote the exceptional curve over $P_i$ by $E_i$ for each $i \in \{ 4, \ldots , 11 \}$.
Denote $l := \alpha_8^* \alpha_7^* L_1$, $e_i := \alpha_8^* E_i$ for $i \in \{ 1,2,3 \}$ and $e_j := E_j$ for $j \in \{ 4, \ldots , 11 \}$.
Let $C \sim xl + \sum_{i=1}^{11} a_i e_i$ be a $(-1)$-curve on $Y_6$.

Since $(L_1)_{Y_6}$ is a $(-4)$-curve on $Y_6$, we see that it is enough to show the following claim.

\begin{claim}

If $C$ does not cross any $(-3)$-curves, $(L_1)_{Y_{6}} \cdot C = 1$ and $(L_2)_{Y_{6}} \cdot C = 0$, then $C=E_{10}$. 

\end{claim}

\begin{proof}

We have
\[
(L_1)_{Y_6} \sim l - e_2 - e_3 - e_7 - e_8 - e_{10}
\]
and 
\[
(L_2)_{Y_6} \sim l - e_1 - e_3 - e_4 - e_9 - e_{11}.
\]
We see that $(-3)$-curves on $Y_6$ are the strict transforms of $L_3, E_1, E_2$ and $E_3$.
We have
\[
(L_3)_{Y_6} \sim l - e_1 - e_2 - e_5 - e_6 \ , 
\]
\[
(E_1)_{Y_6} \sim e_1 - e_4 - e_5 \ ,
\]
\[
(E_2)_{Y_6} \sim e_2 - e_6 - e_7
\]
and 
\[
(E_3)_{Y_6} \sim e_3 - e_8 - e_9 \ .
\]
By assumption, we have 
\begin{equation}\label{crazy}
\begin{cases}
\ 1 = x + a_2 + a_3 + a_7 + a_8 + a_{10}  \\
\ 0 = x + a_1 + a_3 + a_4 + a_9 + a_{11}  \\ 
\ 0 = x + a_1 + a_2 + a_5 + a_6 \\
\ 0 = -a_1 + a_4 + a_5 \\
\ 0 = -a_2 + a_6 + a_7 \\
\ 0 = -a_3 + a_8 + a_9 \ .
\end{cases}
\end{equation}
Since $-K_{Y_6} \cdot C = 1$, we also have
\[
1 = 3x + \sum_{i=1}^{11} a_i \ . 
\]
By these relations, we have
\[
a_4 + a_5 + a_6 + a_7 + a_8 +a_9 = 0 .
\]
Here we see that $C$ is not one of $E_4, \ldots , E_9$.
Thus for $i \in \{ 4, \ldots , 9 \}$, $C \cdot E_i \le 0$, that is, $a_i \le 0$.
Therefore, we see that 
\[
a_4 = \cdots = a_9 = 0.
\]
Then we see that $C \sim e_{10}$ by the relations (\ref{crazy}).
Hence $C = E_{10}$.

\end{proof}

Thus we distinguish the four pairs.
Therefore, we see that all surfaces in Table \ref{main_table} are distinct.


\end{document}